\documentclass[10pt,a4paper]{article}
\usepackage[margin=1in]{geometry}
\usepackage{amsmath,amssymb,amsthm,graphicx,amsxtra,setspace}

\usepackage{mathtools}

\usepackage{mathrsfs}
\usepackage{bbm}
\usepackage{bm}
\usepackage{braket}
\usepackage{stmaryrd}
\numberwithin{equation}{section}
\usepackage{graphicx}
\usepackage{enumitem}
\usepackage{xcolor}
\usepackage[
    colorlinks=true,
    linkcolor=blue,
    citecolor=blue,
    urlcolor=blue
]{hyperref}
\usepackage[backend=biber,style=numeric]{biblatex}
\DeclareFieldFormat
  [article,inbook,incollection,inproceedings,patent,thesis,unpublished]
  {title}{#1}
\renewbibmacro{in:}{}
\DeclareFieldFormat{pages}{#1}
\usepackage{titlesec}

\titleformat{\section}
  {\large\bfseries}
  {\thesection}{1em}{}

\titleformat{\subsection}
  {\normalsize\bfseries}
  {\thesubsection}{1em}{}

\theoremstyle{definition}
\newtheorem{definition}{Definition}[section]

\theoremstyle{remark}

\newtheorem{remark}[definition]{Remark}

\theoremstyle{plain}
\newtheorem{theorem}[definition]{Theorem}
\newtheorem{lemma}[definition]{Lemma}

\newtheorem{proposition}[definition]{Proposition}

\title{\Large\bfseries Optimal Control of the 2D Boussinesq Equations with Finite Codimension Constraints}
\author{
Utpal Manna\thanks{School of Mathematics, Indian Institute of Science Education and Research Thiruvananthapuram, Vithura, Thiruvananthapuram 695551, India.\\
\textit{E-mail:} manna.utpal@iisertvm.ac.in.}
\and
Manil T. Mohan\thanks{Department of Mathematics, Indian Institute of Technology Roorkee, Haridwar Highway, Roorkee, Uttarakhand 247667, India.\\
\textit{E-mail:} maniltmohan@ma.iitr.ac.in, maniltmohan@gmail.com.}
\and
Sangram Satpathi\thanks{School of Mathematics, Indian Institute of Science Education and Research Thiruvananthapuram, Vithura, Thiruvananthapuram 695551, India.\\
\textit{E-mail:} sangramsatpathi20@iisertvm.ac.in.\\
\textit{Keywords:} Optimal control problem; Boussinesq equations; constrained dynamics; geometric PDEs.\\
\textit{MSC (2020):} 49J20, 49K20, 35R01, 35Q35, 35Q93
}
}
\date{}

\begin{document}
\maketitle

\begin{abstract}

We study an optimal control problem (OCP) governed by the two-dimensional (2D) Boussinesq equations, where the right-hand sides of the velocity and temperature equations are projected onto the corresponding tangent spaces of the manifolds $\mathcal{M}_1:=\{\bm y\in\mathbb{H}:\|\bm y\|_{\mathbb{H}}=1\}$
and $\mathcal{M}_2:=\{\theta\in H:\|\theta\|_H=1\}$,
respectively. The resulting evolution is subject to
finite-codimensional constraints, under which the state $(\bm y,\theta)$ remains on $\mathcal{M}:=\mathcal{M}_1\times\mathcal{M}_2$
throughout the time interval. A Faedo--Galerkin construction provides a unique, global, strong solution to the constrained Boussinesq system for external controls
$\mathfrak{H}\in L^2(0,T;\mathbb{H})$, along with uniform energy bounds. In addition, we establish the existence of an optimal triplet $(\bm y^*,\theta^*,\mathfrak{H}^*)$
that minimizes the quadratic cost subject to the constrained Boussinesq system. Necessary conditions of the first order are obtained through a coupled state--adjoint--control system. The optimal state $(\bm y^*,\theta^*)$ remains on $\mathcal{M}$, while the optimal control lies in the tangent space:
\[
\mathfrak{H}^*(t)\in T_{\bm y^*(t)}\mathcal{M}_1,
\qquad \text{for all }t\in[0,T].
\]
We also show that the map
\[
\mathfrak{H}\mapsto(\bm y_{\mathfrak{H}},\theta_{\mathfrak{H}})
\quad\text{in}\quad
\bigl(C([0,T];\mathbb{H})\cap L^2(0,T;\mathbb{V})\bigr)
\times
\bigl(C([0,T];H)\cap L^2(0,T;V)\bigr)
\]
is Fr\'echet differentiable. The linearized state is the unique weak solution to the constrained linearized Boussinesq system and satisfies the necessary tangency condition. Hence, the well-posedness and first-order optimality conditions are established for the Boussinesq system under nonlinear manifold constraints.
\end{abstract}
\tableofcontents

\section{Introduction}
\subsection{Constrained PDEs and literature survey}
Geometric partial differential equations (PDEs) describe the evolution
of functions, curves, and surfaces according to their geometry. Their study uses methods from both analysis and differential geometry
to understand flows and fields on manifolds. Constrained PDEs have attracted a considerable amount of attention over the past two decades and have been studied in various contexts, including the classical heat equation~\cite{caf,rybka}, nonlinear heat equations~\cite{anto,bj,jh,lcheng,rybka,bsh}, and the Navier--Stokes equations~\cite{gb}.

Caglioti et al.~\cite{cag} studied the vorticity equation
associated with the 2D Navier--Stokes equations subject to conservation constraints on energy and moment of inertia; their results include well-posedness and a description of the asymptotic behavior of solutions. The governing equation is given by
\[
\frac{\partial \zeta}{\partial t}
+ \bm{y}\cdot\nabla\zeta
= \nu\Delta\zeta
-\nu\,\mathrm{div}\left[
\zeta\nabla\left(
\lambda\phi+\mu\frac{|x|^2}{2}
\right)\right],
\]
or equivalently,
\begin{equation}
\frac{\partial \zeta}{\partial t}
+ \bm{y}\cdot\nabla\zeta
= \nu\,\mathrm{div}\left[
\zeta\nabla\left(
\log\zeta-\lambda\phi
-\mu\frac{|x|^2}{2}
\right)\right].
\label{cagg}
\end{equation}
Here, $\zeta=\mathrm{curl}(\bm{y})$, and
$\lambda=\lambda(\zeta)$ and $\mu=\mu(\zeta)$
are the Lagrange multipliers corresponding to
the conserved quantities
\[
\mathcal{E}(\zeta)
=\int_{\mathbb{R}^2}\phi\,\zeta\,dx,
\qquad
\mathcal{I}(\zeta)
=\int_{\mathbb{R}^2}|x|^2\zeta\,dx,
\qquad
\phi=-\Delta^{-1}\zeta.
\]
Rybka~\cite{rybka} and Caffarelli and Lin~\cite{caf} studied the heat equation under constraints. In particular, Rybka~\cite{rybka} considered the heat flow on the manifold
\[
\mathcal{M} = \left\{ {y} \in L^2(\mathcal{O}) \cap C(\mathcal{O}) : \int_{\mathcal{O}} {y}^k(x)\,dx = C_k,\; k = 1, \dots, N \right\},
\]
where
\[
I_k({y}) = \int_{\mathcal{O}} {y}^k(x)\,dx, \quad k = 1, \dots, N,
\]
are conserved quantities. Here, $\mathcal{O}\subset\mathbb{R}^2$ denotes a bounded and connected domain, and its boundary $\partial\mathcal{O}$ is considered to be smooth.
According to \cite[Theorem 2.5]{rybka}, the projected heat
equation admits a unique global solution.
\begin{equation}
\begin{cases}
\displaystyle \frac{d{y}}{dt} = \Delta {y} - \sum_{k=1}^{N} \lambda_k {y}^{k-1} & \text{in } \mathcal{O}\times (0,T), \\
\displaystyle \frac{\partial {y}}{\partial n} = 0 & \text{on } \partial \mathcal{O}\times (0,T), \\
{y}(0,x) = {y}_0, & \text{in }  \mathcal{O},
\end{cases}
\label{linh}
\end{equation}
where the coefficients $\lambda_k = \lambda_k({y})$ are chosen so that $\frac{d{y}}{d t}$ is orthogonal to $\mathrm{span}\{{y}^{k-1}\}$. Moreover, Rybka~\cite{rybka} showed that solutions to \eqref{linh} approach a steady state as $t\to\infty$. Caffarelli and Lin~\cite{caf} also established that the heat equation admits a unique global solution that preserves energy (see~\cite[Theorem 1.1]{caf}). They also extended these results to a wider class of singularly perturbed nonlocal parabolic equations~\cite[Theorem 3.1]{caf}, and showed strong convergence of solutions to weak solutions of the limiting constrained problem.

Brze\'zniak et al.~\cite{sgb} extended these ideas to the
stochastic setting by considering the Navier--Stokes equations (NSE) studied in~\cite{cag}, with the energy constraint as in~\cite{caf,rybka} and gradient-type Stratonovich noise. They established the existence of a global strong solution on $\mathbb{T}^2$.
For further related results, see also~\cite{Dhariwal2017}. The paper~\cite{stb} addresses a stochastic nonlinear heat equation with a finite-codimensional constraint. With \(H_0^1\) initial data, a fixed-point argument gives the existence of a unique global mild solution driven by Stratonovich noise in a bounded smooth 2D domain. More recently, \cite{bawalia2026existenceuniquenessresultsstochastic} studied a stochastic nonlinear heat equation with a codimension-one constraint and proved its well-posedness. 
In~\cite{cann}, the authors studied the two-dimensional
Navier--Stokes system on $\mathbb{T}^2$, incorporating
a bilinear control term and a nonlinear, nonlocal
perturbation to ensure that the solutions remain
within prescribed hyperplanes. They also proved
the well-posedness of the resulting system.

In fluid mechanics, the problem of controlling fluid flow and turbulence in a specified domain under prescribed initial conditions, boundary controls, or temperature variations has been studied for a long time (see, e.g.,~\cite{abergel,Fursikov2000,Lions1971,Sritharan1998,Gunzburger2003,Troeltzsch2010}).

Recent progress in optimal control for nonlinear infinite-dimensional systems and PDEs has strongly influenced the study of Navier--Stokes control problems. Several variations of these problems have been studied in the existing literature;  see, for example,~\cite{bew,casass,Delos,desai,fatt,FS1994,GM1999,HK2001,Mohan2021,Wang2002,WW2003} and the references therein.
For the constrained Boussinesq equations, the present work serves as an initial step toward studying their optimal control problem. 
Although the optimal control of constrained partial differential equations has been less explored, a related study by the second author ~\cite{MohanNS} addresses the optimal control of constrained NSE. On the other hand, there are several works on the control of the (unconstrained) Boussinesq equations; see, for example,~\cite{BFR2007,LI2000,LI2000b,Belmiloudi2000,LS2000} and the references therein.
\subsection{Deterministic Boussinesq equations and physical motivation.}

The Boussinesq equations are an important model for large-scale atmospheric and oceanic flows, the built environment, the dispersion of dense gases, and the internal dynamics of stars (see, e.g., \cite{ped,rch}). Besides their physical importance, these equations have also attracted much attention in the mathematical community because of their nonlinear and coupled structure.
The physical derivation of the Boussinesq equations goes back to the early twentieth century, particularly to the work of Rayleigh \cite{Rayleigh01121916}. In his study of buoyancy-driven convection, Rayleigh used the equations introduced earlier by Boussinesq \cite{Boussinesq1897} to explain the experimental observations of B\'enard.
From a mathematical point of view, the two-dimensional Boussinesq equations share several important features with the classical Euler and Navier--Stokes equations, such as vortex stretching mechanisms. Moreover, Boussinesq flows can be viewed as an analogue of three-dimensional axisymmetric flows, where vortex stretching terms arise naturally in the vorticity formulation.
The Boussinesq equations, under various boundary conditions and on both bounded and unbounded domains, have been widely studied, and there is a substantial literature on the mathematical theory of the deterministic Boussinesq system (see, e.g., \cite{CHAE1,hou,chae}). One of the main open questions is whether classical solutions of the two-dimensional Boussinesq equations can develop singularities in finite time. This problem remains unresolved and is listed among Yudovich's “eleven great problems of mathematical hydrodynamics” \cite{yud}.
\subsection{Objectives, main difficulties and contributions}

Optimal control of the 2D Boussinesq equations subject to finite-codimensional state constraints on a manifold in \(\mathbb{H}\times{H}\) is investigated here (Section~\ref{s3}). Specifically, the velocity \(y(t)\) and temperature \(\theta(t)\) are constrained, for all \(t\in[0,T]\), to lie on
\[
\mathcal{M}_1=\{y\in\mathbb{H}:\|y\|_{\mathbb{H}}=1\},\ \
\mathcal{M}_2=\{\theta\in H:\|\theta\|_{H}=1\}.
\]
We set
\[
\mathcal{M}=\mathcal{M}_1\times\mathcal{M}_2.
\]
These constraints represent the preservation of kinetic energy and thermal energy during the evolution of the
Boussinesq system. Upon projecting the Boussinesq 
dynamics onto the tangent
space $T_{(\bm{y},\theta)}\mathcal{M}$ leads to a constrained system (see \eqref{eq:constrained_Boussinesq}) whose state
remains on $\mathcal{M}$ and whose evolution is tangent to $\mathcal{M}$.

We prove that the constrained Boussinesq system
\eqref{eq:constrained_Boussinesq} has a unique strong solution, and a suitable quadratic cost functional admits an optimal control. The first-order necessary conditions for optimality are obtained from a coupled state--adjoint--control formulation that enforces the required tangency properties of the constraints.

The analysis involves several difficulties arising from
the equations and the geometry of the constraint manifold.
In particular, the nonlinear constraints
\[
\|\bm{y}(t)\|_{\mathbb{H}}=1,
\ \
\|\theta(t)\|_H=1,
\]
define the nonconvex manifold $\mathcal{M}$.
As a result, the admissible set does not form a linear subspace, so classical convex-analytic methods cannot be applied directly. Furthermore, the geometric structure is closely linked to the system dynamics. In particular, the projection operator onto the tangent space,
\[
\pi_{(\bm{y}(t),\theta(t))} : \mathbb{H} \times H \to T_{(\bm{y}(t),\theta(t))}\mathcal{M},
\]
depends nonlinearly on the evolving state \((\bm{y}(t),\theta(t))\). 
This induces a nonlinear and nonlocal coupling between the evolution equations and the underlying manifold structure. One of the main difficulties we face is that 
\[
( \bm{y}\cdot\nabla\theta, {\mathcal{A}}\theta )_{{H}} \neq 0,
\]
so this term has to be handled with particular care.

A further difficulty is that the control and the linearized perturbations must be tangent to $\mathcal{M}$ at all times.
This tangential invariance property requires delicate consistency arguments, since the tangent spaces themselves vary with time and depend on the solution. 
Such features go beyond the scope of classical approaches for unconstrained fluid control problems and necessitate a careful integration of geometric and analytical techniques.
To address these difficulties, we combine functional analysis, geometry, and variational methods. 
By using a Faedo--Galerkin scheme together with the Banach--Alaoglu theorem, the Aubin--Lions--Simon compactness lemma, and the Lions--Magenes lemma, we obtain a global strong solution and prove that it stays on the constraint manifold $\mathcal{M}$. In contrast, the authors of  \cite{gb} used
semigroup methods and fixed-point techniques to study the two-dimensional constrained NSE. We obtain the adjoint system by variational differentiation and projection onto the tangent spaces of the manifold. The optimal control is then given by the negative tangential projection of the adjoint variable:
\[
\mathfrak{H}^*(t)=-\pi_{\bm{y}^*(t)}(\bm{\vartheta}(t))
=-\bm{\vartheta}(t)+(\bm{\vartheta}(t),\bm{y}^*(t))_{\mathbb{H}}\bm{y}^*(t),
\qquad t\in[0,T].
\]

Thus, the optimal state $(\bm{y}^*(t),\theta^*(t))$ remains on
$\mathcal{M}$, and the optimal control $\mathfrak{H}^*(t)$ belongs to
$T_{\bm{y}^*(t)}\mathcal{M}_1$ (see Theorem~\ref{main}). 
The main contributions of this work are:
\begin{enumerate}
    \item We analyze a Boussinesq optimal control problem where the state is constrained to lie on a nonlinear manifold.
    \item We use methods from geometric PDEs and optimal control
    theory to analyze this problem.
    \item We show how the constraint affects the analysis of the
    system and the necessary conditions for optimality.
\end{enumerate}

To study the first-order optimality conditions, we take
$\mathcal U_{\mathrm{ad}}=L^2(0,T;\mathbb{H})$ as the control space.
This space remains convex and does not depend on the state. Every control \(\mathfrak{H}\in\mathcal{U}_{ad}\) is incorporated into the dynamics only by projecting it onto \(T_{\bm y}\mathcal{M}_1\); such a projection preserves the energy constraint and allows us to derive the adjoint equations, prove that the control-to-state map is differentiable, and show the existence of optimal controls. It appears to be the first study to address these questions together for the constrained two-dimensional Boussinesq system.
\subsection{Outline of the manuscript}
The paper is organized in the following manner. 

Section~\ref{s2} reviews the functional settings and introduces the Stokes operator along with the associated nonlinear terms. We then write the constrained Boussinesq system in abstract form and prove that the
forced two-dimensional system with finite-codimensional constraints has a unique strong solution (see Theorem~\ref{well}).

In Section~\ref{s3}, the optimal control problem is studied. We then introduce the linearized system (Subsection \ref{s3.1}) and the adjoint system (Subsection \ref{s3.2}), and analyze their well-posedness.

The derivation of first-order necessary optimality conditions for the optimal control problem is treated in Section~\ref{s4} (Theorem~\ref{main}). 
As part of this analysis, we establish that the control-to-state map admits a Fr\'echet derivative (Lemma~\ref{incest} and Proposition~\ref{prop4.4}) in the same section.
\section{Well-posedness of the system}{\label{s2}}
In this section, we prove that the forced 2D Boussinesq system with finite-codimensional constraints admits a unique strong solution. We begin by recalling some operators and the function spaces that are required for the proof.
After defining the torus \(\mathbb{T}^2\), we recall certain function spaces on it, which can also be found in \cite{kush,graf}. We then define the mean-zero spaces for both the vector and scalar variables on \(\mathbb{T}^2\). Subsequently, we give a precise definition of the Stokes operator and the convective terms, together with their principal properties. To end the section, we prove the well-posedness of the system using the Faedo–Galerkin approximation method.
\subsection{Functional Setting}

\subsubsection{The Two-Dimensional Torus}

The two-dimensional torus $\mathbb{T}^2$ is represented
by the unit square $[0,1]^2$ with periodic identification
of its opposite edges. Equivalently,
$
\mathbb{T}^2 = \mathbb{R}^2 / \mathbb{Z}^2,$
where two points $x, y \in \mathbb{R}^2$ are identified whenever
$x - y \in \mathbb{Z}^2.$
Functions defined on $\mathbb{T}^2$ can be regarded as functions on $\mathbb{R}^2$ that are periodic with period one in each coordinate, namely,
$
f(x + 1) = f(x), \ \  x \in \mathbb{R}^2. $

The Lebesgue measure on $\mathbb{T}^2$ is induced by the Lebesgue measure on the unit square $[0,1]^2$. Therefore, for any integrable function $f$,
\[
\int_{\mathbb{T}^2} f(x) \, dx
:= \int_{[0,1]^2} f(x) \, dx.
\]
A useful consequence of periodicity is that integration by parts does not produce boundary terms. In particular, for $f, g \in C^1(\mathbb{T}^2)$ and $j = 1,2$, we have 
\[
\int_{\mathbb{T}^2} \partial_j f(x) \, g(x) \, dx
=
-\int_{\mathbb{T}^2} f(x) \, \partial_j g(x) \, dx.
\]
For $m = (m_1, m_2) \in \mathbb{Z}^2$, we define
$
|m| = (m_1^2 + m_2^2)^{1/2}.$
Similarly, for $x = (x_1, x_2) \in \mathbb{T}^2$, $|x|$ denotes the Euclidean norm.

For $k \in \mathbb{N} \cup \{0\}$, the space $C^k(\mathbb{T}^2)$ consists of all functions whose partial derivatives up to order $k$ are continuous. The space $C^\infty(\mathbb{T}^2)$ denotes the collection of infinitely differentiable functions. The Lebesgue spaces $L^p(\mathbb{T}^2)$, $1 \le p \le \infty$, are defined in the usual way. For $s\geq0$, the Sobolev space $H^s(\mathbb{T}^2)$ is
defined through Fourier series by
\[
H^s(\mathbb{T}^2)
=
\left\{
f \in L^2(\mathbb{T}^2) :
\sum_{m \in \mathbb{Z}^2}
(1 + |m|^2)^s
|\widehat{f}(m)|^2
< \infty
\right\},
\]
with norm
\[
\|f\|_{H^s(\mathbb{T}^2)}
=
\left(
\sum_{m \in \mathbb{Z}^2}
(1 + |m|^2)^s
|\widehat{f}(m)|^2
\right)^{1/2},
\]
where $\widehat{f}(m)$ denotes the Fourier coefficient of $f$ corresponding to the frequency $m \in \mathbb{Z}^2$.

\subsubsection{Functional setting on the two-dimensional torus}

For \(1 \le p \le \infty\) and \(k \in \mathbb{N}\), we define the following spaces
\[
\mathbb{L}^p(\mathbb{T}^2) := L^p(\mathbb{T}^2; \mathbb{R}^2), \
\mathbb{H}^{k,p}(\mathbb{T}^2) := W^{k,p}(\mathbb{T}^2; \mathbb{R}^2),
\]
and write \(\mathbb{H}^k(\mathbb{T}^2) := \mathbb{H}^{k,2}(\mathbb{T}^2)\). With the inner product \( ( \bm{y}, \bm{u} )_{\mathbb{L}^2} = \int_{\mathbb{T}^2} \bm{y}(x) \cdot \bm{u}(x) \, dx\), \(\mathbb L^2(\mathbb{T}^2)\) is a Hilbert space, and \(\mathbb{H}^1(\mathbb{T}^2)\) is equipped with the inner product \( ( \bm{y}, \bm{u} )_{\mathbb{H}^1} = ( \bm{y}, \bm{u} )_{\mathbb{L}^2} + ( \nabla \bm{y}, \nabla \bm{u} )_{\mathbb{L}^2}\) and norm \(\|\bm{y}\|_{\mathbb{H}^1} = (\|\bm{y}\|_{\mathbb{L}^2}^2 + \|\nabla \bm{y}\|_{\mathbb{L}^2}^2)^{1/2}\).

\subsubsection{Mean-zero and divergence-free spaces}

We define the mean-zero Lebesgue space
\[
\mathbb{L}_0^2(\mathbb{T}^2) := \left\{ \bm{y} \in \mathbb{L}^2(\mathbb{T}^2) : \int_{\mathbb{T}^2} \bm{y}(x) \, dx = \boldsymbol{0} \right\}.
\]
We define \(\mathbb{H}\) as the collection of velocity fields that satisfy both the zero-mean and divergence-free conditions:
\[
\mathbb{H} := \left\{ \bm{y} \in \mathbb{L}_0^2(\mathbb{T}^2) :
\operatorname{div} \bm{y} = 0 \right\}.
\]
and the corresponding Sobolev space is \(\mathbb{V} := \mathbb{H}^1(\mathbb{T}^2) \cap \mathbb{H}\). 
The inner product on $\mathbb{H}$ is defined by $(\pmb{y}, \pmb{u})_{\mathbb{H}} = (\pmb{y}, \pmb{u})_{\mathbb{L}^2}$, with corresponding norm $\|\pmb{y}\|_{\mathbb{H}} = \|\pmb{y}\|_{\mathbb{L}^2}$. Similarly, for $\mathbb{V}$, we use the equivalent inner product $(\pmb{y}, \pmb{u})_{\mathbb{V}} = (\nabla \pmb{y}, \nabla \pmb{u})_{\mathbb{L}^2}$ and the norm $\|\pmb{y}\|_{\mathbb{V}} = \|\nabla \pmb{y}\|_{\mathbb{L}^2}$.
By the Poincaré-Wirtinger inequality on $\mathbb{T}^2$, there exists a constant $\lambda_1 > 0$ satisfying $\|\pmb{y}\|_{\mathbb{H}} \leq \lambda_1^{-1/2}\|\nabla \pmb{y}\|_{\mathbb{L}^2}$ for all $\pmb{y} \in \mathbb{V}$ (see the following subsection). As a result, the norm $\|\cdot\|_{\mathbb{V}}$ is equivalent to the usual $\mathbb{H}^1$-norm on $\mathbb{V}$.

Similarly, for the temperature equation, which is scalar-valued, we define the mean-zero spaces
\[
H := L_0^2(\mathbb{T}^2) := \left\{ \theta \in L^2(\mathbb{T}^2;\mathbb{R}) : \int_{\mathbb{T}^2} \theta(x) \, dx = 0 \right\}, \ \
V := H^1(\mathbb{T}^2) \cap H.
\]
We denote the norm in $H$ by $\|\theta\|_{H}:=\|\theta\|_{L^2}$. On \(V\), we take \(\|\theta\|_V := \|\nabla\theta\|_{\mathbb H}\). The Poincar\'e--Wirtinger inequality guarantees that this norm is equivalent to the usual \(H^1(\mathbb{T}^2)\)-norm on \(V\). The dual spaces \(\mathbb{V}'\) and \(V'\) are equipped with the norms
\[
\|\bm{g}\|_{\mathbb{V}'} := \sup \left\{ {}_{\mathbb{V}'}\langle \bm{g}, \bm{\phi} \rangle_{\mathbb{V}} : \bm{\phi} \in \mathbb{V},\ \|\bm{\phi}\|_{\mathbb{V}} \le 1 \right\}, \quad
\|g\|_{V'} := \sup \left\{ {}_{V'}\langle g, f \rangle_{V} : f \in V,\ \|f\|_{V} \le 1 \right\}.
\]
Throughout this paper, \({}_{\mathbb{V}'}\langle \cdot, \cdot \rangle_{\mathbb{V}}\) and \({}_{V'}\langle \cdot, \cdot \rangle_{V}\) denote the duality pairings between \(\mathbb{V}'\) and \(\mathbb{V}\), and \(V'\) and \(V\), respectively.
\subsection{Linear operators}
We denote by \({\textbf{Proj}_{\mathbb H}}:\mathbb{L}^{2}(\mathbb{T}^{2})\to\mathbb{H}\) the Leray--Helmholtz projection. The operator \(A : D(A) \subset \mathbb{H} \to \mathbb{H}\) (the Stokes operator) as \(A\bm y := -{\textbf{Proj}_{\mathbb H}}(\Delta \bm y)\), with \(D(A) = H^2(\mathbb{T}^2) \cap \mathbb{V}\),  is positive and self-adjoint on \(\mathbb{H}\) with compact inverse, and satisfies \(D(A^{1/2})=\mathbb{V}\) and \( ( A\bm{y},\bm{y})_{\mathbb{H}}=\|\nabla\bm{y}\|_{\mathbb{H}^{}}^{2}=\|\bm{y}\|_{\mathbb{V}}^{2}\). Its eigenvalues \(0<\lambda_{1}\le\lambda_{2}\le\cdots\to\infty\) and we take the eigenfunctions associated with these eigenvalues \(\{\bm{w}_{k}\}_{k\ge1}\) forming an orthonormal basis of \(\mathbb{H}\) and orthogonal basis of \(\mathbb{V}\). Hence, for every \(\bm{y}\in D(A)\), we have 
\[
\bm{y}=\sum_{k=1}^{\infty}( \bm{y},\bm{w}_{k})_{\mathbb{H}}\bm{w}_{k}, \ \
A\bm{y}=\sum_{k=1}^{\infty}\lambda_{k}( \bm{y},\bm{w}_{k})_{\mathbb{H}}\bm{w}_{k},
\]
and consequently,
\[
\|\nabla\bm{y}\|_{\mathbb{L}^{2}}^{2}
=\sum_{k=1}^{\infty}\lambda_{k}|( \bm{y},\bm{w}_{k})_{\mathbb{H}}|^{2}
\ge \lambda_{1}\|\bm{y}\|_{\mathbb{H}}^{2},
\]
which leads to the Poincar\'e--Wirtinger inequality \(\|\bm{y}\|_{\mathbb{H}}^{2}\le \lambda_{1}^{-1}\|\nabla\bm{y}\|_{\mathbb{L}^{2}}^{2}\) for \(\bm{y}\in\mathbb{V}\).

Similarly, the operator \(\mathcal{A}:D(\mathcal{A})\subset H\to H\), \(\mathcal{A}\theta:=-\Delta\theta\), with \(D(\mathcal{A})=H^{2}(\mathbb{T}^{2})\cap V\), is positive self-adjoint on \(H\) with compact inverse, \(D(\mathcal{A}^{1/2})=V\), and \(( \mathcal{A}\theta,\theta)_{H}=\|\nabla\theta\|_{\mathbb H}^{2}\). Its eigenvalues \(0<\mu_{1}\le\mu_{2}\le\cdots\to\infty\) yield \(\|\nabla\theta\|_{\mathbb H}^{2}\ge \mu_{1}\|\theta\|_{H}^{2}\), or equivalently, \(\|\theta\|_{H}^{2}\le \mu_{1}^{-1}\|\nabla\theta\|_{\mathbb H}^{2}\) for \(\theta\in V\).
\medskip

\subsection{The Bilinear Operator}
We now introduce the trilinear map \(b(\cdot,\cdot,\cdot) : \mathbb{V} \times \mathbb{V} \times \mathbb{V} \to \mathbb{R}\)   by
\[
b(\bm{y},\bm{u},\bm{w}) := \int_{\mathbb{T}^2} (\bm{y} \cdot \nabla)\bm{u} \cdot \bm{w} \, dx
= \sum_{i,j=1}^2 \int_{\mathbb{T}^2} y_i(x)\, \frac{\partial u_j(x)}{\partial x_i}\, w_j(x)\, dx.
\]
Integration by parts, together with the periodic boundary conditions, yields the following skew-symmetry properties:
\begin{align}
b(\bm{y},\bm{u},\bm{w}) &= -b(\bm{y},\bm{w},\bm{u}), \ \ \bm{y},\bm{u},\bm{w} \in \mathbb{V}, \label{eq:skew}\\
b(\bm{y},\bm{u},\bm{u}) &= 0, \ \  \bm{y},\bm{u} \in \mathbb{V}. \label{trii}
\end{align}
From Hölder's, Ladyzhenskaya's (see \cite[Lemma 1, Chapter 1]{M3}), and Poincaré's inequalities, it follows that
\begin{align}{\label{eq:b_bound}}
|b(\bm{y},\bm{u},\bm{w})|
&\le \|\bm{y}\|_{\mathbb{L}^4}\|\bm{u}\|_{\mathbb{V}}\|\bm{w}\|_{\mathbb{L}^4}\leq 2^{1/2} \|\bm{y}\|_{\mathbb{H}}^{1/2}\|\bm{y}\|_{\mathbb{V}}^{1/2} \|\bm{u}\|_{\mathbb{V}} \|\bm{w}\|_{\mathbb{H}}^{1/2}\|\bm{w}\|_{\mathbb{V}}^{1/2} \nonumber\\& \leq C \|\bm{y}\|_{\mathbb{V}}  \|\bm{u}\|_{\mathbb{V}}  \|\bm{w}\|_{\mathbb{V}} 
\ \ \bm{y},\bm{u},\bm{w} \in \mathbb{V}.
\end{align}
For \(\bm y,\bm u\in\mathbb{V}\), if \(b(\bm y,\bm u,\cdot)\) is a continuous linear functional on \(\mathbb{V}\), it is denoted by \(B(\bm y,\bm u)\in\mathbb{V}'\). Consequently, we have 
\[
\|B(\bm{y},\bm{u})\|_{\mathbb{V}'}
\le C \|\bm{y}\|_{\mathbb{V}} \|\bm{u}\|_{\mathbb{V}},
\ \  \bm{y},\bm{u} \in \mathbb{V}.
\] 
For notational convenience, we set
\[
B(\bm{y}) := B(\bm{y},\bm{y}) := {\textbf{Proj}_{\mathbb H}} ((\bm{y}\cdot\nabla)\bm{y}).
\]
Once again, an application of Hölder's, Ladyzhenskaya's (see \cite[Lemma 1, Chapter 1]{M3}), and Poincaré's inequalities leads to 
\begin{align}{\label{eq:b_refined}}
|b(\bm{y},\bm{u},\bm{w})|
\le C \|\bm{y}\|_{\mathbb{H}}^{1/2} \|\bm{y}\|_{\mathbb{V}}^{1/2}
\|\bm{u}\|_{\mathbb{V}}^{1/2} \|A\bm{u}\|_{\mathbb{H}}^{1/2}
\|\bm{w}\|_{\mathbb{H}},
\ \ \bm{y} \in \mathbb{V}, \ \bm{u} \in D(A),\ \bm{w} \in \mathbb{H}. 
\end{align}
Hence, \(b\) possesses a unique extension to a bounded trilinear map \(b:\mathbb{V}\times D(A)\times\mathbb{H}\to\mathbb{R}\). In the same manner, \(B\) can be uniquely extended to a bounded bilinear operator \(B:\mathbb{V}\times D(A)\to\mathbb{H}\). Moreover, for all \(y\in D(A)\), the cancellation identity
\begin{align}{\label{eq:B_A_cancel}}
(B(\bm y,\bm y), A\bm y)_{\mathbb{H}} = 0
\end{align}
is satisfied on \(\mathbb{T}^2\) (see \cite[Lemma 3.1]{M3}). This identity is essential in the energy estimates for the velocity equation

\subsection{Abstract formulation of the constrained Boussinesq system}

The two-dimensional Boussinesq system on the torus $\mathbb{T}^2$ is given by
\begin{equation}\label{eq:Boussinesq_strong}
\begin{cases}
\displaystyle
\frac{\partial \bm{y}(x,t)}{\partial t}
- \nu \Delta \bm{y}(x,t)
+ \big(\bm{y}(x,t) \cdot \nabla\big) \bm{y}(x,t)
+ \nabla p(x,t)
= \beta \,\theta(x,t)\, e_2 + \bm{f}(x,t),
& (x,t) \in \mathbb{T}^2 \times (0,T), \\[1ex]

\displaystyle
\frac{\partial \theta(x,t)}{\partial t}
+ \bm{y}(x,t) \cdot \nabla \theta(x,t)
= \kappa \Delta \theta(x,t),
& (x,t) \in \mathbb{T}^2 \times (0,T), \\[1ex]

\nabla \cdot \bm{y}(x,t) = 0,
& (x,t) \in \mathbb{T}^2 \times (0,T), \\[1ex]

\bm{y}(x,0) = \bm{y}_0(x), \quad \theta(x,0) = \theta_0(x),
& x \in \mathbb{T}^2.
\end{cases}
\end{equation} 
where 
\[
\bm{y} : \mathbb{T}^2 \times [0,T] \to \mathbb{R}^2,\ \ 
\bm{f} : \mathbb{T}^2 \times [0,T] \to \mathbb{R}^2,\ \
p : \mathbb{T}^2 \times (0,T) \to \mathbb{R},\ \
\theta : \mathbb{T}^2 \times [0,T] \to \mathbb{R}
\]
denote the velocity field, external force (control), pressure, and temperature, respectively. 
The constants \( \nu > 0 \) and \( \kappa > 0 \) represent the kinematic viscosity and thermal diffusivity, respectively, and \( \beta = g\alpha > 0 \) is the buoyancy coefficient, with \( e_2 = (0,1) \).
Moreover, \(\bm{y}(\cdot,t),\, p(\cdot,t),\, \theta(\cdot,t)\) and \(\bm{f}(\cdot,t)\) satisfy the periodic conditions
\[
\bm{y}(x+Le_i,t)=\bm{y}(x,t),\; p(x+Le_i,t)=p(x,t),\; \theta(x+Le_i,t)=\theta(x,t),\; \bm{f}(x+Le_i,t)=\bm{f}(x,t),
\]
for all \(x\in\mathbb{R}^2,\; t\in(0,T)\) and \(i=1,2\), where \(\{e_1, e_2\}\) is the canonical basis of \(\mathbb{R}^2\).
Applying the Leray--Helmholtz projection ${\textbf{Proj}_{\mathbb H}}$ onto divergence-free vector fields, the velocity equation can be expressed in the abstract form
\begin{equation}\label{eq:u_abstract}
\begin{cases}
\displaystyle
\frac{d\bm{y}(t)}{dt}
= -\nu A \bm{y}(t) - {B}(\bm{y}(t),\bm{y}(t)) + \beta {\textbf{Proj}_{\mathbb H}}(\theta(t)e_2) + \mathfrak{H}(t), 
\quad t>0, \\[1ex]
\bm{y}(0) = \bm{y}_0,
\end{cases}
\end{equation}
where $\bm{y}_0 \in \mathbb{H}$ and $\mathfrak{H} := {\textbf{Proj}_{\mathbb H}} \bm{f} \in {{L}^2(0,T;\mathbb{H})}$.
We write the temperature equation in the following abstract form
\begin{equation}\label{eq:theta_abstract}
\begin{cases}
\displaystyle
\frac{d\theta(t)}{dt}
= -\kappa \mathcal{A} \theta(t) - (\bm{y}(t)\cdot\nabla)\theta(t), \ t>0,\\[1ex]
\theta(0) = \theta_0,
\end{cases}
\end{equation}

We introduce the constraint sets
\begin{equation}\label{eq:manifolds}
\mathcal{M}_1 := \{\bm{y} \in \mathbb{H} : \|\bm{y}\|_{\mathbb{H}} = 1\}, \quad
\mathcal{M}_2 := \{\theta \in H : \|\theta\|_{H} = 1\}.
\end{equation}
Define $\mathcal{M} := \mathcal{M}_1 \times \mathcal{M}_2 \subset \mathbb{H} \times H$, which is a smooth embedded submanifold.
For $(\bm{y},\theta) \in \mathcal{M}$, the tangent space is given by 
\[
T_{(\bm{y},\theta)} \mathcal{M}
= T_{\bm{y}} \mathcal{M}_1 \times T_\theta \mathcal{M}_2
= \{\bm{\varrho} \in \mathbb{H} : ( \bm{y},\bm{\varrho})_{\mathbb{H}} = 0\}
\times \{\phi \in H : (\phi,\theta)_{H} = 0\}.
\]
Define the projection $\pi_{(\bm{y},\theta)} : \mathbb{H} \times H \to T_{(\bm{y},\theta)} \mathcal{M}$ by
\[
\pi_{(\bm y,\theta)}(\bm{\varrho},\psi)
= \big( \bm{\varrho} - ( \bm{\varrho},\bm{\bm y})_{\mathbb{H}} \bm{y},\ \psi - (\psi,\theta)_{H}\theta \big),
\quad (\bm{\varrho},\psi)\in \mathbb{H} \times H,
\]
which is the orthogonal projection onto $T_{(\bm{y},\theta)} \mathcal{M}$.

For an initial condition $(\bm{y}_0,\theta_0)\in\mathcal{M}$, we aim to ensure that the solution remains on the manifold $\mathcal M$ for all times. To enforce this constraint, we project the dynamics onto the tangent space $T_{(\bm{y},\theta)}\mathcal{M}$.
For the velocity component, we have
\begin{align}\label{eq:projection_velocity_full}
\pi_{\bm{y}}\big(-\nu A\bm{y} - {B}(\bm{y},\bm{y}) + \beta {\textbf{Proj}_{\mathbb H}}(\theta e_2) + \mathfrak{H}\big)
&= -\nu A\bm{y} - {B}(\bm{y},\bm{y}) + \beta {\textbf{Proj}_{\mathbb H}}(\theta e_2) + \mathfrak{H} \notag\\
&\quad - \big( -\nu A\bm{y} - {B}(\bm{y},\bm{y}) + \beta {\textbf{Proj}_{\mathbb H}}(\theta e_2) + \mathfrak{H},\, \bm{y} \big)_{\mathbb{H}}\, \bm{y} \notag\\
&= -\nu A\bm{y} + \nu \|\nabla \bm{y}\|_{\mathbb{L}^2}^2 \bm{y}
- {B}(\bm{y},\bm{y}) + \beta {\textbf{Proj}_{\mathbb H}}(\theta e_2) \notag\\
&\quad - \beta ( \theta e_2,\bm{y})_{\mathbb{H}}\,\bm{y}
+ \mathfrak{H} - ( \mathfrak{H},\bm{y})_{\mathbb{H}}\,\bm{y},
\end{align}
where we have used $( {B}(\bm{y},\bm{y}),\bm{y})_{\mathbb{H}}=0$ and $ ( A\bm{y},\bm{y})_{\mathbb{H}}=\|\nabla \bm{y}\|_{\mathbb{H}}^2$.
Similarly, for the temperature component, we have 
\begin{align}\label{eq:projection_theta_full}
\pi_\theta\big(-\kappa {\mathcal{A}} \theta - \bm{y}\cdot\nabla \theta\big)
&= -\kappa {\mathcal{A}} \theta - \bm{y}\cdot\nabla \theta  - \big( -\kappa {\mathcal{A}} \theta - \bm{y}\cdot\nabla \theta,\, \theta \big)_{H}\, \theta \notag\\
&= -\kappa {\mathcal{A}} \theta + \kappa \|\nabla \theta\|_{{\mathbb H}}^2 \theta
- \bm{y}\cdot\nabla \theta,
\end{align}
where we have used $( \bm{y}\cdot\nabla \theta,\theta)_{H}=0$.

By projecting the equations \eqref{eq:u_abstract}--\eqref{eq:theta_abstract} onto the tangent space \(T_{(\bm{y},\theta)}\mathcal{M}\), we obtain a two-dimensional Boussinesq system subject to finite-codimensional constraints, given by
\begin{equation}\label{eq:constrained_Boussinesq}
\begin{cases}
\displaystyle
\frac{d\bm{y}(t)}{dt}
= -\nu A \bm{y}(t)
+ \nu \|\nabla \bm{y}(t)\|_{\mathbb{L}^2}^2 \bm{y}(t)
- {B}(\bm{y}(t),\bm{y}(t))
+ \beta {\textbf{Proj}_{\mathbb H}}(\theta(t)e_2)
- \beta (\theta(t)e_2,\bm{y}(t))_{\mathbb{H}}\,\bm{y}(t) \\[1ex]
\hspace{2.5cm}
+ \mathfrak{H}(t) - ( \mathfrak{H}(t),\bm{y}(t))_{\mathbb{H}}\,\bm{y}(t), \\[1.5ex]
\displaystyle
\frac{d\theta(t)}{dt}
= -\kappa {\mathcal{A}} \theta(t)
+ \kappa \|\nabla \theta(t)\|_{\mathbb H}^2 \theta(t)
- \bm{y}(t)\cdot\nabla \theta(t), \\[1.2ex]
(\bm{y},\theta)(0) = (\bm{y}_0,\theta_0).
\end{cases}
\end{equation}
where $(\bm{y}_0,\theta_0) \in \mathbb{V}\cap\mathcal{M}_1 \times V \cap \mathcal{M}_2$ and 
$\mathfrak{H} \in {{L}^2(0,T;\mathbb{H})}$ is the control added to the system. The main advantage of considering the problem on $\mathbb{T}^2$ is that,
\[
({B}(\bm{y},\bm{y}), A\bm{y} )_{\mathbb{H}} = 0, 
\  \text{ for all } \ \bm{y} \in D(A),
\]
which plays a crucial role in deriving energy estimates. We also point out that the term
$(\bm y\cdot\nabla\theta,\mathcal{A}\theta)_H$
does not vanish in general for
$\bm y\in\mathbb{V}$ and $\theta\in D(\mathcal{A})$.
\begin{proposition}
Let $\mathbb{H}$ and $H$ be nontrivial real Hilbert spaces. Define
\[
\mathcal{M}_1 := \{\bm{y} \in \mathbb{H} : \|\bm{y}\|_{\mathbb{H}} = 1\}, \ \ 
\mathcal{M}_2 := \{\theta \in H : \|\theta\|_{H} = 1\}.
\]
Then $\mathcal{M} := \mathcal{M}_1 \times \mathcal{M}_2$ is an embedded submanifold of $\mathbb{H} \times H$ of codimension $2$.
\end{proposition}

\begin{proof}
Let us define the smooth map
\[
F : \mathbb{H} \times H \to \mathbb{R}^2, 
\ \ 
F(\bm{w},\psi) := \big( \|\bm{w}\|_{\mathbb{H}}^2 - 1,\ \|\psi\|_{H}^2 - 1 \big).
\]
Then $\mathcal{M} = F^{-1}(0,0)$. For any $(\bm{y},\theta) \in \mathcal{M}$, the Fr\'echet derivative of $F$ is
\[
DF(\bm{y},\theta)(\bm{\varrho},\phi) = \big(2( \bm{y},\bm{\varrho})_{\mathbb{H}},\ 2( \theta,\phi)_{H} \big),
\ \  (\bm{\varrho},\phi) \in \mathbb{H} \times H.
\]
This linear map is surjective onto $\mathbb{R}^2$.
Indeed, given $(a,b)\in\mathbb{R}^2$, choosing
$\bm \varrho=\frac{a}{2}\bm{y},$ $\phi=\frac{b}{2}\theta,$
yields $DF(\bm{y},\theta)(\bm \varrho,\phi)=(a,b).$ Hence $DF(\bm{y},\theta)$ is surjective for every $(\bm{y},\theta)\in\mathcal{M}$. Moreover,
$\ker DF(\bm{y},\theta)
=\bm{y}^\perp\times\theta^\perp$
is complemented in $\mathbb H\times H$, since $\mathbb{H}\times H$ is a Hilbert space.
Thus, by the Local Submersion Theorem
\cite[Theorem 2.5.13]{AbrahamMarsdenRatiu1988},
$\mathcal{M}=F^{-1}(0,0)$ is an embedded submanifold of
$\mathbb H\times H$ of codimension $2$.
 Moreover, we have 
\[
T_{(\bm{y},\theta)}\mathcal{M}
= \ker DF(\bm{y},\theta) =\bm{y}^\perp\times\theta^\perp
= \{(\bm{\varrho},\phi) \in \mathbb{H} \times H : ( \bm{y},\bm{\varrho})_{\mathbb{H}} = 0,\ ( \theta,\phi)_{H} = 0 \},
\]
which coincides with $T_{\bm{y}} \mathcal{M}_1 \times T_\theta \mathcal{M}_2$.
\end{proof}

\medskip

\begin{lemma}{\label{lemLocal}}
Let $\mathfrak{H}\in L^2(0,T;\mathbb{H})$. The mapping
$D(A) \times D(\mathcal{A}) \ni (\bm{y},\theta) \mapsto \mathcal{F}(\bm{y},\theta) \in \mathbb{H} \times H,$
defined by
\begin{align}
\mathcal{F}(\bm{y},\theta):=(F_1(\bm{y},\theta),F_2(\bm{y}, \theta))
:=
\Big(&-\nu A\bm{y} + \nu \|\nabla \bm{y}\|_{\mathbb{H}}^2 \bm{y} - {B}(\bm{y},\bm{y})
+ \beta {\mathbf{Proj}_{\mathbb H}}(\theta e_2)
- \beta ( \theta e_2,\bm{y})_{\mathbb{H}} \bm{y} \notag\\
&+ \mathfrak{H} - ( \mathfrak{H},\bm{y})_{\mathbb{H}} \bm{y},\ 
-\kappa {\mathcal{A}} \theta + \kappa \|\nabla \theta\|_{\mathbb H}^2 \theta - \bm{y}\cdot\nabla \theta \Big)
\label{eq:F_boussinesq}
\end{align}
is locally Lipschitz. More precisely, for any \(R>0\), a constant \(L_{R,\|\mathfrak{H}\|_{\mathbb{H}}}>0\) exists for which
\begin{align}
\|\mathcal{F}(\bm{y},\theta) - \mathcal{F}(\bm{\varrho},\eta)\|_{\mathbb{H} \times H}
\le L_{R,\|\mathfrak{H}\|_{\mathbb{H}}}
\Big( \|\bm{y}-\bm{\varrho}\|_{D(A)} + \|\theta-\eta\|_{D(\mathcal{A})} \Big),
\label{eq:Lip_boussinesq}
\end{align}
for all $(\bm{y},\theta),(\bm{\varrho},\eta)$ satisfying
\[
\|\bm{y}\|_{D(A)},\ \|\bm{\varrho}\|_{D(A)},\ \|\theta\|_{D(\mathcal{A})},\ \|\eta\|_{D(\mathcal{A})} \le R.
\]
\end{lemma}
\begin{proof}
We estimate each term in \eqref{eq:F_boussinesq} separately.
Since $A$ and ${\mathcal{A}}$ are linear, we have
\begin{align}
\|\nu A\bm{y} - \nu A\bm{\varrho}\|_{\mathbb{H}} 
&= \nu \|A(\bm{y}-\bm{\varrho})\|_{\mathbb{H}} 
\le \nu \|\bm{y}-\bm{\varrho}\|_{D(A)}, \label{eq:lin_u}\\
\|\kappa {\mathcal{A}} \theta - \kappa {\mathcal{A}} \eta\|_{H}
&= \kappa \|{\mathcal{A}}(\theta-\eta)\|_{H}
\le \kappa \|\theta-\eta\|_{D(\mathcal{A})}. \label{eq:lin_theta}
\end{align}
We now focus on the following nonlinear term
\begin{equation}\label{eq:nonlinear_term_estimate}
\begin{aligned}
\left\| \nu \|\nabla \bm{y}\|_{\mathbb{H}}^2 \bm{y} - \nu \|\nabla \bm{\varrho}\|_{\mathbb{H}}^2 \bm{\varrho} \right\|_{\mathbb{H}}
&\le \nu \|\nabla \bm{y}\|_{\mathbb{H}}^2 \|\bm{y} - \bm{\varrho}\|_{\mathbb{H}} 
+ \nu \left| \|\nabla \bm{y}\|_{\mathbb{H}}^2 - \|\nabla \bm{\varrho}\|_{\mathbb{H}}^2 \right| \|\bm{\varrho}\|_{\mathbb{H}} \\
&\le \frac{\nu}{\sqrt{\lambda_1}} 
\left( \|\nabla \bm{y}\|_{\mathbb{H}}^2 
+ \|\nabla \bm{y}\|_{\mathbb{H}} \|\nabla \bm{\varrho}\|_{\mathbb{H}} 
+ \|\nabla \bm{\varrho}\|_{\mathbb{H}}^2 \right)
\|\nabla (\bm{y} - \bm{\varrho})\|_{\mathbb{H}}.
\end{aligned}
\end{equation}
Similarly, we estimate
\begin{align}
\kappa \big\| \|\nabla \theta\|_{\mathbb H}^2 \theta 
- \|\nabla \eta\|_{H}^2 \eta \big\|_{H}
&\le \kappa \|\nabla \theta\|_{\mathbb H}^2 \|\theta-\eta\|_{H}
+ \kappa \big|\|\nabla \theta\|_{\mathbb H}^2 
- \|\nabla \eta\|_{\mathbb H}^2\big| \|\eta\|_{H} \notag\\
&\le \frac{\kappa}{\sqrt{\mu_1}}
\left( \|\nabla \theta\|_{\mathbb H}^2
+ \|\nabla \theta\|_{\mathbb H}\|\nabla \eta\|_{\mathbb H}
+ \|\nabla \eta\|_{\mathbb H}^2 \right)
\|\nabla(\theta-\eta)\|_{\mathbb H}.
\label{eq:quad_theta}
\end{align}
Using Hölder's, Gagliardo--Nirenberg's and Agmon's inequalities, we obtain
\begin{equation}\label{eq:B_estimate}
\begin{aligned}
\|B(\bm{y},\bm{y})-B(\bm \varrho,\bm \varrho)\|_{\mathbb H}
&\le
\|B(\bm{y}-\bm \varrho,\bm{y})\|_{\mathbb H}
+\|B(\bm \varrho,\bm{y}-\bm \varrho)\|_{\mathbb H} \\
&\le
C\|\bm{y}-\bm \varrho\|_{\mathbb L^4}\|\nabla\bm{y}\|_{\mathbb L^4}
+C\|\bm \varrho\|_{L^\infty}\|\nabla(\bm{y}-\bm \varrho)\|_{\mathbb H} \\
&\le
C\|\bm{y}-\bm \varrho\|_{\mathbb H}^{1/2}
\|\nabla(\bm{y}-\bm \varrho)\|_{\mathbb H}^{1/2}
\|\nabla\bm{y}\|_{\mathbb H}^{1/2}
\|A\bm{y}\|_{\mathbb H}^{1/2}  \\
&\quad
+C\|\bm \varrho\|_{\mathbb H}^{1/2}
\|A\bm \varrho\|_{\mathbb H}^{1/2}
\|\nabla(\bm{y}-\bm \varrho)\|_{\mathbb H}.
\end{aligned}
\end{equation}
Moreover, we deduce 
\begin{equation}\label{eq:U_estimate}
\begin{aligned}
\|\mathfrak{H} - ( \mathfrak{H},\bm{y})_{\mathbb{H}} \bm{y} - (\mathfrak{H} - ( \mathfrak{H},\bm{\varrho})_{\mathbb{H}} \bm{\varrho})\|_{\mathbb{H}} 
&= \|( \mathfrak{H},\bm{y})_{\mathbb{H}} \bm{y} - ( \mathfrak{H},\bm{\varrho})_{\mathbb{H}} \bm{\varrho}\|_{\mathbb{H}} \\
&\le |( \mathfrak{H},\bm{y})_{\mathbb{H}}|\,\|\bm{y}-\bm{\varrho}\|_{\mathbb{H}} 
+ |( \mathfrak{H},\bm{y}-\bm{\varrho})_{\mathbb{H}}|\,\|\bm{\varrho}\|_{\mathbb{H}} \\
&\le \|\mathfrak{H}\|_{\mathbb{H}}(\|\bm{y}\|_{\mathbb{H}} + \|\bm{\varrho}\|_{\mathbb{H}})
\|\bm{y}-\bm{\varrho}\|_{\mathbb{H}}.
\end{aligned}
\end{equation}
Since \({\textbf{Proj}_{\mathbb H}}\) is bounded on \(\mathbb{H}\), we have
\begin{equation}\label{eq:coupling1}
\|\beta {\textbf{Proj}_{\mathbb H}}(\theta e_2) - \beta {\textbf{Proj}_{\mathbb H}}(\eta e_2)\|_{\mathbb{H}}
\le C \|\theta-\eta\|_{H}.
\end{equation}
Similar to the above calculations, we find 
\begin{align}
&\|\beta ( \theta e_2,\bm{y})_{\mathbb{H}}\, \bm{y} 
- \beta ( \eta e_2,\bm{\varrho})_{\mathbb{H}}\, \bm{\varrho}\|_{\mathbb{H}} \notag\\
&\quad = \beta \| ( \theta e_2,\bm{y})_{\mathbb{H}}(\bm{y}-\bm{\varrho}) 
+ \big[( \theta e_2,\bm{y})_{\mathbb{H}} 
- ( \eta e_2,\bm{\varrho})_{\mathbb{H}}\big] \bm{\varrho} \|_{\mathbb{H}} \notag\\
&\quad \le \beta |( \theta e_2,\bm{y})_{\mathbb{H}}| \|\bm{y}-\bm{\varrho}\|_{\mathbb{H}}
+ \beta \big|( \theta e_2,\bm{y}-\bm{\varrho})_{\mathbb{H}} 
+ ( (\theta-\eta)e_2,\bm{\varrho})_{\mathbb{H}}\big| \|\bm{\varrho}\|_{\mathbb{H}} \notag\\
&\quad \le \beta \|\theta\|_{H}\|\bm{y}\|_{\mathbb{H}} \|\bm{y}-\bm{\varrho}\|_{\mathbb{H}}
+ \beta \|\theta\|_{H}\|\bm{y}-\bm{\varrho}\|_{\mathbb{H}}\|\bm{\varrho}\|_{\mathbb{H}}
+ \beta \|\theta-\eta\|_{H}\|\bm{\varrho}\|_{\mathbb{H}}^2.
\label{eq:coupling2}
\end{align}
\allowdisplaybreaks
For the bilinear term in the temperature equation, we estimate by applying Hölder's, Gagliardo--Nirenberg's, and Agmon's inequalities as 
\begin{align}
\|\bm{y}\cdot\nabla \theta-\bm{\varrho}\cdot\nabla\eta\|_{H}
&\le
\|(\bm{y}-\bm{\varrho})\cdot\nabla\theta\|_{H}
+\|\bm{\varrho}\cdot\nabla(\theta-\eta)\|_{H}
\notag\\
&\le
C\|\bm{y}-\bm{\varrho}\|_{\mathbb L^4}\|\nabla\theta\|_{\mathbb L^4}
+C\|\bm{\varrho}\|_{\mathbb L^\infty}\|\nabla(\theta-\eta)\|_{\mathbb H}
\notag\\
&\le
C\|\bm{y}-\bm{\varrho}\|_{\mathbb H}^{1/2}
\|\bm{y}-\bm{\varrho}\|_{\mathbb H^1}^{1/2}
\|\nabla\theta\|_{\mathbb H}^{1/2}
\|{\mathcal{A}}\theta\|_{H}^{1/2}
\notag\\
&\quad
+C\|\bm{\varrho}\|_{\mathbb H}^{1/2}\|A\bm{\varrho}\|_{\mathbb H}^{1/2}
\|\nabla(\theta-\eta)\|_{\mathbb H}.
\label{eq:c_theta}
\end{align}
Using the bounds $\|\bm{y}\|_{D(A)},\; \|\bm{\varrho}\|_{D(A)},\; \|\theta\|_{D(\mathcal{A})},\; \|\eta\|_{D(\mathcal{A})} \le R$ and combining \eqref{eq:lin_u}--\eqref{eq:c_theta}, we have \eqref{eq:Lip_boussinesq}.
\end{proof}

\subsection{Well-posedness for the constrained Boussinesq system}
We prove that the constrained Boussinesq system \eqref{eq:constrained_Boussinesq} has a unique strong solution by using the Faedo--Galerkin approximation method, along with the Banach--Alaoglu theorem, the Aubin--Lions--Simon compactness lemma, and the Lions--Magenes lemma. We begin by introducing the definition of a strong solution to \eqref{eq:constrained_Boussinesq}.
\begin{definition}
Let \((\bm{y}_0,\theta_0) \in (\mathbb{V} \cap \mathcal{M}_1) \times (V \cap \mathcal{M}_2)\) be fixed. A pair
\[
(\bm{y},\theta) \in \big( C([0,T]; \mathbb{V} ) \cap L^2(0,T; D({A})) \big)
\times \big( C([0,T]; V) \cap L^2(0,T; D(\mathcal{A})) \big)
\]
with $\bm{y}(t) \in \mathcal{M}_1$ and $\theta(t) \in \mathcal{M}_2$ for all $t \in [0,T]$,
is called a \emph{strong solution} of the constrained Boussinesq system \eqref{eq:constrained_Boussinesq} if:
\begin{itemize}\setlength\itemsep{2pt}\setlength\parskip{0pt}
\item[(i)] \(\frac{d\bm{y}}{dt} \in {{L}^2(0,T;\mathbb{H})}\), \text{ and } \(\frac{d\theta}{dt} \in L^2(0,T;H)\),

\item[(ii)] For all \((\bm{\phi},\psi) \in L^2(0,T; \mathbb{H} )\times L^2(0,T; H)\),
\[
\int_0^T \left(\tfrac{d\bm{y}(t)}{dt} - F_1(\bm{y}(t),\theta(t)), \bm{\phi}(t) \right)_{\mathbb{H}} \, dt = 0,\quad
\int_0^T \left( \tfrac{d\theta(t)}{dt} - F_2(\bm{y}(t),\theta(t)), \psi(t) \right)_{H} \, dt = 0,
\]
where \(F_1(\bm{y},\theta),\, F_2(\bm{y},\theta)\) are defined in \eqref{eq:F_boussinesq}.

\item[(iii)]The initial conditions hold in the following sense:
\[
\bm{y}(0)=\bm{y}_0 \  \text{in } \ \mathbb{H}, \ \
\theta(0)=\theta_0 \  \text{in }\  H.
\]
\end{itemize}
\end{definition}
We now prove that problem \eqref{eq:constrained_Boussinesq} has a unique strong solution.
\begin{theorem}{\label{well}}
Let \((\bm{y}_0,\theta_0) \in (\mathbb{V} \cap \mathcal{M}_1)\times (V\cap \mathcal{M}_2)\) and \(\mathfrak{H} \in {{L}^2(0,T;\mathbb{H})}\). 
Then the problem admits a unique strong solution
\[
(\bm{y},\theta) \in \big(C([0,T];\mathbb{V})\cap L^2(0,T;D(A))\big)
\times \big(C([0,T];V)\cap L^2(0,T;D(\mathcal{A}))\big),
\]
with
\[
\frac{d\bm{y}}{dt} \in {{L}^2(0,T;\mathbb{H})}, \quad
\frac{d\theta}{dt} \in L^2(0,T;H),
\]
such that
\[
\bm{y}(t)\in \mathcal{M}_1, \quad \theta(t)\in \mathcal{M}_2 \  \text{ for all }\  t\in[0,T].
\]
\end{theorem}
\begin{proof}
\textbf{Existence.}
Consider a complete orthonormal basis \(\{w_1, w_2, \dots\}\) of \(\mathbb{H}\) which is orthogonal in \(\mathbb{V}\); for example, the eigenfunctions of the Stokes operator \(A\). Define the finite-dimensional subspace \(\mathbb{H}_m := \operatorname{span}\{w_1, \dots, w_m\} \subset \mathbb{H}\), yielding the chain \(\mathbb{H}_m \subset \mathbb{H}_{m+1} \subset \mathbb{V}\) and take \(\Pi_m : \mathbb{H} \to \mathbb{H}_m\) to be the orthogonal projection
\[
\Pi_m \bm{x}=\sum_{i=1}^m ( \bm{x},w_i)_{\mathbb{H}} w_i.
\]
Furthermore, we take a complete orthonormal basis \(\{e_1, e_2, \dots\}\) of \(H\) orthogonal in \(V\), for example, the eigenfunctions of the Dirichlet Laplacian operator \(A\), and define $H_m:=\mathrm{span}\{e_1,\dots,e_m\}\subset L^2$ with $H_m\subset H_{m+1}\subset V$, with orthogonal projection $Q_m: H\to H_m$ given by
\[
Q_m\theta=\sum_{i=1}^m ( \theta,e_i)_{H} e_i.
\]
Let \(T>0\) and \(m\in\mathbb{N}\). We then seek approximate solutions
\[
\bm{y}_m(t)=\sum_{k=1}^m g_k^m(t)w_k,\ \
\theta_m(t)=\sum_{k=1}^m h_k^m(t)e_k,\ \ t\in[0,T],
\]
satisfying the initial value problem for the finite-dimensional system given below:
\begin{equation}\label{eq:galerkin_u}
\left\{
\begin{aligned}
\frac{d\bm{y}_m}{dt}
&=-\nu A\bm{y}_m(t)
+\nu\|\nabla\bm{y}_m(t)\|_{\mathbb{H}}^2\bm{y}_m(t)
-\Pi_m B(\bm{y}_m(t),\bm{y}_m(t))
+\mathfrak{H}_m(t)\\
&\quad-(\mathfrak{H}_m(t),\bm{y}_m(t))_{\mathbb{H}}\bm{y}_m(t)
+\beta\Pi_m{{\textbf{Proj}_{\mathbb H}}}(\theta_m(t)e_2)
-\beta(\theta_m(t)e_2,\bm{y}_m(t))_{\mathbb{H}}\bm{y}_m(t),
\quad t\in(0,T),\\
\bm{y}_m(0)
&=\frac{\bm{y}_0^m}{\|\bm{y}_0^m\|_{\mathbb{H}}}.
\end{aligned}
\right.
\end{equation}
and
\begin{equation}\label{eq:galerkin_theta}
\left\{
\begin{aligned}
\frac{d\theta_m}{dt}
&=-\kappa\mathcal{A}\theta_m(t)
+\kappa\|\nabla\theta_m(t)\|_{\mathbb H}^2\theta_m(t)
-Q_m(\bm{y}_m(t)\cdot\nabla\theta_m(t)),
\quad t\in(0,T),\\
\theta_m(0)
&=\frac{\theta_0^m}{\|\theta_0^m\|_{H}}.
\end{aligned}
\right.
\end{equation}
where \(\mathfrak{H}_m=\Pi_m\mathfrak{H}\), \(\bm{y}_0^m=\Pi_m\bm{y}_0\), and
\(\theta_0^m=Q_m\theta_0\), and
\[
g_k^m(0)
=\frac{(\bm{y}_0,w_k)_{\mathbb{H}}}
{\|\Pi_m\bm{y}_0\|_{\mathbb{H}}},
\qquad
h_k^m(0)
=\frac{(\theta_0,e_k)_{H}}
{\|Q_m\theta_0\|_{H}},
\qquad k=1,\dots,m.
\]
Since $\Pi_m \bm{y}_0 \to \bm{y}_0$ in $\mathbb{H}$, we have 
$\|\bm{y}_0^m\|_{\mathbb{H}} \to \|\bm{y}_0\|_{\mathbb{H}} = 1$ as $m \to \infty$. Hence,
\begin{align*}
\|\bm{y}_m(0)-\bm{y}_0\|_{\mathbb{H}}
&\le \left\|\frac{\Pi_m \bm{y}_0}{\|\Pi_m \bm{y}_0\|_{\mathbb{H}}}-\Pi_m \bm{y}_0\right\|_{\mathbb{H}}
+ \|\Pi_m \bm{y}_0-\bm{y}_0\|_{\mathbb{H}}\to 0 \  \text{as }\  m \to \infty.
\end{align*}
Similarly, since $Q_m \theta_0 \to \theta_0$ in ${H}$, we obtain
\begin{align*}
\|\theta_m(0)-\theta_0\|_{{H}}
&\le \left\|\frac{Q_m \theta_0}{\|Q_m \theta_0\|_{{H}}}-Q_m \theta_0\right\|_{{H}}
+ \|Q_m \theta_0-\theta_0\|_{{H}}\to 0 \  \text{ as } \ m \to \infty.
\end{align*}
Since
\[
F_1(\bm{y},\theta):=-\nu A\bm{y}+\nu\|\nabla \bm{y}\|_{\mathbb{H}}^2 \bm{y}-B(\bm{y},\bm{y})+\mathfrak{H}
-(\mathfrak{H},\bm{y})_{\mathbb{H}}\bm{y}+\beta {{\textbf{Proj}_{\mathbb H}}}(\theta e_2)
-\beta(\theta e_2,\bm{y})_{\mathbb{H}}\bm{y},
\]
and
\[
F_2(\bm{y},\theta):=-\kappa \mathcal{A}\theta+\kappa\|\nabla \theta\|_{\mathbb H}^2\theta-\bm{y}\cdot\nabla \theta
\]
are locally Lipschitz (Lemma~\ref{lemLocal}), hence, by Carathéodory's existence theorem, there exists a unique local solution
$(\bm{y}_m,\theta_m)\in C([0,T_m^*];\mathbb{V}_m)\times C([0,T_m^*];V_m),$
to the problem \eqref{eq:galerkin_u}-\eqref{eq:galerkin_theta} for some $0<T_m^*\le T$. Moreover, by deriving uniform energy estimates, one can extend the local solution to the whole interval $[0,T]$.
Taking the inner product in $\mathbb{H}$ with $\bm{y}_m(\cdot)$ in the Galerkin equation \eqref{eq:galerkin_u}, and using $(B(\bm{y}_m,\bm{y}_m),\bm{y}_m)_{\mathbb{H}}=0$, we obtain
\[
\frac{1}{2}\frac{d}{dt}\big(\|\bm{y}_m(t)\|_{\mathbb{H}}^2-1\big)
=\Big[\nu\|\nabla \bm{y}_m(t)\|_{\mathbb{H}}^2
-( \mathfrak{H}_m(t),\bm{y}_m(t))_{\mathbb{H}}
-\beta ( \theta_m(t)e_2,\bm{y}_m(t))_{\mathbb{H}}\Big](\|\bm{y}_m(t)\|_{\mathbb{H}}^2-1),
\]
for a.e. $t\in[0,T]$. Hence, we deduce 
\[
\frac{d}{dt}(\|\bm{y}_m(t)\|_{\mathbb{H}}^2-1)
=2\Big[\nu\|\nabla \bm{y}_m(t)\|_{\mathbb{H}}^2
-( \mathfrak{H}(t),\bm{y}_m(t))_{\mathbb{H}}
-\beta ( \theta_m(t)e_2,\bm{y}_m(t))_{\mathbb{H}}\Big](\|\bm{y}_m(t)\|_{\mathbb{H}}^2-1).
\]
Applying the variation of constants formula, we obtain
\[
\|\bm{y}_m(t)\|_{\mathbb{H}}^2-1
=(\|\bm{y}_m(0)\|_{\mathbb{H}}^2-1)
\exp\!\left(2\int_0^t\Big[\nu\|\nabla \bm{y}_m(s)\|_{\mathbb{H}}^2
-( \mathfrak{H}(s),\bm{y}_m(s))_{\mathbb{H}}
-\beta ( \theta_m(s)e_2,\bm{y}_m(s))_{\mathbb{H}}\Big]ds\right).
\]
\allowdisplaybreaks
Since $\bm{y}_m(0)\in \mathcal{M}_1$, it follows that $\|\bm{y}_m(t)\|_{\mathbb{H}}=1$ for all $t\in[0,T]$, provided
\begin{align*}
&\int_0^t \Big|\nu\|\nabla \bm{y}_m(s)\|_{\mathbb{H}}^2
-( \mathfrak{H}(s),\bm{y}_m(s))_{\mathbb{H}}
-\beta ( \theta_m(s)e_2,\bm{y}_m(s))_{\mathbb{H}}\Big|ds
\\&\le \nu \int_0^t \|\nabla \bm{y}_m(s)\|_{\mathbb{H}}^2 ds 
 + \int_0^t \|\mathfrak{H}(s)\|_{\mathbb{H}}\|\bm{y}_m(s)\|_{\mathbb{H}} ds + \beta \int_0^t \|\theta_m(s)\|_{H}\|\bm{y}_m(s)\|_{\mathbb{H}} ds<\infty,
\end{align*}
for all $t\in[0,T_m^*]$. In a similar manner, we consider the \(L^2\) inner product with $\theta_m(\cdot)$, and using $( \bm{y}_m\cdot\nabla \theta_m,\theta_m)_{H}=0$, we find 
\[
\frac{1}{2}\frac{d}{dt}\big(\|\theta_m(t)\|_{H}^2-1\big)
=\kappa\|\nabla \theta_m(t)\|_{\mathbb H}^2(\|\theta_m(t)\|_{H}^2-1).
\]
Therefore, we get 
\[
\|\theta_m(t)\|_{H}^2-1
=(\|\theta_m(0)\|_{H}^2-1)
\exp\!\left(2\kappa\int_0^t \|\nabla \theta_m(s)\|_{\mathbb H}^2 ds\right).
\]
Since $\theta_m(0)\in\mathcal{M}_2$, we have $\theta_m(t)\in\mathcal{M}_2$ for every $t\in[0,T]$, provided that
\[
\int_0^t \|\nabla \theta_m(s)\|_{\mathbb H}^2 ds<\infty.
\]
Since \(\bm{y}_m \in C([0, T_m^*];\mathbb{V}_m)\) and \(\theta_m \in C([0, T_m^*]; V_m)\), and because on finite-dimensional spaces all the norms are equivalent, the corresponding norms are finite for all \(t \in [0,T_m^*]\). 
The energy estimate below allows us to extend the maximal existence time $T_m^*$ to $T$. Consequently, the following invariance property follows
\[
\bm{y}_m(t) \in \mathcal{M}_1,\ \
\theta_m(t) \in \mathcal{M}_2,\ 
\text{ for all } \ t \in [0,T].
\]
To obtain uniform energy estimates, we first define some energy functionals. We use two functionals instead of one to avoid having $\mathfrak{H}_m$ and $\theta_m$ multiplied together. This helps us overcome some technical difficulties.

Let us  define the following energy functionals:
\[
\mathcal{E}_1(\bm{y}_m)
:=\frac{\nu}{4}\int_{\mathbb{T}^2}|\nabla \bm{y}_m|^2\,dx
-\frac{1}{2}\int_{\mathbb{T}^2}( \mathfrak{H}_m,\bm{y}_m)\,dx,
\]
and
\[
\mathcal{E}_2(\bm{y}_m)
=\frac{\nu}{4}\int_{\mathbb{T}^2}|\nabla \bm{y}_m|^2\,dx
-\frac{\beta}{2}\int_{\mathbb{T}^2} ( \theta_m e_2,\bm{y}_m)\,dx.
\]
Before proceeding with the calculations, set
$\bm b_m:=\Pi_m{\textbf{Proj}}_{\mathbb H}(\theta_m e_2)$. Since ${\textbf{Proj}}_{\mathbb H}$ and $\Pi_m$ are orthogonal projections
and $\Pi_m \bm \varrho=\bm \varrho$ for every $\bm{\varrho}\in\mathbb H_m$, we have
\[
(\theta_m e_2,\bm \varrho)_{\mathbb L^2}
=({\textbf{Proj}}_{\mathbb H}(\theta_m e_2),\bm \varrho)_{\mathbb H}
=(\Pi_m{\textbf{Proj}}_{\mathbb H}(\theta_m e_2),\bm \varrho)_{\mathbb H}
=(\bm b_m,\bm \varrho)_{\mathbb H}.
\]
Then, the constrained gradients on $\mathcal{M}_1$ are given by
\[
\begin{aligned}
\nabla_{\mathcal{M}_1}\mathcal{E}_1(\bm{y}_m)
&= \Pi_{\bm{y}_m}(\nabla \mathcal{E}_1(\bm{y}_m)) \\
&= \frac{\nu}{2}A \bm{y}_m - \frac{\nu}{2}\|\nabla \bm{y}_m\|_{\mathbb{H}}^2 \bm{y}_m - \frac{1}{2}\big(\mathfrak{H}_m - ( \mathfrak{H}_m,\bm{y}_m)_{\mathbb{H}}\bm{y}_m\big),
\end{aligned}
\]
and
\[
\begin{aligned}
\nabla_{\mathcal{M}_1}\mathcal{E}_2(\bm{y}_m)
&= \Pi_{\bm{y}_m}(\nabla \mathcal{E}_2(\bm{y}_m)) \\
&= \frac{\nu}{2}A \bm{y}_m - \frac{\nu}{2}\|\nabla \bm{y}_m\|_{\mathbb{H}}^2 \bm{y}_m- \frac{\beta}{2}\big(\Pi_m{{\textbf{Proj}_{\mathbb H}}}\theta_m e_2 - ( \theta_m e_2,\bm{y}_m)_{\mathbb{H}}\bm{y}_m\big)\\
&=\frac{\nu}{2}A\bm{y}_m
-\frac{\nu}{2}\|\nabla\bm{y}_m\|_{\mathbb H}^2\bm{y}_m
-\frac{\beta}{2}
\big(\bm b_m-(\bm b_m,\bm{y}_m)_{\mathbb H}\bm{y}_m\big).
\end{aligned}
\]
\allowdisplaybreaks
Moreover, for $\bm{y}_m \in D(A)\cap \mathcal{M}_1$ (since $\bm{y}_m$ is finite-dimensional), we compute
\begin{align}{\label{e1}}
\|\nabla_{\mathcal{M}_1}\mathcal{E}_1(\bm{y}_m)\|_{\mathbb{H}}^2
&= \Big\| \frac{\nu}{2}A \bm{y}_m - \frac{\nu}{2}\|\nabla \bm{y}_m\|_{\mathbb{H}}^2 \bm{y}_m
- \frac{1}{2}\big(\mathfrak{H}_m - ( \mathfrak{H}_m,\bm{y}_m)_{\mathbb{H}}\bm{y}_m\big) \Big\|_{\mathbb{H}}^2 \notag \\
&= \frac{\nu^2}{4}\|A \bm{y}_m - \|\nabla \bm{y}_m\|_{\mathbb{H}}^2 \bm{y}_m\|_{\mathbb{H}}^2
+ \frac{1}{4}\|\mathfrak{H}_m - ( \mathfrak{H}_m,\bm{y}_m)_{\mathbb{H}}\bm{y}_m\|_{\mathbb{H}}^2 \notag \\
&\quad - \frac{\nu}{2}\Big( A \bm{y}_m - \|\nabla \bm{y}_m\|_{\mathbb{H}}^2 \bm{y}_m,\,
\mathfrak{H}_m - ( \mathfrak{H}_m,\bm{y}_m)_{\mathbb{H}}\bm{y}_m \Big)_{\mathbb{H}} \notag \\
&= \frac{\nu^2}{4}\Big[\|A \bm{y}_m\|_{\mathbb{H}}^2
+ \|\nabla \bm{y}_m\|_{\mathbb{H}}^4 \|\bm{y}_m\|_{\mathbb{H}}^2
- 2\|\nabla \bm{y}_m\|_{\mathbb{H}}^2 ( A \bm{y}_m,\bm{y}_m)_{\mathbb{H}}\Big] \notag \\
&\quad + \frac{1}{4}\Big[\|\mathfrak{H}_m\|_{\mathbb{H}}^2
+ |( \mathfrak{H}_m,\bm{y}_m)_{\mathbb{H}}|^2 \|\bm{y}_m\|_{\mathbb{H}}^2
- 2|( \mathfrak{H}_m,\bm{y}_m)_{\mathbb{H}}|^2\Big] \notag \\
&\quad - \frac{\nu}{2}\Big[( A \bm{y}_m,\mathfrak{H}_m)_{\mathbb{H}}
- ( \mathfrak{H}_m,\bm{y}_m)_{\mathbb{H}} ( A \bm{y}_m,\bm{y}_m)_{\mathbb{H}} \notag \\
&\qquad\qquad - \|\nabla \bm{y}_m\|_{\mathbb{H}}^2 ( \bm{y}_m,\mathfrak{H}_m)_{\mathbb{H}}
+ \|\nabla \bm{y}_m\|_{\mathbb{H}}^2 ( \mathfrak{H}_m,\bm{y}_m)_{\mathbb{H}}\|\bm{y}_m\|_{\mathbb{H}}^2 \Big] \notag \\
&= \frac{\nu^2}{4}\Big[\|A \bm{y}_m\|_{\mathbb{H}}^2 - \|\nabla \bm{y}_m\|_{\mathbb{H}}^4\Big]
+ \frac{1}{4}\Big[\|\mathfrak{H}_m\|_{\mathbb{H}}^2 - |( \mathfrak{H}_m,\bm{y}_m)_{\mathbb{H}}|^2\Big] \notag \\
&\quad - \frac{\nu}{2}\Big[( A \bm{y}_m,\mathfrak{H}_m)_{\mathbb{H}}
- ( \mathfrak{H}_m,\bm{y}_m)_{\mathbb{H}}\|\nabla \bm{y}_m\|_{\mathbb{H}}^2 \Big].
\end{align}
Similarly, we compute
\begin{align}\label{e2}
\|\nabla_{\mathcal M_{1}}\mathcal E_2(\bm{y}_m)\|_{\mathbb H}^2
&= \Big\|
\frac{\nu}{2}A\bm{y}_m
-\frac{\nu}{2}\|\nabla\bm{y}_m\|_{\mathbb H}^2\bm{y}_m
-\frac{\beta}{2}
\big(\bm b_m-(\bm b_m,\bm{y}_m)_{\mathbb H}\bm{y}_m\big)
\Big\|_{\mathbb H}^2 \notag\\[6pt]
&= \frac{\nu^2}{4}
\Big\|A\bm{y}_m-\|\nabla\bm{y}_m\|_{\mathbb H}^2\bm{y}_m
\Big\|_{\mathbb H}^2
+\frac{\beta^2}{4}
\Big\|\bm b_m-(\bm b_m,\bm{y}_m)_{\mathbb H}\bm{y}_m
\Big\|_{\mathbb H}^2 \notag\\[6pt]
&\quad-\frac{\nu\beta}{2}
\Big(
A\bm{y}_m-\|\nabla\bm{y}_m\|_{\mathbb H}^2\bm{y}_m,\,
\bm b_m-(\bm b_m,\bm{y}_m)_{\mathbb H}\bm{y}_m
\Big)_{\mathbb H} \notag\\[6pt]
&= \frac{\nu^2}{4}\Big[
\|A\bm{y}_m\|_{\mathbb H}^2
+\|\nabla\bm{y}_m\|_{\mathbb H}^4\|\bm{y}_m\|_{\mathbb H}^2
-2\|\nabla\bm{y}_m\|_{\mathbb H}^2
(A\bm{y}_m,\bm{y}_m)_{\mathbb H}
\Big] \notag\\[6pt]
&\quad+\frac{\beta^2}{4}\Big[
\|\bm b_m\|_{\mathbb H}^2
+|(\bm b_m,\bm{y}_m)_{\mathbb H}|^2\|\bm{y}_m\|_{\mathbb H}^2
-2|(\bm b_m,\bm{y}_m)_{\mathbb H}|^2
\Big] \notag\\[6pt]
&\quad-\frac{\nu\beta}{2}\Big[
(A\bm{y}_m,\bm b_m)_{\mathbb H}
-(\bm b_m,\bm{y}_m)_{\mathbb H}
(A\bm{y}_m,\bm{y}_m)_{\mathbb H} \notag\\
&\qquad\qquad\quad
-\|\nabla\bm{y}_m\|_{\mathbb H}^2
(\bm{y}_m,\bm b_m)_{\mathbb H}
+\|\nabla\bm{y}_m\|_{\mathbb H}^2
(\bm b_m,\bm{y}_m)_{\mathbb H}\|\bm{y}_m\|_{\mathbb H}^2
\Big] \notag\\[6pt]
&= \frac{\nu^2}{4}\Big[
\|A\bm{y}_m\|_{\mathbb H}^2
-\|\nabla\bm{y}_m\|_{\mathbb H}^4
\Big]
+\frac{\beta^2}{4}\Big[
\|\bm b_m\|_{\mathbb H}^2
-|(\theta_m e_2,\bm{y}_m)_{\mathbb{H}}|^2
\Big] \notag\\[6pt]
&\quad-\frac{\nu\beta}{2}\Big[
(A\bm{y}_m,\theta_m e_2)_{\mathbb{H}}
-(\theta_m e_2,\bm{y}_m)_{\mathbb{H}}
\|\nabla\bm{y}_m\|_{\mathbb H}^2
\Big].
\end{align}
Taking the inner product in \(\mathbb{H}\) of \eqref{eq:galerkin_u} with \(\frac{\nu}{2}A \bm{y}_m(\cdot)\), 
and using $(\Pi_mB(\bm{y}_m,\bm{y}_m),A\bm{y}_m)_{\mathbb H}
=(B(\bm{y}_m,\bm{y}_m),A\bm{y}_m)_{\mathbb H}=0$, 
 together with the facts that \(\bm{y}_m(t)\in \mathcal{M}_1\) and \(\theta_m(t)\in \mathcal{M}_2\) for all \(t\in[0,T]\), we obtain
\begin{align}
\frac{\nu}{4}\frac{d}{dt}\|\nabla\bm{y}_m(t)\|_{\mathbb H}^2
&=\Big(\frac{\nu}{2}A\bm{y}_m(t),
\frac{d\bm{y}_m(t)}{dt}\Big)_{\mathbb H} \notag\\[6pt]
&=\Big(\frac{\nu}{2}A\bm{y}_m(t),
-\nu A\bm{y}_m(t)
+\nu\|\nabla\bm{y}_m(t)\|_{\mathbb H}^2\bm{y}_m(t)
-\Pi_mB(\bm{y}_m(t),\bm{y}_m(t))\Big)_{\mathbb H}
\notag\\[6pt]
&\quad+\Big(\frac{\nu}{2}A\bm{y}_m(t),
\mathfrak{H}_m(t)-(\mathfrak{H}_m(t),\bm{y}_m(t))_{\mathbb H}
\bm{y}_m(t)\Big)_{\mathbb H} \notag\\[6pt]
&\quad+\Big(\frac{\nu}{2}A\bm{y}_m(t),
\beta\bm b_m(t)-\beta(\bm b_m(t),\bm{y}_m(t))_{\mathbb H}
\bm{y}_m(t)\Big)_{\mathbb H} \notag\\[6pt]
&=-\frac{\nu^2}{2}\|A\bm{y}_m(t)\|_{\mathbb H}^2
+\frac{\nu^2}{2}\|\nabla\bm{y}_m(t)\|_{\mathbb H}^4
\notag\\[6pt]
&\quad+\frac{\nu}{2}(A\bm{y}_m(t),\mathfrak{H}_m(t))_{\mathbb H}
-\frac{\nu}{2}(\mathfrak{H}_m(t),\bm{y}_m(t))_{\mathbb H}
\|\nabla\bm{y}_m(t)\|_{\mathbb H}^2 \notag\\[6pt]
&\quad+\frac{\nu\beta}{2}(A\bm{y}_m(t),\bm b_m(t))_{\mathbb H}
-\frac{\nu\beta}{2}(\bm b_m(t),\bm{y}_m(t))_{\mathbb H}
\|\nabla\bm{y}_m(t)\|_{\mathbb H}^2 \notag\\[6pt]
&=-\|\nabla_{\mathcal M_{1}}\mathcal E_1(\bm{y}_m(t))\|_{\mathbb H}^2
-\|\nabla_{\mathcal M_{1}}\mathcal E_2(\bm{y}_m(t))\|_{\mathbb H}^2
\notag\\[6pt]
&\quad+\frac14\Big(
\|\mathfrak{H}_m(t)\|_{\mathbb H}^2
-|(\mathfrak{H}_m(t),\bm{y}_m(t))_{\mathbb H}|^2
\Big)+\frac{\beta^2}{4}\Big(
\|\bm b_m(t)\|_{\mathbb H}^2
-|(\bm b_m(t),\bm{y}_m(t))_{\mathbb H}|^2
\Big) \notag\\[6pt]
&\le
-\|\nabla_{\mathcal M_{1}}\mathcal E_1(\bm{y}_m(t))\|_{\mathbb H}^2
-\|\nabla_{\mathcal M_{1}}\mathcal E_2(\bm{y}_m(t))\|_{\mathbb H}^2
+\frac14\|\mathfrak{H}_m(t)\|_{\mathbb H}^2
+\frac{\beta^2}{4},
\label{eq:inner}
\end{align}
for a.e.\ $t\in(0,T)$, where
$\bm b_m(t):=\Pi_m{\textbf{Proj}}_{\mathbb H}(\theta_m(t)e_2)$.
In the last inequality, we have used
$\|\bm b_m(t)\|_{\mathbb H}\le\|\theta_m(t)\|_H=1$.
Integrating \eqref{eq:inner} over $(0,t)$, we obtain
\begin{align}\label{eq:finll}
\frac{\nu}{4}\|\nabla \bm{y}_m(t)\|_{\mathbb{H}}^2
&\;+\;\int_0^t \Big(
\|\nabla_{\mathcal{M}_1}\mathcal{E}_1(\bm{y}_m(s))\|_{\mathbb{H}}^2
+\|\nabla_{\mathcal{M}_1}\mathcal{E}_2(\bm{y}_m(s))\|_{\mathbb{H}}^2
\Big)\,ds \notag \\
&\le \frac{\nu}{4}\|\nabla \bm{y}_m(0)\|_{\mathbb{H}}^2
+ \frac{1}{4}\int_0^t \|\mathfrak{H}_m(s)\|_{\mathbb{H}}^2\,ds
+ \frac{\beta^2}{4}\, t,
\end{align}
for all $t\in[0,T]$. From \eqref{eq:finll}, we immediately obtain
\begin{align}
\sup_{0\le t\le T}\|\nabla \bm{y}_m(t)\|_{\mathbb{H}}^{2}
&\le
C_{\nu,\beta,T}
\left(
1+\|\nabla \bm{y}_0\|_{\mathbb{H}}^{2}
+\int_{0}^{T}\|\mathfrak{H}(s)\|_{\mathbb{H}}^{2}\,ds
\right),
\label{bd}
\end{align}
where \(C_{\nu,\beta,T}>0\) is a constant whose value depends on \(\nu\), \(\beta\), and \(T\).
Moreover, again by \eqref{eq:finll}, we find 
\begin{align}
\int_{0}^{T}
\Bigl(
\|\nabla_{\mathcal{M}_{1}}\mathcal{E}_{1}(\bm{y}_m(s))\|_{\mathbb{H}}^{2}
+
\|\nabla_{\mathcal{M}_{1}}\mathcal{E}_{2}(\bm{y}_m(s))\|_{\mathbb{H}}^{2}
\Bigr)\,ds
\le
C_{\nu,\beta,T}
\left(
1+\|\nabla \bm{y}_0\|_{\mathbb{H}}^{2}
+\int_{0}^{T}\|\mathfrak{H}(s)\|_{\mathbb{H}}^{2}\,ds
\right).
\label{eq:energy-grad}
\end{align}
Summing \eqref{e1} and \eqref{e2}, we obtain
\begin{align}{\label{e1e2}}
\|\nabla_{\mathcal M_{1}}\mathcal E_1(\bm{y}_m)\|_{\mathbb H}^2
+\|\nabla_{\mathcal M_{1}}\mathcal E_2(\bm{y}_m)\|_{\mathbb H}^2=&\frac{\nu^2}{2}
\Big(\|A\bm{y}_m\|_{\mathbb H}^2
-\|\nabla\bm{y}_m\|_{\mathbb H}^4\Big)
+\frac14\Big(
\|\mathfrak{H}_m\|_{\mathbb H}^2
-|(\mathfrak{H}_m,\bm{y}_m)_{\mathbb H}|^2
\Big)\nonumber\\&+\frac{\beta^2}{4}\Big(
\|\bm b_m\|_{\mathbb H}^2
-|(\bm b_m,\bm{y}_m)_{\mathbb H}|^2
\Big)
-\frac{\nu}{2}(A\bm{y}_m,\mathfrak{H}_m)_{\mathbb H}
-\frac{\nu\beta}{2}(A\bm{y}_m,\bm b_m)_{\mathbb H}\nonumber\\
&+\frac{\nu}{2}(\mathfrak{H}_m,\bm{y}_m)_{\mathbb H}
\|\nabla\bm{y}_m\|_{\mathbb H}^2
+\frac{\nu\beta}{2}(\bm b_m,\bm{y}_m)_{\mathbb H}
\|\nabla\bm{y}_m\|_{\mathbb H}^2.
\end{align}
Rearranging \eqref{e1e2}, we deduce
\begin{align}\label{eqn-est-1}
\frac{\nu^2}{2}\|A\bm{y}_m\|_{\mathbb H}^2
&=\|\nabla_{\mathcal M_{1}}\mathcal E_1(\bm{y}_m)\|_{\mathbb H}^2
+\|\nabla_{\mathcal M_{1}}\mathcal E_2(\bm{y}_m)\|_{\mathbb H}^2
+\frac{\nu^2}{2}\|\nabla\bm{y}_m\|_{\mathbb H}^4
\notag\\
&\quad-\frac14\|\mathfrak{H}_m\|_{\mathbb H}^2
-\frac{\beta^2}{4}\|\bm b_m\|_{\mathbb H}^2
+\frac14|(\mathfrak{H}_m,\bm{y}_m)_{\mathbb H}|^2
+\frac{\beta^2}{4}|(\bm b_m,\bm{y}_m)_{\mathbb H}|^2
\notag\\
&\quad+\frac{\nu}{2}(A\bm{y}_m,\mathfrak{H}_m)_{\mathbb H}
+\frac{\nu\beta}{2}(A\bm{y}_m,\bm b_m)_{\mathbb H}
\notag\\
&\quad-\frac{\nu}{2}(\mathfrak{H}_m,\bm{y}_m)_{\mathbb H}
\|\nabla\bm{y}_m\|_{\mathbb H}^2
-\frac{\nu\beta}{2}(\bm b_m,\bm{y}_m)_{\mathbb H}
\|\nabla\bm{y}_m\|_{\mathbb H}^2.
\end{align}
By Cauchy--Schwarz inequality and the facts that 
\(\|\bm{y}_m\|_{\mathbb{H}}=1\) and \(\|\theta_m\|_{\mathbb{H}}=1\), we estimate
\[
|( \mathfrak{H}_m,\bm{y}_m)_{\mathbb{H}}|
\le \|\mathfrak{H}_m\|_{\mathbb{H}},
\ \
|( \theta_m e_2,\bm{y}_m)_{\mathbb{H}}|
\le 1.
\]
Hence, it is immediate from \eqref{eqn-est-1} that
\begin{align}
\frac{\nu^2}{2}\|A\bm{y}_m\|_{\mathbb H}^2
&\le
\|\nabla_{\mathcal M_{1}}\mathcal E_1(\bm{y}_m)\|_{\mathbb H}^2
+\|\nabla_{\mathcal M_{1}}\mathcal E_2(\bm{y}_m)\|_{\mathbb H}^2
+\frac{\nu^2}{2}\|\nabla\bm{y}_m\|_{\mathbb H}^4
\notag\\
&\quad+\frac{\nu}{2}
|(A\bm{y}_m,\mathfrak{H}_m)_{\mathbb H}|
+\frac{\nu\beta}{2}
|(A\bm{y}_m,\bm b_m)_{\mathbb H}|
\notag\\
&\quad+\frac{\nu}{2}\|\mathfrak{H}_m\|_{\mathbb H}
\|\nabla\bm{y}_m\|_{\mathbb H}^2
+\frac{\nu\beta}{2}\|\nabla\bm{y}_m\|_{\mathbb H}^2.
\label{3.4}
\end{align}
Applying Young's inequality, we obtain
\begin{align}
\frac{\nu}{2}|(A\bm{y}_m,\mathfrak{H}_m)_{\mathbb H}|
&\le\frac{\nu^2}{8}\|A\bm{y}_m\|_{\mathbb H}^2
+\frac12\|\mathfrak{H}_m\|_{\mathbb H}^2,\\
\frac{\nu\beta}{2}|(A\bm{y}_m,\bm b_m)_{\mathbb H}|
&\le\frac{\nu^2}{8}\|A\bm{y}_m\|_{\mathbb H}^2
+\frac{\beta^2}{2}\|\bm b_m\|_{\mathbb H}^2
\le\frac{\nu^2}{8}\|A\bm{y}_m\|_{\mathbb H}^2
+\frac{\beta^2}{2},\\
\frac{\nu}{2}\|\mathfrak{H}_m\|_{\mathbb H}
\|\nabla\bm{y}_m\|_{\mathbb H}^2
&\le\frac14\|\mathfrak{H}_m\|_{\mathbb H}^2
+\frac{\nu^2}{4}\|\nabla\bm{y}_m\|_{\mathbb H}^4,\\
\frac{\nu\beta}{2}\|\nabla\bm{y}_m\|_{\mathbb H}^2
&\le\frac{\nu^2}{4}\|\nabla\bm{y}_m\|_{\mathbb H}^4
+\frac{\beta^2}{4}.
\end{align}
Substituting these estimates into \eqref{3.4} and
absorbing the terms involving $\|A\bm{y}_m\|_{\mathbb H}^2$
into the left-hand side, we obtain
\begin{align}
\frac{\nu^2}{4}\|A\bm{y}_m\|_{\mathbb H}^2
&\le
\|\nabla_{\mathcal M_{1}}\mathcal E_1(\bm{y}_m)\|_{\mathbb H}^2
+\|\nabla_{\mathcal M_{1}}\mathcal E_2(\bm{y}_m)\|_{\mathbb H}^2+\nu^2\|\nabla\bm{y}_m\|_{\mathbb H}^4
+\frac34\|\mathfrak{H}_m\|_{\mathbb H}^2
+\frac{3\beta^2}{4}.
\label{Auest}
\end{align}
Integrating \eqref{Auest} over \((0,T)\) gives
\begin{align}
\frac{\nu^2}{4}\int_0^T\|A\bm{y}_m\|_{\mathbb H}^2\,dt
&\le
\int_0^T
\left(
\|\nabla_{\mathcal M_1}\mathcal{E}_1(\bm{y}_m)\|_{\mathbb H}^2
+
\|\nabla_{\mathcal M_1}\mathcal{E}_2(\bm{y}_m)\|_{\mathbb H}^2
\right)dt \notag\\
&\quad
+\nu^2\int_0^T\|\nabla\bm{y}_m\|_{\mathbb H}^4\,dt
+\frac{3}{4}\int_0^T\|\mathfrak{H}\|_{\mathbb H}^2\,dt
+\frac{3\beta^2 T}{4}.
\label{3.6}
\end{align}
Hence, using
\eqref{bd} and \eqref{eq:energy-grad}, from \eqref{3.6} we get
\begin{align}
\nu^2\int_0^T\|A\bm{y}_m(t)\|_{\mathbb H}^2\,dt
&\le
4C_{\nu,\beta,T}
\left(
1+\|\nabla\bm{y}_0\|_{\mathbb H}^2
+\int_0^T\|\mathfrak{H}(t)\|_{\mathbb H}^2\,dt
\right) \notag\\
&\quad+4\nu^2T C_{\nu,\beta,T}^2
\left(
1+\|\nabla\bm{y}_0\|_{\mathbb H}^2
+\int_0^T\|\mathfrak{H}(t)\|_{\mathbb H}^2\,dt
\right)^2 \notag\\
&\quad+3\int_0^T\|\mathfrak{H}(t)\|_{\mathbb H}^2\,dt
+3\beta^2T \notag\\
&\le C_{\nu,\beta,T}
\left(
1+\|\nabla\bm{y}_0\|_{\mathbb H}^2
+\int_0^T\|\mathfrak{H}(t)\|_{\mathbb H}^2\,dt
\right)^2.
\label{aum-est}
\end{align}
Here, the constant $C_{\nu,\beta,T}$ may vary from line to line and is independent of $m$.
Using the fact that $\bm{y}_m\in \mathcal{M}_1$ for all $t\in[0,T]$, we have
\[
\left(\frac{d\bm{y}_m(t)}{dt},\bm{y}_m(t)\right)_{\mathbb{H}}=0
\quad \text{for a.e. } t\in(0,T).
\]
From \eqref{eq:galerkin_u}, we obtain
\begin{align*}
\frac{\nu}{2}\frac{d}{dt}
\|\nabla\bm{y}_m(t)\|_{\mathbb H}^2
&=
\left(
\nu A\bm{y}_m(t),\frac{d\bm{y}_m(t)}{dt}
\right)_{\mathbb H}\\
&=
\left(
\frac{d\bm{y}_m(t)}{dt},
-\frac{d\bm{y}_m(t)}{dt}
+\nu\|\nabla\bm{y}_m(t)\|_{\mathbb H}^2\bm{y}_m(t)
-\Pi_mB(\bm{y}_m(t),\bm{y}_m(t))
\right)_{\mathbb H}\\
&\quad+
\left(
\frac{d\bm{y}_m(t)}{dt},
\mathfrak{H}_m(t)-(\mathfrak{H}_m(t),\bm{y}_m(t))_{\mathbb H}\bm{y}_m(t)
\right)_{\mathbb H}\\
&\quad+\beta
\left(
\frac{d\bm{y}_m(t)}{dt},
\bm b_m(t)-(\bm b_m(t),\bm{y}_m(t))_{\mathbb H}\bm{y}_m(t)
\right)_{\mathbb H}.
\end{align*}
Thus, by Young's inequality, we write
\begin{align}{\label{fracnab}}
\frac{\nu}{2}\frac{d}{dt}
\|\nabla\bm{y}_m(t)\|_{\mathbb H}^2
&=
-\left\|\frac{d\bm{y}_m(t)}{dt}\right\|_{\mathbb H}^2
-\left(
\frac{d\bm{y}_m(t)}{dt},
B(\bm{y}_m(t),\bm{y}_m(t))
\right)_{\mathbb H}\nonumber\\
&\quad+
\left(
\frac{d\bm{y}_m(t)}{dt},\mathfrak{H}_m(t)
\right)_{\mathbb H}
+\beta\left(
\frac{d\bm{y}_m(t)}{dt},\bm b_m(t)
\right)_{\mathbb H}\nonumber\\
&\le
-\frac12\left\|\frac{d\bm{y}_m(t)}{dt}\right\|_{\mathbb H}^2
+\frac32\|B(\bm{y}_m(t),\bm{y}_m(t))\|_{\mathbb H}^2
+\frac32\|\mathfrak{H}_m(t)\|_{\mathbb H}^2
+\frac32\beta^2.
\end{align}
Using Hölder's inequality, Agmon's inequality in two dimensions, and Young's inequality, we have
\begin{align}{\label{Besst}}
\|B(\bm{y}_m,\bm{y}_m)\|_{\mathbb H}^{2}
&\le
\|\bm{y}_m\|_{\mathbb L^\infty}^{2}
\|\nabla\bm{y}_m\|_{\mathbb H}^{2} \nonumber\\
&\le
C
\|\bm{y}_m\|_{\mathbb H}
\|A\bm{y}_m\|_{\mathbb H}
\|\nabla\bm{y}_m\|_{\mathbb H}^{2} \nonumber\\
&\le
\frac14
\|A\bm{y}_m\|_{\mathbb H}^{2}
+
C
\|\nabla\bm{y}_m\|_{\mathbb H}^{4}.
\end{align}
Combining \eqref{fracnab} and \eqref{Besst}, we deduce
\begin{align}{\label{nab_du}}
\frac{\nu}{2}\frac{d}{dt}\|\nabla\bm{y}_m(t)\|_{\mathbb H}^{2}
+\frac12
\left\|
\frac{d\bm{y}_m}{dt}
\right\|_{\mathbb H}^{2}
&\le
\frac C4
\|A\bm{y}_m\|_{\mathbb H}^{2}
+C\|\nabla\bm{y}_m\|_{\mathbb H}^{4}
+C\|\mathfrak{H}_m\|_{\mathbb H}^{2}
+C\beta^{2}.
\end{align}
for a.e.\ $t\in[0,T]$.
Integrating \eqref{nab_du} over $[0,T]$, and using \eqref{bd},
\eqref{aum-est}, we have
\begin{align}
\int_0^T
\left\|\frac{d\bm{y}_m(t)}{dt}\right\|_{\mathbb H}^{2}\,dt
&\le C_{\nu,\beta,T}
\left(
1+\|\nabla\bm{y}_0\|_{\mathbb H}^{2}
+\int_0^T\|\mathfrak{H}(t)\|_{\mathbb H}^{2}\,dt
\right)^2.
\label{ut-est}
\end{align}
We now derive similar estimates for $\theta_m$. In order to do this, we consider the Galerkin approximation $\theta_m(\cdot)$ which satisfies
\begin{align}{\label{the_m}}
\frac{d\theta_m(t)}{dt}
-\kappa\Delta\theta_m(t)
+Q_m\bigl(\bm{y}_m(t)\cdot\nabla\theta_m(t)\bigr)
=\kappa\|\nabla\theta_m(t)\|_{\mathbb H}^2\theta_m(t),
\end{align}
for a.e.\ $t\in[0,T]$.
Since $\theta_m(t)\in\mathcal M_2$, we have
\[
\frac{d}{dt}\|\theta_m(t)\|_H^2
=2\left(\frac{d\theta_m(t)}{dt},\theta_m(t)\right)_H=0.
\]
Moreover, since $Q_m$ is an orthogonal projection and
$\frac{d\theta_m(t)}{dt}\in H_m$,
\[
\left(
\frac{d\theta_m(t)}{dt},
Q_m(\bm{y}_m(t)\cdot\nabla\theta_m(t))
\right)_H
=
\left(
\frac{d\theta_m(t)}{dt},
\bm{y}_m(t)\cdot\nabla\theta_m(t)
\right)_H.
\]
We therefore obtain
\[
\frac{\kappa}{2}\frac{d}{dt}
\|\nabla\theta_m(t)\|_{\mathbb H}^2
=
-\left\|\frac{d\theta_m(t)}{dt}\right\|_H^2
-\left(
\frac{d\theta_m(t)}{dt},
\bm{y}_m(t)\cdot\nabla\theta_m(t)
\right)_H.
\]
By Young's inequality,
\[
\frac{\kappa}{2}\frac{d}{dt}
\|\nabla\theta_m(t)\|_{\mathbb H}^2
+\frac12\left\|\frac{d\theta_m(t)}{dt}\right\|_H^2
\le
\frac12\|\bm{y}_m(t)\cdot\nabla\theta_m(t)\|_H^2.
\]
Applying the H\"older's inequality, we observe
\begin{align}{\label{thett}}
\kappa\|\nabla\theta_m(t)\|_{\mathbb H}^2
+\int_0^t\left\|\frac{d\theta_m(s)}{ds}\right\|_H^2\,ds
\le
\kappa\|\nabla\theta_m(0)\|_{H}^2
+\int_0^t
\|\bm{y}_m(s)\|_{\mathbb L^\infty}^2
\|\nabla\theta_m(s)\|_{\mathbb H}^2\,ds,
\end{align}
for every $t\in[0,T]$.
By Agmon's inequality in two dimensions, we have 
\begin{align}{\label{agmon}}
\|\bm{y}_m\|_{\mathbb L^\infty}^2
\le
C\|\bm{y}_m\|_{\mathbb H}\|A\bm{y}_m\|_{\mathbb H}.
\end{align}
Substituting \eqref{agmon} into \eqref{thett} yields
\[
\kappa\|\nabla \theta_m(t)\|_{\mathbb H}^2
+ \int_0^t \left\| \frac{d\theta_m(s)}{ds} \right\|_{H}^2 ds
\le
\kappa\|\nabla \theta_m(0)\|_{\mathbb H}^2
+ C\int_0^t \|\bm{y}_m(s)\|_{\mathbb H}\|A\bm{y}_m(s)\|_{\mathbb H}\|\nabla \theta_m(s)\|_{\mathbb H}^2 ds.
\]
Applying Gr\"onwall's and Young's inequalities, and using
\eqref{aum-est}, we obtain
\begin{align}\label{the_st}
&\sup_{0\leq t\leq T}\|\nabla\theta_m(t)\|_{\mathbb H}^2
+\int_0^T\left\|\frac{d\theta_m(s)}{ds}\right\|_{ H}^2\,ds
\nonumber\\
&\quad\leq
C_\kappa\|\nabla\theta_m(0)\|_{\mathbb H}^2
\exp\left(
\frac{C}{\kappa}\int_0^T
\|\bm{y}_m(s)\|_{\mathbb H}\|A\bm{y}_m(s)\|_{\mathbb H}\,ds
\right)
\nonumber\\
&\quad\leq
C_\kappa\|\nabla\theta_0\|_{\mathbb H}^2
\exp\left[
C_{\nu,\beta,\kappa,T}
\left(
1+\|\nabla \bm{y}_0\|_{\mathbb H}^2
+\int_0^T\|\mathfrak{H}(s)\|_{\mathbb H}^2\,ds
\right)
\right].
\end{align}
Taking the inner product in $H$ of \eqref{the_m} with
$\kappa\mathcal A\theta_m(t)$, we find
\begin{align*}
\frac{\kappa}{2} \frac{d}{dt} \|\nabla \theta_m(t)\|_{\mathbb H}^2
&= \left( \kappa {\mathcal{A}}\theta_m(t), \frac{d\theta_m(t)}{dt} \right)_{H}\nonumber\\
&= \left( \kappa {\mathcal{A}}\theta_m(t), -\kappa {\mathcal{A}}\theta_m(t)
+ \kappa \|\nabla \theta_m(t)\|_{\mathbb H}^2 \theta_m(t)
- Q_m(\bm{y}_m(t)\cdot \nabla \theta_m(t) )\right)_{H}
\nonumber\\
&= -\kappa^2 \|{\mathcal{A}}\theta_m(t)\|_{H}^2
+ \kappa^2 \|\nabla \theta_m(t)\|_{\mathbb H}^4
- \kappa ( {\mathcal{A}}\theta_m(t), \bm{y}_m(t)\cdot \nabla \theta_m(t) )_{H},
\end{align*}
for a.e. $t\in[0,T].$
\begin{remark}
On $\mathbb T^2$, the cancellation
$(B(\bm{y},\bm{y}),A\bm{y})_{\mathbb H}=0$ holds for
$\bm{y}\in D(A)$, but it need not hold on bounded domains.
Moreover, the temperature convection term
$(\bm{y}\cdot\nabla\theta,\mathcal A\theta)_H$
does not vanish in general, even on $\mathbb T^2$.
We estimate this term as follows.
\end{remark}
Applying H\"older's and Agmon's inequalities, and then Young's inequality, we get
\begin{align*}
&\big|\kappa ( {\mathcal{A}}\theta_m, \bm{y}_m \cdot \nabla \theta_m )_{H}\big|
\le \kappa \|{\mathcal{A}}\theta_m\|_{H} \|\bm{y}_m\|_{L^\infty} \|\nabla \theta_m\|_{\mathbb H}\nonumber\\
&\le \frac{\kappa^2}{2} \|{\mathcal{A}}\theta_m\|_{H}^2
+ C \|A\bm{y}_m\|_{\mathbb H} \|\bm{y}_m\|_{\mathbb H} \|\nabla \theta_m\|_{\mathbb H}^2.
\end{align*}
Therefore, for all \(t\in[0,T]\), we obtain 
\begin{align}
&\kappa\|\nabla\theta_m(t)\|_{\mathbb H}^{2}
+\kappa^{2}\int_{0}^{t}
\|{\mathcal{A}}\theta_m(s)\|_{H}^{2}\,ds
\nonumber\\&\le
\kappa\|\nabla\theta_m(0)\|_{\mathbb H}^{2}
+2\kappa^{2}\int_{0}^{t}
\|\nabla\theta_m(s)\|_{\mathbb H}^{4}\,ds
+C\int_{0}^{t}
\|A\bm{y}_m(s)\|_{\mathbb H}
\|\bm{y}_m(s)\|_{\mathbb H}
\|\nabla\theta_m(s)\|_{\mathbb H}^{2}\,ds.
\label{theta-A}
\end{align}
Using H\"older's inequality together with $\bm{y}_m\in\mathcal{M}_1$, we get
\begin{align}{\label{hold_um}}
\int_{0}^{t}
\|A\bm{y}_m(s)\|_{\mathbb H}
\|\bm{y}_m(s)\|_{\mathbb H}
\|\nabla\theta_m(s)\|_{\mathbb H}^{2}\,ds
&\le
\left(
\int_{0}^{t}
\|A\bm{y}_m(s)\|_{\mathbb H}^{2}\,ds
\right)^{\frac12}
\left(
\int_{0}^{t}
\|\nabla\theta_m(s)\|_{\mathbb H}^{4}\,ds
\right)^{\frac12}.
\end{align}
Also from \eqref{the_st}, we have 
\begin{align}
\int_{0}^{T}
\|\nabla\theta_m(s)\|_{\mathbb H}^{4}\,ds
&\le
\left(
\sup_{0\le t\le T}
\|\nabla\theta_m(t)\|_{\mathbb H}^{2}
\right)
\int_0^T
\|\nabla\theta_m(s)\|_{\mathbb H}^{2}\,ds \notag\\
&\le
T
\|\nabla\theta(0)\|_{\mathbb H}^{4}
\exp\!\Bigg(
2C_{\nu,\beta,\kappa,T}
\Bigg(
1+
\Bigg(
1+\|\nabla\bm{y}_0\|_{\mathbb H}^{2}
+\int_0^T
\|\mathfrak{H}(s)\|_{\mathbb H}^{2}\,ds
\Bigg)^2
\Bigg)^{\frac12}
\Bigg).
\label{theta-grad4}
\end{align}
Combining \eqref{hold_um} and \eqref{theta-grad4}, we obtain
\begin{align}{\label{ath}}
\kappa^{2}
\int_{0}^{T}
\|{\mathcal{A}}\theta_m(s)\|_{H}^{2}\,ds
&\le
\kappa\|\nabla\theta_m(0)\|_{H}^{2}
+ 2\kappa^{2}\int_{0}^{T}
\|\nabla\theta_m(s)\|_{\mathbb H}^{4}\,ds \notag\\
&\quad
+ C
\left(
\int_{0}^{T}
\|A\bm{y}_m(s)\|_{\mathbb H}^{2}\,ds
\right)^{\frac12}
\left(
\int_{0}^{T}
\|\nabla\theta_m(s)\|_{\mathbb H}^{4}\,ds
\right)^{\frac12}.
\end{align}
By \eqref{aum-est}, \eqref{the_st}, and \eqref{theta-grad4}, and using the Young's inequality, \eqref{ath} yields
\begin{equation}\label{eq:theta_H2_estimate}
\kappa \|\nabla \theta_m(t)\|_{\mathbb H}^2
+ \kappa^2 \int_0^T \|{\mathcal{A}}\theta_m(s)\|_{H}^2 \, ds
\le
C\!\left(
\nu,\beta,\kappa,T,
\|\nabla \bm{y}_0\|_{\mathbb{H}},
\|\nabla \theta_0\|_{\mathbb H},
\|\mathfrak{H}\|_{{{L}^2(0,T;\mathbb{H})}}
\right).
\end{equation}
The uniform bounds in \eqref{aum-est}, \eqref{ut-est}, \eqref{the_st}, and \eqref{eq:theta_H2_estimate},
together with the Banach--Alaoglu theorem, enable us to extract a subsequence, again denoted by $\{\bm{y}_m,\theta_m\}$, such that
\[
\left\{
\begin{aligned}
\bm{y}_m &\stackrel{*}{\rightharpoonup} \bm{y}
&&\text{in }L^\infty(0,T;\mathbb V),\\
\bm{y}_m &\rightharpoonup \bm{y}
&&\text{in }L^2(0,T;D(A)),\\
\frac{d\bm{y}_m}{dt}
&\rightharpoonup
\frac{d\bm{y}}{dt}
&&\text{in }L^2(0,T;\mathbb H),
\end{aligned}
\right.
\qquad
\left\{
\begin{aligned}
\theta_m &\stackrel{*}{\rightharpoonup} \theta
&&\text{in }L^\infty(0,T;V),\\
\theta_m &\rightharpoonup \theta
&&\text{in }L^2(0,T;D(\mathcal{A})),\\
\frac{d\theta_m}{dt}
&\rightharpoonup
\frac{d\theta}{dt}
&&\text{in }L^2(0,T;H).
\end{aligned}
\right.
\]
Furthermore, by the Aubin--Lions--Simon compactness lemma, we infer 
\begin{equation}\label{conv_um_thetam}
\left\{
\begin{aligned}
\bm{y}_m &\to \bm{y}
&&\text{strongly in }L^2(0,T;\mathbb V)\cap C([0,T];\mathbb H),\\
\theta_m &\to \theta
&&\text{strongly in }L^2(0,T;V)\cap C([0,T];H).
\end{aligned}
\right.
\end{equation}
Since $\bm{y}_m(t) \in \mathcal{M}_1$ for all $t \in [0,T]$ and $\bm{y}_m \to \bm{y}$ in $C([0,T];\mathbb H)$, we deduce
\[
\|\bm{y}(t)\|_{\mathbb H} = \lim_{m \to \infty} \|\bm{y}_m(t)\|_{\mathbb H} = 1, 
\quad \text{for all } t \in [0,T].
\]
Similarly, since $\theta_m(t) \in \mathcal{M}_2$ for all $t \in [0,T]$ and $\theta_m \to \theta$ in $C([0,T];H)$, we obtain
\[
\|\theta(t)\|_{H} = \lim_{m \to \infty} \|\theta_m(t)\|_{H} = 1,
\quad \text{for all } t \in [0,T].
\]
Thus, $(\bm{y}(t),\theta(t)) \in \mathcal{M}$ for all $t \in [0,T]$.

The convergence of the linear terms in
\eqref{eq:galerkin_u}-\eqref{eq:galerkin_theta}
is immediate. It remains to justify the convergence of the nonlinear terms. Using Hölder's
inequality, Agmon's inequality, the uniform estimates
\eqref{aum-est}, \eqref{ut-est}, \eqref{the_st}, and \eqref{eq:theta_H2_estimate}, together with the strong convergence given in \eqref{conv_um_thetam}, we estimate
\begin{align}
&\int_0^T
\|B(\bm{y}_m(t),\bm{y}_m(t))
-
B(\bm{y}(t),\bm{y}(t))\|_{\mathbb H}^{2}\,dt
\nonumber\\&\le
2\int_0^T
\|B(\bm{y}_m(t)-\bm{y}(t),\bm{y}_m(t))\|_{\mathbb H}^{2}\,dt
 +2\int_0^T
\|B(\bm{y}(t),\bm{y}_m(t)-\bm{y}(t))\|_{\mathbb H}^{2}\,dt.
\label{nonlinear-conv-1}
\end{align}
The first term can be estimated using H\"older's and Agmon's inequalities as follows:
\begin{align}
&\int_{0}^{T}
\|B(\bm{y}_m(t)-\bm{y}(t),\bm{y}_m(t))\|_{\mathbb H}^{2}\,dt
\notag\\
&\le
C
\int_{0}^{T}
\|\bm{y}_m(t)-\bm{y}(t)\|_{\mathbb L^\infty}^{2}
\|\nabla\bm{y}_m(t)\|_{\mathbb H}^{2}\,dt
\notag\\
&\le
C
\left(
\int_{0}^{T}
\|\bm{y}_m(t)-\bm{y}(t)\|_{\mathbb L^\infty}^{4}\,dt
\right)^{\frac12}
\left(
\int_{0}^{T}
\|\nabla\bm{y}_m(t)\|_{\mathbb H}^{4}\,dt
\right)^{\frac12}
\notag\\
&\le
C
\left(
\int_{0}^{T}
\|\bm{y}_m(t)-\bm{y}(t)\|_{\mathbb H}^{2}
\|A(\bm{y}_m(t)-\bm{y}(t))\|_{\mathbb H}^{2}\,dt
\right)^{\frac12}
\left(
\int_{0}^{T}
\|\nabla\bm{y}_m(t)\|_{\mathbb H}^{4}\,dt
\right)^{\frac12}
\notag\\
&\le
C
\sup_{0\le t\le T}
\|\bm{y}_m(t)-\bm{y}(t)\|_{\mathbb H}
\left(
\int_{0}^{T}
\|A(\bm{y}_m(t)-\bm{y}(t))\|_{\mathbb H}^{2}\,dt
\right)^{\frac12}
\notag\\
&\qquad\times
\left(
\sup_{0\le t\le T}
\|\nabla\bm{y}_m(t)\|_{\mathbb H}^{2}
\right)
\left(
\int_{0}^{T}
\|\nabla\bm{y}_m(t)\|_{\mathbb H}^{2}\,dt
\right)^{\frac12}.
\label{nonlinear-conv-2}
\end{align}
Since \(\{\bm{y}_m\}\) is uniformly bounded in \( L^\infty(0,T;\mathbb V)\cap L^2(0,T;D(A))\) and \(\bm{y}_m\to\bm{y}\) strongly in \(C([0,T];\mathbb H)\), the right-hand side of \eqref{nonlinear-conv-2} tends to zero as \(m\to\infty\).
Similarly, for the second term, we get
\begin{align}
&\int_{0}^{T}\|B(\bm{y}(t),\bm{y}_m(t)-\bm{y}(t))\|_{\mathbb H}^{2}\,dt \notag\\
&\le C\int_{0}^{T}\|\bm{y}(t)\|_{\mathbb L^\infty}^{2}\|\nabla(\bm{y}_m(t)-\bm{y}(t))\|_{\mathbb H}^{2}\,dt \notag\\
&\le C\left(\int_{0}^{T}\|\bm{y}(t)\|_{\mathbb L^\infty}^{4}\,dt\right)^{\frac12}
\left(\int_{0}^{T}\|\nabla(\bm{y}_m(t)-\bm{y}(t))\|_{\mathbb H}^{4}\,dt\right)^{\frac12} \notag\\
&\le C\left(\int_{0}^{T}\|\bm{y}(t)\|_{\mathbb H}^{2}\|A\bm{y}(t)\|_{\mathbb H}^{2}\,dt\right)^{\frac12}
\left(\int_{0}^{T}\|\nabla(\bm{y}_m(t)-\bm{y}(t))\|_{\mathbb H}^{4}\,dt\right)^{\frac12} \notag\\
&\le C\sup_{0\le t\le T}\|\bm{y}(t)\|_{\mathbb H}
\left(\int_{0}^{T}\|A\bm{y}(t)\|_{\mathbb H}^{2}\,dt\right)^{\frac12}
\left(\int_{0}^{T}\|\nabla(\bm{y}_m(t)-\bm{y}(t))\|_{\mathbb H}^{4}\,dt\right)^{\frac12},
\label{conv22}
\end{align}
where we have used Agmon's inequality in the form
$
\|\bm v\|_{\mathbb L^\infty}^{2}\le C\|\bm v\|_{\mathbb H}\|A\bm v\|_{\mathbb H}.
$
Using the interpolation inequality, we obtain 
\[
\int_{0}^{T}\|\nabla(\bm{y}_m(t)-\bm{y}(t))\|_{\mathbb H}^{4}\,dt
\le
\sup_{0\le t\le T}\|\nabla(\bm{y}_m(t)-\bm{y}(t))\|_{\mathbb H}^{2}
\int_{0}^{T}\|\bm{y}_m(t)-\bm{y}(t)\|_{\mathbb V}^{2}\,dt,
\]
so that \eqref{conv22} implies 
\begin{align}
&\int_{0}^{T}\|B(\bm{y}(t),\bm{y}_m(t)-\bm{y}(t))\|_{\mathbb H}^{2}\,dt \notag\\
&\le C\sup_{0\le t\le T}\|\bm{y}(t)\|_{\mathbb H}
\left(\int_{0}^{T}\|A\bm{y}(t)\|_{\mathbb H}^{2}\,dt\right)^{\frac12}
\left(\sup_{0\le t\le T}\|\nabla(\bm{y}_m(t)-\bm{y}(t))\|_{\mathbb H}^{2}\right)^{\frac12}
\left(\int_{0}^{T}\|\bm{y}_m(t)-\bm{y}(t)\|_{\mathbb V}^{2}\,dt\right)^{\frac12}.
\label{conv33}
\end{align}
Using the uniform boundedness of \(\{\bm{y}_m\}\) in \( L^\infty(0,T;\mathbb V)\) and \(\bm{y}\in L^\infty(0,T;\mathbb V)\cap L^2(0,T;D(A))\), together with the strong convergence \(\bm{y}_m\to\bm{y}\) in \(L^2(0,T;\mathbb V)\), we obtain from \eqref{conv33} that
\[
\int_{0}^{T}\|B(\bm{y}(t),\bm{y}_m(t)-\bm{y}(t))\|_{\mathbb H}^{2}\,dt
\le C\left(\int_{0}^{T}\|\bm{y}_m(t)-\bm{y}(t)\|_{\mathbb V}^{2}\,dt\right)^{\frac12}
\to 0\ \ \text{ as } \  m\to\infty.
\]
Combining the above estimates with \eqref{nonlinear-conv-2}, we conclude that
\begin{align}
\int_{0}^{T}\|B(\bm{y}_m(t),\bm{y}_m(t))-B(\bm{y}(t),\bm{y}(t))\|_{\mathbb H}^{2}\,dt
\longrightarrow 0 \ \text{ as }\ 
 m\to\infty.
\label{convfinal}
\end{align}
Next, using \(\bm{y}_m(t)\in \mathcal{M}_1\) for all \(t\in[0,T]\), we estimate
\begin{align}
&\int_0^T \big\|\|\nabla \bm{y}_m(t)\|_{\mathbb H}^2 \bm{y}_m(t)
- \|\nabla \bm{y}(t)\|_{\mathbb H}^2 \bm{y}(t)\big\|_{\mathbb H}^2 dt \notag\\
&\le 2\int_0^T \big|\|\nabla \bm{y}_m(t)\|_{\mathbb H}^2 - \|\nabla \bm{y}(t)\|_{\mathbb H}^2\big|^2 \|\bm{y}_m(t)\|_{\mathbb H}^2 dt + 2\int_0^T \|\nabla \bm{y}(t)\|_{\mathbb H}^4 \|\bm{y}_m(t)-\bm{y}(t)\|_{\mathbb H}^2 dt.
\label{conv-term1}
\end{align}
Using the estimate
\[
\big|\|\nabla \bm{y}_m\|_{\mathbb H}^2 - \|\nabla \bm{y}\|_{\mathbb H}^2\big|
\le \big(\|\nabla \bm{y}_m\|_{\mathbb H} + \|\nabla \bm{y}\|_{\mathbb H}\big) \|\nabla(\bm{y}_m-\bm{y})\|_{\mathbb H},
\]
and the uniform boundedness of \(\{\bm{y}_m\}\) and \(\bm{y}\) in \(L^\infty(0,T;\mathbb V)\), we obtain
\begin{align}
\int_0^T \big|\|\nabla \bm{y}_m(t)\|_{\mathbb H}^2 - \|\nabla \bm{y}(t)\|_{\mathbb H}^2\big|^2 \|\bm{y}_m(t)\|_{\mathbb H}^2 dt 
&\le C\sup_{0\le t\le T}\|\bm{y}_m(t)\|_{\mathbb H}^2
\int_0^T \|\bm{y}_m(t)-\bm{y}(t)\|_{\mathbb V}^2 dt \notag\\
&\le C_T \int_0^T \|\bm{y}_m(t)-\bm{y}(t)\|_{\mathbb V}^2 dt \to 0,
\label{conv-term2}
\end{align}
as \(m\to\infty\), where we have used the strong convergence \(\bm{y}_m\to\bm{y}\) in \(L^2(0,T;\mathbb V)\).
Also, using the uniform boundedness of \(\bm{y}\) in \(L^\infty(0,T;\mathbb V)\), we have
\begin{align}
\int_0^T \|\nabla \bm{y}(t)\|_{\mathbb H}^4 \|\bm{y}_m(t)-\bm{y}(t)\|_{\mathbb H}^2 dt 
&\le \sup_{0\le t\le T}\|\nabla \bm{y}(t)\|_{\mathbb H}^4
\int_0^T \|\bm{y}_m(t)-\bm{y}(t)\|_{\mathbb H}^2 dt \notag\\
&\le C_T \sup_{0\le t\le T}\|\bm{y}_m(t)-\bm{y}(t)\|_{\mathbb H}^2 \to 0,
\label{conv-term3}
\end{align}
as \(m\to\infty\). Consequently, from \eqref{conv-term1}--\eqref{conv-term3}, we conclude  
\[
\int_0^T \big\|\|\nabla \bm{y}_m(t)\|_{\mathbb H}^2 \bm{y}_m(t)
- \|\nabla \bm{y}(t)\|_{\mathbb H}^2 \bm{y}(t)\big\|_{\mathbb H}^2 dt
\longrightarrow 0
\  \text{ as }\  m \to \infty.
\]
Now, we consider the term
\(
\theta_m e_2 - ( \theta_m e_2,\bm{y}_m)_{\mathbb H} \bm{y}_m
\)
and show its convergence. Writing it as
\begin{align}
&\big\|\theta_m e_2 - ( \theta_m e_2,\bm{y}_m)_{\mathbb H} \bm{y}_m
- \theta e_2 + ( \theta e_2,\bm{y})_{\mathbb H} \bm{y}\big\|_{\mathbb H} \notag\\
&\le \|(\theta_m-\theta)e_2\|_{\mathbb H}
+ \big\|( \theta_m e_2,\bm{y}_m)_{\mathbb H} - ( \theta e_2,\bm{y})_{\mathbb H}\big\|_{\mathbb{H}}\|\bm{y}_m\|_{\mathbb H}
+ |( \theta e_2,\bm{y})_{\mathbb H}|\|\bm{y}_m-\bm{y}\|_{\mathbb H}
\nonumber\\&\leq  \|\theta_m-\theta\|_{H}+ \left(\|\theta_m-\theta\|_{H}\|\bm{y}_m\|_{\mathbb H}
+ \|\theta\|_{H}\|\bm{y}_m-\bm{y}\|_{\mathbb H}\right)\|\bm{y}_m\|_{\mathbb H}+\|\theta\|_{H}\|\bm{y}\|_{\mathbb H}\|\bm{y}_m-\bm{y}\|_{\mathbb H}
\nonumber\\&\leq C\big(\|\theta_m-\theta\|_{H} + \|\bm{y}_m-\bm{y}\|_{\mathbb H}\big)\nonumber\\&
\to  0\ \text{ as }\ m\to\infty, 
\label{theta-conv1}
\end{align}
since \((\bm{y}_m,\theta_m)\in {\mathcal{M}}\) for all $t \in [0,T]$ and using the strong convergence given in \eqref{conv_um_thetam}.
The convergence of the terms in \eqref{eq:galerkin_theta} follows by arguments similar to those used above. Indeed, the term \(\|\nabla\theta_m(t)\|_{\mathbb H}^2\theta_m(t)\) can be handled exactly as in \eqref{conv-term1}--\eqref{conv-term3}. 
For the convective term \(\bm{y}_m\cdot\nabla\theta_m\), we proceed as in the estimate for \(B(\bm{y}_m,\bm{y}_m)\). First, we write
\[
\bm{y}_m\cdot\nabla\theta_m - \bm{y}\cdot\nabla\theta
= (\bm{y}_m-\bm{y})\cdot\nabla\theta_m + \bm{y}\cdot\nabla(\theta_m-\theta).
\]
Thus, for all \(t\in[0,T]\), we infer
\begin{align}
&\int_0^T \|\bm{y}_m(t)\cdot\nabla\theta_m(t) - \bm{y}(t)\cdot\nabla\theta(t)\|_{H}^2 dt \notag\\
&\le 2\int_0^T \|(\bm{y}_m(t)-\bm{y}(t))\cdot\nabla\theta_m(t)\|_{H}^2 dt
+ 2\int_0^T \|\bm{y}(t)\cdot\nabla(\theta_m(t)-\theta(t))\|_{H}^2 dt.
\label{conv-theta-main}
\end{align}
Using Hölder's inequality and Agmon's inequality, we estimate
\begin{align}
&\int_0^T
\|(\bm{y}_m(t)-\bm{y}(t))\cdot\nabla\theta_m(t)\|_{H}^{2}\,dt
\notag\\
&\le
C
\int_0^T
\|\bm{y}_m(t)-\bm{y}(t)\|_{\mathbb L^\infty}^{2}
\|\nabla\theta_m(t)\|_{\mathbb H}^{2}\,dt
\notag\\
&\le
C
\sup_{0\le t\le T}
\|\bm{y}_m(t)-\bm{y}(t)\|_{\mathbb H}
\left(
\int_0^T
\|A(\bm{y}_m(t)-\bm{y}(t))\|_{\mathbb H}^{2}\,dt
\right)^{\frac12}
\left(
\int_0^T
\|\nabla\theta_m(t)\|_{\mathbb H}^{4}\,dt
\right)^{\frac12}
\notag\\
&\le
C
\sup_{0\le t\le T}
\|\bm{y}_m(t)-\bm{y}(t)\|_{\mathbb H}
\left(
\int_0^T
\|A\bm{y}_m(t)\|_{\mathbb H}^{2}\,dt
+
\int_0^T
\|A\bm{y}(t)\|_{\mathbb H}^{2}\,dt
\right)^{\frac12}
\left(
\int_0^T
\|\nabla\theta_m(t)\|_{\mathbb H}^{4}\,dt
\right)^{\frac12}
\longrightarrow0,
\label{conv-theta1}
\end{align}
as \(m\to\infty\), where we have used the uniform boundedness of $\{\theta_m\}$ in $L^\infty(0,T;V)$ and \(\bm{y}\) in
\( L^\infty(0,T;\mathbb V)\cap L^2(0,T;D(A))\),
and the strong convergence
\[
\bm{y}_m\to\bm{y}
\quad\text{in }L^2(0,T;\mathbb V)\cap C([0,T];\mathbb H).
\]
Similarly, we have
\begin{align}\label{conv-theta2}
&\int_0^T
\|\bm{y}(t)\cdot\nabla(\theta_m(t)-\theta(t))\|_H^2\,dt
\nonumber\\
&\quad\leq
\left(\int_0^T\|\bm{y}(t)\|_{\mathbb{L}^\infty}^4\,dt\right)^{1/2}
\left(\int_0^T\|\theta_m(t)-\theta(t)\|_V^4\,dt\right)^{1/2}
\nonumber\\
&\quad\leq
C\left(
\operatorname*{ess\,sup}_{0<t<T}
\|\theta_m(t)-\theta(t)\|_V^2
\right)^{1/2}
\left(
\int_0^T\|\theta_m(t)-\theta(t)\|_V^2\,dt
\right)^{1/2}
\longrightarrow0.
\end{align}
as \(m\to\infty\). Combining \eqref{conv-theta-main}--\eqref{conv-theta2}, we conclude that
\[
\int_0^T \|\bm{y}_m(t)\cdot\nabla\theta_m(t) - \bm{y}(t)\cdot\nabla\theta(t)\|_{H}^2 dt
\longrightarrow 0,
\qquad m\to\infty.
\]
For any $\varphi\in C([0,T];\mathbb V_n)$ with $n<m$, we have
\begin{align}\label{eq:control-convergence}
&\left|\int_0^T \Big(
\big(\mathfrak{H}_m(t)-(\mathfrak{H}_m(t),\bm{y}_m(t))_{\mathbb H}\bm{y}_m(t)\big)
-\big(\mathfrak{H}(t)-(\mathfrak{H}(t),\bm{y}(t))_{\mathbb H}\bm{y}(t)\big),
\varphi(t)
\Big)_{\mathbb H}\,dt\right| \notag\\
&\le
\left|\int_0^T
(\mathfrak{H}_m(t)-\mathfrak{H}(t),\varphi(t))_{\mathbb H}\,dt\right| \notag\\
&\quad+
\left|\int_0^T
\big(
(\mathfrak{H}_m(t),\bm{y}_m(t))_{\mathbb H}
-(\mathfrak{H}(t),\bm{y}(t))_{\mathbb H}
\big)
(\bm{y}(t),\varphi(t))_{\mathbb H}\,dt\right| \notag\\
&\quad+
\left|\int_0^T
(\mathfrak{H}_m(t),\bm{y}_m(t))_{\mathbb H}
(\bm{y}_m(t)-\bm{y}(t),\varphi(t))_{\mathbb H}\,dt\right| \notag\\
&=:I_{1,m}+I_{2,m}+I_{3,m}.
\end{align}
It follows directly from the strong convergence of $\bm{y}_m\to \bm{y}$ (see \eqref{conv_um_thetam})  and the fact that $\mathfrak{H}_m = \Pi_m \mathfrak{H}\to \mathfrak{H}$, that
\[
I_{1,m}\to 0 \ \ \text{and}\ \ I_{2,m}\to 0
\ \ \text{as } m\to\infty.
\]
Moreover, by the convergence in \eqref{conv_um_thetam}, the fact that $\mathfrak{H}\in L^2(0,T;\mathbb H)$, and the uniform boundedness of $\bm{y}_m$, we obtain
\[
I_{3,m}\to 0 \qquad\text{as } m\to\infty.
\]
Therefore, we write
\[
\int_0^T \Big(
\big(\mathfrak{H}_m(t)-( \mathfrak{H}_m(t),\bm{y}_m(t))_{\mathbb H}\bm{y}_m(t)\big),\varphi(t) \Big)_{\mathbb H}\,dt
\to  \int_0^T \Big( \big(\mathfrak{H}(t)-( \mathfrak{H}(t),\bm{y}(t))_{\mathbb H}\bm{y}(t)\big),
\,\varphi(t)
\Big)_{\mathbb H}\,dt
\]
for every $\varphi\in C([0,T];\mathbb V_n)$ and every $n\in\mathbb N$. By density of \(C([0,T]; \mathbb V_n)\) in \(L^2(0,T; \mathbb H)\), the above convergence extends to all $\varphi\in L^2(0,T;\mathbb H)$.
Taking the limit in \eqref{eq:galerkin_u}--\eqref{eq:galerkin_theta}, we find that $(\bm{y},\theta)$ is a strong solution of the constrained Boussinesq system \eqref{eq:constrained_Boussinesq}. By the convergences established above and the weak lower semicontinuity of the norms, we infer
\begin{equation}\label{reg-um}
\bm{y} \in L^\infty(0,T; \mathbb V) \cap L^2(0,T; D(A)), \qquad 
\frac{d\bm{y}}{dt} \in L^2(0,T; \mathbb H),
\end{equation}
and
\begin{equation}\label{reg-theta}
\theta \in L^\infty(0,T; V) \cap L^2(0,T; D(\mathcal{A})), \qquad 
\frac{d\theta}{dt} \in L^2(0,T; H).
\end{equation}
By the Lions--Magenes lemma, the norms
\[
[0,T] \ni t \mapsto \|\nabla \bm{y}(t)\|_{\mathbb H}^2\in[0,\infty)
\ \text{ and }\ 
[0,T] \ni t \mapsto \|\nabla \theta(t)\|_{\mathbb H}^2\in[0,\infty),
\]
are absolutely continuous.

\noindent\textbf{Uniqueness of solutions.}
Let $(\bm{y}_1,\theta_1)$ and $(\bm{y}_2,\theta_2)$ be two strong solutions to \eqref{eq:constrained_Boussinesq} corresponding to the same initial data $(\bm{y}_0,\theta_0)\in\mathcal{M}$ and the same control $\mathfrak{H}\in L^2(0,T;\mathbb{H})$. Define \(\bm{w}:=\bm{y}_1-\bm{y}_2\) and \(\xi:=\theta_1-\theta_2\). Then \((\bm{w},\xi)\) satisfies
\begin{align}{\label{eqw}}
\frac{d\bm{w}(t)}{dt}
&= -\nu A \bm{w}(t) 
+ \nu\big(\|\nabla \bm{y}_1(t)\|_{\mathbb{H}}^2 \bm{y}_1(t) - \|\nabla \bm{y}_2(t)\|_{\mathbb{H}}^2 \bm{y}_2(t)\big)
- \big(B(\bm{y}_1(t),\bm{y}_1(t)) - B(\bm{y}_2(t),\bm{y}_2(t))\big)\nonumber \\
&\quad + \beta {\textbf{Proj}_{\mathbb H}}\big((\theta_1(t)-\theta_2(t))e_2\big)
- \beta\big( (\theta_1(t)e_2,\bm{y}_1(t))_{\mathbb{H}} \bm{y}_1(t) -(\theta_2(t)e_2,\bm{y}_2(t))_{\mathbb{H}} \bm{y}_2(t)\big)\nonumber \\
&\quad - \big(( \mathfrak{H}(t),\bm{y}_1(t))_{\mathbb{H}} \bm{y}_1(t) - ( \mathfrak{H}(t),\bm{y}_2(t))_{\mathbb{H}} \bm{y}_2(t)\big),
\end{align}
and
\begin{align}{\label{eqxi}}
\frac{d\xi(t)}{dt}
= -\kappa {\mathcal{A}} \xi(t)
+ \kappa\big(\|\nabla \theta_1(t)\|_{\mathbb H}^2 \theta_1(t) - \|\nabla \theta_2(t)\|_{\mathbb H}^2 \theta_2(t)\big)
- \big(\bm{y}_1(t) \cdot \nabla \theta_1(t) - \bm{y}_2(t) \cdot \nabla \theta_2(t)\big),
\end{align}
for a.e. $t\in (0,T)$. 
Taking the $\mathbb{H}$-inner product with $\bm{w}(\cdot)$ in \eqref{eqw}, we obtain
\begin{equation}\label{eqn-est-2}
\begin{aligned}
\frac{1}{2}\frac{d}{dt}\|\bm{w}(t)\|_{\mathbb{H}}^2 + \nu \|\nabla \bm{w}(t)\|_{\mathbb{H}}^2
&= -( B(\bm{w}(t),\bm{y}_1(t)), \bm{w}(t) )_{\mathbb{H}}
- ( B(\bm{y}_2(t),\bm{w}(t)), \bm{w}(t) )_{\mathbb{H}} \\
&\quad + \nu( \|\nabla \bm{y}_1(t)\|_{\mathbb H}^2 \bm{y}_1(t) - \|\nabla \bm{y}_2(t)\|_{\mathbb H}^2 \bm{y}_2(t), \bm{w}(t) )_{\mathbb{H}} \\
&\quad + \beta ( (\theta_1(t)-\theta_2(t))e_2, \bm{w}(t) )_{\mathbb{H}} \\
&\quad - \beta ( ( \theta_1(t)e_2,\bm{y}_1(t))_{\mathbb H}\bm{y}_1(t) - ( \theta_2(t)e_2,\bm{y}_2(t))_{\mathbb H}\bm{y}_2(t), \bm{w}(t) )_{\mathbb{H}} \\
&\quad - ( ( \mathfrak{H}(t),\bm{y}_1(t))_{\mathbb H}\bm{y}_1(t) - ( \mathfrak{H}(t),\bm{y}_2(t))_{\mathbb H}\bm{y}_2(t), \bm{w}(t) )_{\mathbb{H}},
\end{aligned}
\end{equation}
for a.e. \(t\in(0,T)\). By applying H\"older's, Ladyzhenskaya's, and Young's inequalities, we bound the term \(( B(\bm{w},\bm{y}_1), \bm{w} )_{\mathbb H}\) as follows. Since \(( B(\bm{y}_2,\bm{w}), \bm{w} )_{\mathbb H} = 0\), we have
\[
\begin{aligned}
|( B(\bm{y}_1,\bm{y}_1) - B(\bm{y}_2,\bm{y}_2), \bm{w} )_{\mathbb H}|
=|( B(\bm{w},\bm{y}_1), \bm{w} )_{\mathbb H}|
&\le C \|\nabla \bm{y}_1\|_{\mathbb H} \|\bm{w}\|_{\mathbb L^4}^2 \\
&\le C \|\nabla \bm{y}_1\|_{\mathbb H} \|\nabla \bm{w}\|_{\mathbb H} \|\bm{w}\|_{\mathbb H} \\
&\le \frac{\nu}{6}\|\nabla \bm{w}\|_{\mathbb H}^2 + C_{\nu}\|\nabla \bm{y}_1\|_{\mathbb H}^2 \|\bm{w}\|_{\mathbb H}^2.
\end{aligned}
\]
Writing
\[
\|\nabla \bm{y}_1\|_{\mathbb{H}}^2 \bm{y}_1
- \|\nabla \bm{y}_2\|_{\mathbb{H}}^2 \bm{y}_2
= \|\nabla \bm{y}_1\|_{\mathbb{H}}^2 \bm{w}
+ \big(\|\nabla \bm{y}_1\|_{\mathbb{H}}^2
- \|\nabla \bm{y}_2\|_{\mathbb{H}}^2\big)\bm{y}_2,
\]
and then, applying the Cauchy--Schwarz, triangle, and Young inequalities, we obtain the following estimate:
\[
\begin{aligned}
&\left|\nu\left( \|\nabla \bm{y}_1\|_{\mathbb{H}}^2 \bm{y}_1 - \|\nabla \bm{y}_2\|_{\mathbb{H}}^2 \bm{y}_2, \bm{w} \right)_{\mathbb{H}}\right| \\
&\le \nu\|\nabla \bm{y}_1\|_{\mathbb{H}}^2 \|\bm{w}\|_{\mathbb{H}}^2
+ \nu(\|\nabla \bm{y}_1\|_{\mathbb{H}} + \|\nabla \bm{y}_2\|_{\mathbb{H}})
\|\nabla \bm{w}\|_{\mathbb{H}} \|\bm{w}\|_{\mathbb{H}} \\
&\le \frac{\nu}{6}\|\nabla \bm{w}\|_{\mathbb{H}}^2
+ C_\nu\left(\|\nabla \bm{y}_1\|_{\mathbb{H}}^2 + \|\nabla \bm{y}_2\|_{\mathbb{H}}^2\right)\|\bm{w}\|_{\mathbb{H}}^2.
\end{aligned}
\]
Since $\bm{y}_1(t), \bm{y}_2(t) \in \mathcal{M}_1$ for all $t \in [0,T]$, the Cauchy--Schwarz inequality yields
\[
\left| -( \mathfrak{H},\bm{y}_1)_{\mathbb H}\|\bm{w}\|_{\mathbb H}^2 - ( \mathfrak{H},\bm{w})_{\mathbb H}( \bm{y}_2,\bm{w})_{\mathbb H} \right|
\le 2\|\mathfrak{H}\|_{\mathbb H}\|\bm{w}\|_{\mathbb H}^2.
\]
For the coupling term, we have
\[
\left|\beta( (\theta_1-\theta_2)e_2, \bm{w} )_{\mathbb H}\right|
\le \beta\|\xi\|_{H}\|\bm{w}\|_{\mathbb H}\leq \frac{1}{2}\|\xi\|_{H}^2+\frac{\beta^2}{2}\|\bm{w}\|_{\mathbb{H}}^2.
\]
Applying the Cauchy--Schwarz inequality followed by Young's inequality, together with the facts that
\(\bm{y}_i(t)\in\mathcal{M}_1\) and \(\theta_i(t)\in\mathcal{M}_2\) for \(i=1,2\) and for all $t\in[0,T],$ we have
\begin{align*}
&\left|\beta\left(
(\theta_1 e_2,\bm{y}_1)_{\mathbb H}\bm{y}_1
-(\theta_2 e_2,\bm{y}_2)_{\mathbb H}\bm{y}_2,
\bm w\right)_{\mathbb H}\right|\\
&\quad\le
\left|\beta\left(
(\theta_1 e_2,\bm{y}_1)_{\mathbb H}\bm w,
\bm w\right)_{\mathbb H}\right|
+\left|\beta\left(
((\theta_1-\theta_2)e_2,\bm{y}_1)_{\mathbb H}\bm{y}_2,
\bm w\right)_{\mathbb H}\right|\\
&\qquad+
\left|\beta\left(
(\theta_2 e_2,\bm w)_{\mathbb H}\bm{y}_2,
\bm w\right)_{\mathbb H}\right|\\
&\quad\le
2|\beta|\|\bm w\|_{\mathbb H}^2
+|\beta|\|\xi\|_H\|\bm w\|_{\mathbb H}\\
&\quad\le
\left(2|\beta|+\frac{\beta^2}{2}\right)
\|\bm w\|_{\mathbb H}^2
+\frac12\|\xi\|_H^2.
\end{align*}
Thus,  we infer from \eqref{eqn-est-2} that 
\begin{align}{\label{weqq}}
\frac{1}{2}\frac{d}{dt}\|\bm{w}(t)\|_{\mathbb{H}}^2 + \frac{2\nu}{3} \|\nabla \bm{w}(t)\|_{\mathbb{H}}^2
\le C_{\beta,\nu}\big(1 + \|\nabla \bm{y}_1(t)\|_{\mathbb{H}}^2 + \|\nabla \bm{y}_2(t)\|_{\mathbb{H}}^2 + \|\mathfrak{H}(t)\|_{\mathbb{H}}\big)\|\bm{w}(t)\|_{\mathbb{H}}^2
+ C\|\xi(t)\|_{H}^2,
\end{align}
for a.e. $t\in(0,T)$. Taking the inner product in $H$ of \eqref{eqxi} with $\xi(\cdot)$, we obtain
\[
\begin{aligned}
\frac{1}{2}\frac{d}{dt}\|\xi(t)\|_{H}^2 + \kappa \|\nabla \xi(t)\|_{\mathbb H}^2
&= \kappa\left( \|\nabla \theta_1(t)\|_{\mathbb H}^2\theta_1(t) - \|\nabla \theta_2(t)\|_{\mathbb H}^2\theta_2(t), \xi(t) \right)_{H} \\
&\quad - \left( \bm{y}_1(t)\cdot\nabla\theta_1(t) - \bm{y}_2(t)\cdot\nabla\theta_2(t), \xi(t) \right)_{H},
\end{aligned}
\]
for a.e. \(t \in (0,T)\).
We write $
\bm{y}_1 \cdot \nabla \theta_1 - \bm{y}_2 \cdot \nabla \theta_2
= \bm{w} \cdot \nabla \theta_1 + \bm{y}_2 \cdot \nabla \xi,
$
and using $\nabla \cdot \bm{y}_2 = 0$, we have
$
( \bm{y}_2 \cdot \nabla \xi, \xi )_{H} = 0.
$
Consequently, we write
\[
\begin{aligned}
\frac{1}{2}\frac{d}{dt}\|\xi(t)\|_{H}^2 + \kappa \|\nabla \xi(t)\|_{\mathbb H}^2
&= \kappa\left( \|\nabla \theta_1(t)\|_{\mathbb H}^2\theta_1(t) - \|\nabla \theta_2(t)\|_{\mathbb H}^2\theta_2(t), \xi(t) \right)_{H}
- \left( \bm{w}(t) \cdot \nabla \theta_1(t), \xi(t) \right)_{H},
\end{aligned}
\]
for a.e. $t\in[0,T]$. Hence, by Ladyzhenskaya's inequality and the properties of the trilinear map \eqref{eq:B_A_cancel}, we have
\begin{align}
|( \bm{w} \cdot \nabla \theta_1, \xi )_{H}|
&=|( \bm{w} \cdot \nabla \xi,  \theta_1)_{H}|
\nonumber\\&\leq \|\bm{w}\|_{\mathbb{L}^4}\|\nabla\xi\|_{\mathbb{H}}\|\theta_1\|_{L^4}
\nonumber\\&\leq \frac{\kappa}{4}\|\nabla\xi\|_{\mathbb{H}}^2+C_{\kappa}\|\bm{w}\|_{\mathbb{H}}\|\nabla\bm{w}\|_{\mathbb{H}}\|\theta_1\|_{L^4}^2
\nonumber\\&\leq\frac{\kappa}{4}\|\nabla\xi\|_{\mathbb{H}}^2+\frac{\nu}{6}\|\nabla\bm{w}\|_{\mathbb{H}}^2+C_{\nu,\kappa}\|\theta_1\|_{L^4}^4\|\bm{w}\|_{\mathbb{H}}^2.
\end{align}
Similarly, using the identity
\[
\|\nabla\theta_1\|_{\mathbb H}^2\theta_1
-\|\nabla\theta_2\|_{\mathbb H}^2\theta_2
=
\|\nabla\theta_1\|_{\mathbb H}^2\xi
+
\bigl(\|\nabla\theta_1\|_{\mathbb H}^2
-\|\nabla\theta_2\|_{\mathbb H}^2\bigr)\theta_2.
\]
we estimate
\begin{align*}
&\left|
\kappa\bigl(
\|\nabla\theta_1\|_{\mathbb H}^2\theta_1
-\|\nabla\theta_2\|_{\mathbb H}^2\theta_2,\xi
\bigr)_H
\right|\\
&\quad\leq
\kappa\|\nabla\theta_1\|_{\mathbb H}^2\|\xi\|_H^2
+\kappa\bigl(\|\nabla\theta_1\|_{\mathbb H}+\|\nabla\theta_2\|_{\mathbb H}\bigr)
\|\nabla\xi\|_H\|\xi\|_H\\
&\quad\leq
\frac{\kappa}{4}\|\nabla\xi\|_H^2
+C_\kappa
\bigl(\|\nabla\theta_1\|_{\mathbb H}^2+\|\nabla\theta_2\|_{\mathbb H}^2\bigr)
\|\xi\|_H^2.
\end{align*}
Thus,  we have
\begin{align}
\frac{1}{2}\frac{d}{dt}\|\xi(t)\|_{H}^2 + \frac{\kappa}{2} \|\nabla \xi(t)\|_{\mathbb H}^2
&\le C_{\nu,\kappa}\big(1 + \|\nabla \theta_1(t)\|_{\mathbb H}^2 + \|\nabla \theta_2(t)\|_{\mathbb H}^2\big)\|\xi(t)\|_{H}^2 \notag\\
&\quad +C_{\nu,\kappa}\|\theta_1\|_{L^4}^4\|\bm{w}(t)\|_{\mathbb{H}}^2+\frac{\nu}{6}\|\nabla \bm{w}(t)\|_{\mathbb{H}}^2,
\label{xii}
\end{align}
for a.e. $t\in(0,T)$. Adding \eqref{weqq} and \eqref{xii}, we obtain for a.e. $t\in(0,T)$ that 
\[
\frac{d}{dt}\big(\|\bm{w}(t)\|_{\mathbb{H}}^2 + \|\xi(t)\|_{H}^2\big)
+ \nu\|\nabla \bm{w}(t)\|_{\mathbb{H}}^2 + \kappa\|\nabla \xi(t)\|_{\mathbb H}^2
\le C_{\bm{y},\theta}(t)\big(\|\bm{w}(t)\|_{\mathbb{H}}^2 + \|\xi(t)\|_{H}^2\big),
\]
where
\[
C_{\bm{y},\theta}(t) = C_{\nu,\kappa,\beta}\big(1 + \|\nabla \bm{y}_1(t)\|_{\mathbb{H}}^2 + \|\nabla \bm{y}_2(t)\|_{\mathbb{H}}^2+\|\nabla\theta_1(t)\|_{\mathbb{H}}^2+\|\nabla\theta_2(t)\|_{\mathbb H}^2+ \|\theta_1(t)\|_{L^4}^4 + \|\mathfrak{H}(t)\|_{\mathbb{H}}\big).
\]
Since $\bm{w}(0)=\bm 0$ and $\xi(0)=0$, an application of Grönwall's inequality yields
\[
\|\bm{w}(t)\|_{\mathbb{H}}^2 + \|\xi(t)\|_{H}^2 = 0 \quad \text{for all } t \in [0,T].
\]
Hence, $\bm{y}_1(t)=\bm{y}_2(t)$ and $\theta_1(t)=\theta_2(t)$ for all $t\in[0,T]$, which proves the uniqueness of the strong solution.
\end{proof}
\section{Optimal control problem}{\label{s3}}
Here, we show the existence of an optimal control that minimizes the cost functional given by
\begin{align}
J(\bm{y},\theta,\mathfrak{H}) 
:=\;& \frac{1}{2} \int_0^T \|\bm{y}(t) - \bm{y}_d(t)\|_{\mathbb{H}}^2 \, dt 
+ \frac{1}{2} \int_0^T \|\nabla \bm{y}(t)\|_{\mathbb{H}}^2 \, dt \notag\\
&+ \frac{1}{2} \int_0^T \|\theta(t) - \theta_d(t)\|_{H}^2 \, dt
+ \frac{1}{2} \int_0^T \|\mathfrak{H}(t)\|_{\mathbb{H}}^2 \, dt \notag\\
&+ \frac{1}{2} \|\bm{y}(T) - \bm{y}_f\|_{\mathbb{H}}^2
+ \frac{1}{2} \|\theta(T) - \theta_f\|_{H}^2,
\end{align}
subject to the system \eqref{eq:constrained_Boussinesq}.
Here, $\bm{y}_d \in {{L}^2(0,T;\mathbb{H})}$ and $\theta_d \in L^2(0,T;H)$ denote the desired velocity and temperature fields, respectively, while $\bm{y}_f \in \mathbb{H}$ and $\theta_f \in H$ represent the prescribed target states at the terminal time $T$. The objective is to find a control $\mathfrak{H}$ such that the corresponding state variables $(\bm{y},\theta)$ remain as close as possible to the desired states $(\bm{y}_d,\theta_d)$, while simultaneously minimizing turbulence and thermal fluctuations.
For Banach spaces $\mathbb{X}$ and $\mathbb{Y}$, we define
\[
W^{1,2}(0,T;\mathbb{X},\mathbb{Y})
:=
\left\{
v \in L^2(0,T;\mathbb{X})
:\,
\frac{dv}{dt}\in L^2(0,T;\mathbb{Y})
\right\}.
\]
The Lions--Magenes lemma (\cite{temmm}) gives, the continuous embeddings
\[
W^{1,2}(0,T;\mathbb{V},\mathbb{V}')
\hookrightarrow C([0,T];\mathbb{H})
\]
and
\[
W^{1,2}(0,T;V,V')
\hookrightarrow C([0,T];H)
\]
holds. Furthermore, if additional spatial regularity is available, then
\[
W^{1,2}(0,T;D(A),\mathbb{H})
\hookrightarrow C([0,T];\mathbb{V}).
\]

\begin{definition}[Admissible class]
The admissible class $\mathcal{A}_{ad}$ consists of all triplets $(\bm{y},\theta,\mathfrak{H})$ such that
\[
(\bm{y},\theta,\mathfrak{H}) \in \mathcal{X}
:= \Big( W^{1,2}(0,T;D(A),\mathbb{H}) \Big)
\times \Big( W^{1,2}(0,T;D(\mathcal{A}),H) \Big)
\times {{L}^2(0,T;\mathbb{H})},
\]
where $(\bm{y},\theta)$ is the unique strong solution of the Boussinesq system \eqref{eq:constrained_Boussinesq} associated with the control $\mathfrak{H} \in \mathcal{U}_{ad} := {{L}^2(0,T;\mathbb{H})}$. That is,
\[
\mathcal{A}_{ad}
:= \Big\{ (\bm{y},\theta,\mathfrak{H}) \in \mathcal{X} : (\bm{y},\theta) \text{ solves } \eqref{eq:constrained_Boussinesq} \text{ with } \mathfrak{H}\in \mathcal U_{ad}, \text{ and } (\bm{y}(t),\theta(t))\in\mathcal M,\,\text{ for every }t\in[0,T] \Big\}.
\]
\end{definition}

\noindent\hspace*{2em}We endow the product space $\mathcal{X}$ with its natural weak topology. The admissible set $\mathcal{A}_{ad}$ is nonempty, since for each $\mathfrak{H} \in \mathcal{U}_{ad}$, there is a unique strong solution $(\bm{y},\theta)$ to the system \eqref{eq:constrained_Boussinesq}. Note that $\mathcal{A}_{ad}$ is not convex; however, it is sequentially weakly closed. Let $\{(\bm{y}_n,\theta_n,\mathfrak{H}_n)\} \subset \mathcal{A}_{ad}$ be a sequence such that
\[
(\bm{y}_n,\theta_n,\mathfrak{H}_n) \rightharpoonup (\bm{y}^*,\theta^*,\mathfrak{H}^*)
\quad \text{weakly in } \mathcal{X}.
\]
Then one can show that
$(\bm{y}^*,\theta^*,\mathfrak{H}^*)\in\mathcal{A}_{ad}$.
This means that $(\bm{y}^*,\theta^*)$ is a strong solution of the constrained Boussinesq system corresponding to the control
$\mathfrak{H}^*$, with $\|\bm{y}^*(t)\|_{\mathbb{H}}=1$ and
$\|\theta^*(t)\|_H=1$ for all $t\in[0,T]$.
This solution is also unique.
Thus, we formulate the OCP as follows:
\begin{equation}\label{opt}
\min_{(\bm{y},\theta,\mathfrak{H})\in \mathcal{A}_{ad}} J(\bm{y},\theta,\mathfrak{H}).
\end{equation}
We refer to a solution of \eqref{opt} as an \emph{optimal solution}. The triplet $(\bm{y}^*,\theta^*,\mathfrak{H}^*)$ is then called
an \emph{optimal triplet}, with $\mathfrak{H}^*$ as the
\emph{optimal control}.

\noindent\hspace*{2em} We now prove that the problem \eqref{opt} admits a solution.

\begin{theorem}{\label{optt}}
The optimal control problem \eqref{opt} admits at least one solution 
$(\bm{y}^*,\theta^*,\mathfrak{H}^*) \in \mathcal{A}_{ad}$.
\end{theorem}

\begin{proof}
Let $\{(\bm{y}_n,\theta_n,\mathfrak{H}_n)\} \subset \mathcal{A}_{ad}$ be a minimizing sequence for the optimal control problem \eqref{opt}, that is,
\[
\inf J \le J(\bm{y}_n,\theta_n,\mathfrak{H}_n) \le \inf J + \frac{1}{n}.
\]
Here, $(\bm{y}_n,\theta_n)$ is the unique strong solution of the Boussinesq system associated with the control $\mathfrak{H}_n \in {{L}^2(0,T;\mathbb{H})}$, satisfying
\[
\begin{cases}
\dfrac{d\bm{y}_n(t)}{dt}
= -\nu A \bm{y}_n(t) + \nu \|\nabla \bm{y}_n(t)\|_{\mathbb H}^2 \bm{y}_n(t) - B(\bm{y}_n(t),\bm{y}_n(t)) \\
\qquad\quad + \beta {\textbf{Proj}}_{\mathbb H}(\theta_n(t)e_2) - \beta ( \theta_n(t)e_2,\bm{y}_n(t))_{\mathbb{H}}\, \bm{y}_n(t) \\
\qquad\quad + \mathfrak{H}_n(t) - ( \mathfrak{H}_n(t),\bm{y}_n(t))_{\mathbb{H}}\, \bm{y}_n(t), \\[8pt]
\dfrac{d\theta_n(t)}{dt}
= -\kappa {\mathcal{A}} \theta_n(t) + \kappa \|\nabla \theta_n(t)\|_{H}^2 \theta_n(t) - \bm{y}_n(t)\cdot \nabla \theta_n(t), \\[8pt]
(\bm{y}_n,\theta_n)(0) = (\bm{y}_0,\theta_0),
\end{cases}
\]
where $(\bm{y}_0,\theta_0) \in (\mathbb V\cap\mathcal{M}_1) \times (V \cap \mathcal{M}_2)$.

\noindent\hspace*{2em}Using calculations similar to those in estimates \eqref{bd}--\eqref{ut-est} and \eqref{the_st}--\eqref{eq:theta_H2_estimate} for the Boussinesq system, along with the Banach--Alaoglu theorem, we deduce the existence of subsequences (still denoted by the same symbols) such that
\begin{equation}
\left\{
\begin{aligned}
\begin{aligned}
\bm{y}_n &\rightharpoonup^* \bm{y}^* && \text{in } L^\infty(0,T;\mathbb{V}), \\
\bm{y}_n &\rightharpoonup \bm{y}^* && \text{in } L^2(0,T;D(A)), \\
\dfrac{d\bm{y}_n}{dt} &\rightharpoonup \dfrac{d\bm{y}^*}{dt} && \text{in } L^2(0,T;\mathbb{H}),
\end{aligned}
\quad
\begin{aligned}
\theta_n &\rightharpoonup^* \theta^* && \text{in } L^\infty(0,T;V), \\
\theta_n &\rightharpoonup \theta^* && \text{in } L^2(0,T;D(\mathcal{A})), \\
\dfrac{d\theta_n}{dt} &\rightharpoonup \dfrac{d\theta^*}{dt} && \text{in } L^2(0,T;H), \\
\mathfrak{H}_n &\rightharpoonup \mathfrak{H}^* && \text{in } L^2(0,T;\mathbb{H}).
\end{aligned}
\end{aligned}
\right.
\end{equation}
By the Aubin--Lions lemma, we also obtain the strong convergences
\begin{equation}
\left\{
\begin{aligned}
\bm{y}_n &\to \bm{y}^* && \text{in } L^2(0,T;\mathbb{V}), \\
\bm{y}_n &\to \bm{y}^* && \text{in } C([0,T];\mathbb{H}), \\
\theta_n &\to \theta^* && \text{in } L^2(0,T;V), \\
\theta_n &\to \theta^* && \text{in } C([0,T];H).
\end{aligned}
\right.
\end{equation}
\noindent Since $(\bm{y}_n(t), \theta_n(t)) \in \mathcal{M}$ for all $t \in [0,T]$, by the strong convergences, we infer
\[
\|\bm{y}^*(t)\|_{\mathbb{H}} = 1  \ \text{ and }\ 
\|\theta^*(t)\|_{H} = 1 \  \text{ for all }\  t \in [0,T].
\]
The above convergences allow us to pass to the limit in both the linear and nonlinear terms of the Boussinesq system. Note that
\[
\mathfrak{H}_n \rightharpoonup \mathfrak{H}^* \text{ in } {{L}^2(0,T;\mathbb{H})}, \ \
\bm{y}_n \to \bm{y}^* \text{ in } C([0,T];\mathbb{H}), \ \
\theta_n \to \theta^* \text{ in } L^2(0,T;H).
\]
We claim that
\[
( \mathfrak{H}_n,\bm{y}_n)_{\mathbb{H}} \rightharpoonup ( \mathfrak{H}^*,\bm{y}^*)_{\mathbb{H}}
\quad \text{in } L^2(0,T).
\]
Indeed, for any $\psi \in C_c^\infty(0,T)$,
\[
\int_0^T ( \mathfrak{H}_n(t),\bm{y}_n(t))_{\mathbb{H}} \psi(t)\,dt
= \int_0^T ( \mathfrak{H}_n(t),\bm{y}^*(t))_{\mathbb{H}} \psi(t)\,dt
+ \int_0^T ( \mathfrak{H}_n(t),\bm{y}_n(t)-\bm{y}^*(t))_{\mathbb{H}} \psi(t)\,dt.
\]
The first term converges to $\int_0^T ( \mathfrak{H}^*(t),\bm{y}^*(t))_{\mathbb{H}} \psi(t)\,dt$ by weak convergence of $\mathfrak{H}_n$, while the second term tends to zero due to the boundedness of $\{\mathfrak{H}_n\}$ in ${{L}^2(0,T;\mathbb{H})}$ and the strong convergence of $\bm{y}_n$. Hence, we get
\[
( \mathfrak{H}_n,\bm{y}_n)_{\mathbb{H}} \rightharpoonup ( \mathfrak{H}^*,\bm{y}^*)_{\mathbb{H}} \quad \text{in } L^2(0,T).
\]
Since $C_c^\infty(0,T)$ is dense in $L^2(0,T)$, the above mentioned convergence holds for all $\psi \in L^2(0,T)$. In particular, $(\bm{y}^*,\phi)_{\mathbb{H}} \in L^2(0,T)$ for all $\phi \in L^2(0,T;\mathbb{H})$. 

For every $\phi\in L^2(0,T;\mathbb H)$, we have
\[
\begin{aligned}
&\int_0^T
\big((\mathfrak{H}_n(t),\bm{y}_n(t))_{\mathbb H}\bm{y}_n(t),
\phi(t)\big)_{\mathbb H}\,dt\\
&\quad=
\int_0^T
(\mathfrak{H}_n(t),\bm{y}_n(t))_{\mathbb H}
\big(\bm{y}_n(t)-\bm{y}^*(t),\phi(t)\big)_{\mathbb H}\,dt\\
&\qquad+
\int_0^T
\Big[
(\mathfrak{H}_n(t),\bm{y}_n(t))_{\mathbb H}
-(\mathfrak{H}^*(t),\bm{y}^*(t))_{\mathbb H}
\Big]
(\bm{y}^*(t),\phi(t))_{\mathbb H}\,dt.
\end{aligned}
\]
The first term tends to zero by the boundedness of $\{\mathfrak{H}_n\}$ in $L^2(0,T;\mathbb H)$ and the strong convergence $\bm{y}_n \to \bm{y}^*$ in $C([0,T];\mathbb H)$. The second term tends to zero by the weak convergence of $( \mathfrak{H}_n,\bm{y}_n)_{\mathbb H}$ in $L^2(0,T)$ and the fact that $( \bm{y}^*,\phi)_{\mathbb H} \in L^2(0,T)$. Therefore, we deduce
\[
( \mathfrak{H}_n,\bm{y}_n)_{\mathbb H}\, \bm{y}_n \rightharpoonup ( \mathfrak{H}^*,\bm{y}^*)_{\mathbb H}\, \bm{y}^* 
\quad \text{in } L^2(0,T;\mathbb H).
\]
Since the strong solution corresponding to the control
$\mathfrak{H}^*\in L^2(0,T;\mathbb H)$ is unique, the limit
$(\bm{y}^*,\theta^*)$ is the unique strong solution associated with $\mathfrak{H}^*$. Consequently,
$(\bm{y}^*,\theta^*,\mathfrak{H}^*)\in\mathcal A_{\rm ad}$.
Finally, since the cost functional $J$ is convex and continuous on 
\[
L^2(0,T;\mathbb{V}) \times L^2(0,T;V) \times \mathcal{U}_{ad}
\]
(see Propositions III.1.6 and III.1.10 in~\cite{EkelandTurnbull1983}), it is weakly lower semicontinuous. Therefore, for any minimizing sequence $(\bm{y}_n,\theta_n,\mathfrak{H}_n)$ converging weakly to $(\bm{y}^*,\theta^*,\mathfrak{H}^*)$, we have
\[
J(\bm{y}^*,\theta^*,\mathfrak{H}^*) \le \liminf_{n\to\infty} J(\bm{y}_n,\theta_n,\mathfrak{H}_n).
\]
Consequently, we derive 
\[
\inf J \le J(\bm{y}^*,\theta^*,\mathfrak{H}^*) \le \liminf_{n\to\infty} J(\bm{y}_n,\theta_n,\mathfrak{H}_n) = \lim_{n\to\infty} J(\bm{y}_n,\theta_n,\mathfrak{H}_n) = \inf J,
\]
which shows that $(\bm{y}^*,\theta^*,\mathfrak{H}^*) \in \mathcal{A}_{ad}$ is an optimal triplet for the control problem \eqref{opt}.
\end{proof}
\subsection{The Linearized Problem}{\label{s3.1}}
Now we study the linearized system and prove its global well-posedness. 
Let $(\bm{y}_{\mathfrak{H}},\theta_{\mathfrak{H}})$ be a strong solution of \eqref{eq:constrained_Boussinesq} corresponding to 
$\mathfrak{H} \in L^2(0,T;\mathbb{H})$, and let $\bm{W} \in L^2(0,T;\mathbb{H})$. If we choose $\bm{W} = \mathfrak{H} - ( \mathfrak{H},\bm{y}_{\mathfrak{H}})_{\mathbb{H}}\,\bm{y}_{\mathfrak{H}}$, then the corresponding functions $(\bm{\eta},\xi)$ 
belong to the tangent cone of admissible perturbations. We consider the linearized system:
\begin{equation}\label{linearized-system-a}
\left\{
\begin{aligned}
\frac{d\bm{\eta}(t)}{dt}
&= -\nu A\bm{\eta}(t) + \nu \|\nabla \bm{y}_{\mathfrak{H}}(t)\|_{\mathbb{H}}^{2}\bm{\eta}(t)
+ 2\nu ( \nabla \bm{y}_{\mathfrak{H}}(t), \nabla \bm{\eta}(t) )_{\mathbb{H}}\,\bm{y}_{\mathfrak{H}}(t)
- B'(\bm{y}_{\mathfrak{H}}(t))\eta(t) \\
&\quad + \beta {\textbf{Proj}_{\mathbb H}}(\xi(t) e_2)
- \beta\Big[( \xi(t) e_2, \bm{y}_{\mathfrak{H}}(t))_{\mathbb{H}}\bm{y}_{\mathfrak{H}}(t)
+ ( \theta_{\mathfrak{H}}(t) e_2, \bm{\eta}(t))_{\mathbb{H}}\bm{y}_{\mathfrak{H}}(t)
+ ( \theta_{\mathfrak{H}}(t) e_2, \bm{y}_{\mathfrak{H}}(t))_{\mathbb{H}}\bm{\eta}(t)\Big] \\
&\quad + \bm{W}(t) - \Big[( \mathfrak{H}(t),\bm{y}_{\mathfrak{H}}(t))_{\mathbb{H}}\bm{\eta}(t)
+ ( \mathfrak{H}(t),\bm{\eta}(t))_{\mathbb{H}}\bm{y}_{\mathfrak{H}}(t)\Big], \\
\bm{\eta}(0) &= \bm{\eta}_0,
\end{aligned}
\right.
\end{equation}
\begin{equation}\label{linearized-system-b}
\left\{
\begin{aligned}
\frac{d\xi(t)}{dt}
&= -\kappa A\xi(t) + \kappa \|\nabla \theta_{\mathfrak{H}}(t)\|_{\mathbb H}^{2}\xi(t)
+ 2\kappa ( \nabla \theta_{\mathfrak{H}}(t), \nabla \xi(t))_{\mathbb H}\,\theta_{\mathfrak{H}}(t)
- \bm{y}_{\mathfrak{H}}(t)\cdot \nabla \xi(t) - \bm{\eta}(t)\cdot \nabla \theta_{\mathfrak{H}}(t), \\
\xi(0) &= \xi_0,
\end{aligned}
\right.
\end{equation}
for a.e.\ $t\in[0,T]$, where $B'(\bm{y}_{\mathfrak{H}})\bm{\eta}= B(\bm{y}_{\mathfrak{H}},\bm{\eta}) + B(\bm{\eta},\bm{y}_{\mathfrak{H}})$. Let $\bm{\eta}_0\in\mathbb{H}$ and $\xi_0\in H$. We first derive a formal \emph{a priori} estimate for \eqref{linearized-system-a}--\eqref{linearized-system-b}. A rigorous justification follows from the Faedo--Galerkin method. Taking the $\mathbb{H}$-inner product of \eqref{linearized-system-a} with $\bm{\eta}(\cdot)$, and using $( B(\bm{y}_{\mathfrak{H}},\bm{\eta}),\bm{\eta})_{\mathbb{H}}=0$ and $\bm{y}_{\mathfrak{H}}(t)\in\mathcal{M}_1$ for all $t\in[0,T]$, we obtain
\begin{align*}
\frac{1}{2}\frac{d}{dt}\|\bm{\eta}(t)\|_{\mathbb{H}}^2 + \nu \|\nabla \bm{\eta}(t)\|_{\mathbb{H}}^2
&= \nu \|\nabla \bm{y}_{\mathfrak{H}}(t)\|_{\mathbb{H}}^2 \|\bm{\eta}(t)\|_{\mathbb{H}}^2
+ 2\nu ( \nabla \bm{y}_{\mathfrak{H}}(t),\nabla \bm{\eta}(t))_{\mathbb{H}} ( \bm{y}_{\mathfrak{H}}(t),\bm{\eta}(t))_{\mathbb{H}} \\
&\quad - ( B(\bm{\eta}(t),\bm{y}_{\mathfrak{H}}(t)),\bm{\eta}(t))_{\mathbb{H}} \\
&\quad + \beta\Big[( {\textbf{Proj}_{\mathbb H}}(\xi(t) e_2),\bm{\eta}(t))_{\mathbb{H}}
- ( \xi(t) e_2,\bm{y}_{\mathfrak{H}}(t))_{\mathbb{H}} ( \bm{y}_{\mathfrak{H}}(t),\bm{\eta}(t))_{\mathbb{H}} \\
&\qquad\quad - ( \theta_{\mathfrak{H}}(t) e_2,\bm{\eta}(t))_{\mathbb{H}} ( \bm{y}_{\mathfrak{H}}(t),\bm{\eta}(t))_{\mathbb{H}}
- ( \theta_{\mathfrak{H}}(t) e_2,\bm{y}_{\mathfrak{H}}(t))_{\mathbb{H}} \|\bm{\eta}(t)\|_{\mathbb{H}}^2\Big] \\
&\quad + ( \bm{W}(t),\bm{\eta}(t))_{\mathbb{H}}
- ( \mathfrak{H}(t),\bm{y}_{\mathfrak{H}}(t))_{\mathbb{H}} \|\bm{\eta}(t)\|_{\mathbb{H}}^2
- ( \mathfrak{H}(t),\bm{\eta}(t))_{\mathbb{H}} ( \bm{y}_{\mathfrak{H}}(t),\bm{\eta}(t))_{\mathbb{H}},
\end{align*}
for a.e. $t\in[0,T]$. By the Cauchy--Schwarz and Young inequalities, we have
\begin{align*}
2\nu ( \nabla \bm{y}_{\mathfrak{H}},\nabla \bm{\eta})_{\mathbb{H}} ( \bm{y}_{\mathfrak{H}},\bm{\eta})_{\mathbb{H}}
&\le 2\nu \|\nabla \bm{y}_{\mathfrak{H}}\|_{\mathbb{H}} \|\nabla \bm{\eta}\|_{\mathbb{H}} \|\bm{\eta}\|_{\mathbb{H}} \le \frac{\nu}{6}\|\nabla \bm{\eta}\|_{\mathbb{H}}^2
+ C_{\nu} \|\nabla \bm{y}_{\mathfrak{H}}\|_{\mathbb{H}}^2 \|\bm{\eta}\|_{\mathbb{H}}^2.
\end{align*}
Using the properties of the trilinear map \eqref{eq:b_refined} and Ladyzhenskaya's inequality, we deduce
\begin{align*}
|( B(\bm{\eta},\bm{y}_{\mathfrak{H}}),\bm{\eta})_{\mathbb{H}}|
&\le \|\bm{\eta}\|_{\mathbb L^4}^2 \|\nabla \bm{y}_{\mathfrak{H}}\|_{\mathbb{H}} \\
&\le C \|\bm{\eta}\|_{\mathbb{H}} \|\nabla \bm{\eta}\|_{\mathbb{H}} \|\nabla \bm{y}_{\mathfrak{H}}\|_{\mathbb{H}} \\
&\le \frac{\nu}{6}\|\nabla \bm{\eta}\|_{\mathbb{H}}^2
+ C_{\nu}\|\nabla \bm{y}_{\mathfrak{H}}\|_{\mathbb{H}}^2 \|\bm{\eta}\|_{\mathbb{H}}^2.
\end{align*}
Using $(\bm{y}_{\mathfrak{H}},\theta_{\mathfrak{H}})\in\mathcal M$, together with H\"older's and
Young's inequalities, we infer
\[
\begin{aligned}
&|( \bm{W},\bm{\eta})_{\mathbb H}|
+\beta|( {\textbf{Proj}}_{\mathbb H}(\xi e_2),\bm{\eta})_{\mathbb H}|
+\beta|( \xi e_2,\bm{y}_{\mathfrak{H}})_{\mathbb H}( \bm{y}_{\mathfrak{H}},\bm{\eta})_{\mathbb H}| \\
&\quad
+\beta|( \theta_{\mathfrak{H}} e_2,\bm{\eta})_{\mathbb H}( \bm{y}_{\mathfrak{H}},\bm{\eta})_{\mathbb H}|
+\beta|( \theta_{\mathfrak{H}} e_2,\bm{y}_{\mathfrak{H}})_{\mathbb H}|\,\|\bm{\eta}\|_{\mathbb H}^2 \\
&\quad
+|( \mathfrak{H},\bm{y}_{\mathfrak{H}})_{\mathbb H}|\,\|\bm{\eta}\|_{\mathbb H}^2
+|( \mathfrak{H},\bm{\eta})_{\mathbb H}( \bm{y}_{\mathfrak{H}},\bm{\eta})_{\mathbb H}| \\
&\le
\frac12\|\bm{W}\|_{\mathbb H}^2
+\frac12\|\bm{\eta}\|_{\mathbb H}^2
+2\beta\|\xi\|_{H}\|\bm{\eta}\|_{\mathbb H}
+2\beta\|\bm{\eta}\|_{\mathbb H}^2
+2\|\mathfrak{H}\|_{\mathbb H}\|\bm{\eta}\|_{\mathbb H}^2 \\
&\le
\frac12\|\bm{W}\|_{\mathbb H}^2
+\beta^2\|\xi\|_{H}^2 
+\left(\frac32+2\beta+2\|\mathfrak{H}\|_{\mathbb H}\right)\|\bm{\eta}\|_{\mathbb H}^2.
\end{aligned}
\]
Combining the above estimates, we obtain
\begin{equation}\label{etaf-estimate}
\frac{1}{2}\frac{d}{dt}\|\bm{\eta}(t)\|_{\mathbb{H}}^2 + \frac{2\nu}{3}\|\nabla \bm{\eta}(t)\|_{\mathbb{H}}^2
\le C_{\nu,\beta}\left(1 + \|\nabla \bm{y}_{\mathfrak{H}}(t)\|_{\mathbb{H}}^2 + \|\mathfrak{H}(t)\|_{\mathbb{H}}\right)\|\bm{\eta}(t)\|_{\mathbb{H}}^2
+ C_{\beta}\|\xi(t)\|_{H}^2 + \frac12\|\bm{W}(t)\|_{\mathbb{H}}^2,
\end{equation}
where \(C_{\nu,\beta}, C_{\beta} > 0\) are constants depending on \(\nu\) and \(\beta\).

Now taking the \(L^2\)-inner product of \eqref{linearized-system-b} with \(\xi(\cdot)\), we obtain
\begin{align}
\frac{1}{2}\frac{d}{dt}\|\xi(t)\|_{H}^2
&= -\kappa \|\nabla \xi(t)\|_{\mathbb H}^2
+ \kappa \|\nabla \theta_{\mathfrak{H}}(t)\|_{\mathbb H}^2 \|\xi(t)\|_{H}^2  + 2\kappa ( \nabla \theta_{\mathfrak{H}}(t),\nabla \xi(t))_{H} ( \theta_{\mathfrak{H}}(t),\xi(t))_{H}
\nonumber\\&\quad- ( \bm{\eta}(t)\cdot\nabla \theta_{\mathfrak{H}}(t),\xi(t))_{H},
\label{xi-step-1}
\end{align}
for a.e. $t\in(0,T)$, since \(( \bm{y}_{\mathfrak{H}}\cdot\nabla \xi,\xi)_{H}=0\).
Using Hölder's, Ladyzhenskaya's, and Young's inequalities, we estimate
\begin{align}
\left|(\bm{\eta}\cdot\nabla\theta_{\mathfrak{H}},\xi)_{H}\right|
&=\left|(\bm{\eta}\cdot\nabla\xi,\theta_{\mathfrak{H}})_{H}\right|
\nonumber\\
&\leq \|\bm{\eta}\|_{\mathbb{L}^{4}}
\|\nabla\xi\|_{H}
\|\theta_{\mathfrak{H}}\|_{L^{4}}
\nonumber\\
&\leq \frac{\kappa}{4}\|\nabla\xi\|_{H}^{2}
+C_{\kappa}\|\bm{\eta}\|_{\mathbb{H}}
\|\nabla\bm{\eta}\|_{\mathbb{H}}
\|\theta_{\mathfrak{H}}\|_{L^{4}}^{2}
\nonumber\\
&\leq \frac{\kappa}{4}\|\nabla\xi\|_{H}^{2}
+\frac{\nu}{6}\|\nabla\bm{\eta}\|_{\mathbb{H}}^{2}
+C_{\nu,\kappa}\|\theta_{\mathfrak{H}}\|_{L^{4}}^{4}
\|\bm{\eta}\|_{\mathbb{H}}^{2},
\label{xi-term-1}\\
2\kappa|(\nabla\theta_{\mathfrak{H}},\nabla\xi)_{H}|
(\theta_{\mathfrak{H}},\xi)_{H}
&\leq 2\kappa\|\nabla\theta_{\mathfrak{H}}\|_{\mathbb H}
\|\nabla\xi\|_{\mathbb H}\|\xi\|_{H}
\nonumber\\
&\leq \frac{\kappa}{4}\|\nabla\xi\|_{H}^{2}
+C_{\kappa}\|\nabla\theta_{\mathfrak{H}}\|_{\mathbb H}^{2}
\|\xi\|_{H}^{2}.
\label{xi-term-2}
\end{align}
Combining the estimates \eqref{xi-term-1} and \eqref{xi-term-2}, we get for a.e. \(t\in[0,T]\)
\begin{equation}\label{xi-estimate}
\frac{1}{2}\frac{d}{dt}\|\xi(t)\|_{H}^2 + \frac{\kappa}{2}\|\nabla \xi(t)\|_{\mathbb H}^2
\le C_{\kappa}\left(1 + \|\nabla \theta_{\mathfrak{H}}(t)\|_{\mathbb H}^2\right)\|\xi(t)\|_{H}^2
+ \frac{\nu}{6}\|\nabla\bm{\eta}(t)\|_{\mathbb H}^2
+ C_{\nu,\kappa}\|\theta_{\mathfrak{H}}(t)\|_{L^4}^4\|\bm{\eta}(t)\|_{\mathbb H}^2.
\end{equation}
Adding \eqref{etaf-estimate} and \eqref{xi-estimate}, we obtain
\begin{align}
&\frac{d}{dt}\Big(\|\bm{\eta}(t)\|_{\mathbb{H}}^2 + \|\xi(t)\|_{H}^2\Big)
+ \nu\|\nabla \bm{\eta}(t)\|_{\mathbb{H}}^2
+ \kappa\|\nabla \xi(t)\|_{\mathbb H}^2 \notag\\
&\quad\le C_{\nu,\kappa,\beta}\Big(1 +\|\nabla\bm{y}_{\mathfrak{H}}(t)\|_{\mathbb H}^2+ \|\nabla \theta_{\mathfrak{H}}(t)\|_{\mathbb H}^2+\|\theta_{\mathfrak{H}}(t)\|_{L^4}^4+\|\mathfrak{H}(t)\|_{\mathbb{H}}\Big)
\Big(\|\bm{\eta}(t)\|_{\mathbb{H}}^2 + \|\xi(t)\|_{H}^2\Big) + \|\bm{W}(t)\|_{\mathbb{H}}^2,
\label{coupled-final}
\end{align}
for a.e. \(t\in[0,T]\) and \(C_{\nu,\kappa,\beta}>0\) is a constant depending on \(\nu\), \(\kappa\), and \(\beta\). Let us define for a.e. $t\in(0,T)$
\[
A(t):=C_{\nu,\kappa,\beta}
\Big(
1+\|\nabla\bm{y}_{\mathfrak{H}}(t)\|_{\mathbb H}^2
+\|\nabla\theta_{\mathfrak{H}}(t)\|_{\mathbb H}^2
+\|\theta_{\mathfrak{H}}(t)\|_{L^4}^4
+\|\mathfrak{H}(t)\|_{\mathbb H}
\Big).
\]
Then, from \eqref{coupled-final}, we have for a.e. \(t\in[0,T]\)
\begin{equation}\label{coupled-ineq}
\frac{d}{dt}\Big(\|\bm{\eta}(t)\|_{\mathbb{H}}^2 + \|\xi(t)\|_{H}^2\Big)
+ \nu\|\nabla \bm{\eta}(t)\|_{\mathbb{H}}^2
+ \kappa\|\nabla \xi(t)\|_{\mathbb H}^2
\le A(t)\Big(\|\bm{\eta}(t)\|_{\mathbb{H}}^2 + \|\xi(t)\|_{H}^2\Big)
+ \|\bm{W}(t)\|_{\mathbb{H}}^2.
\end{equation}
Integrating \eqref{coupled-ineq} from $0$ to $t$, we obtain
\[
\begin{aligned}
&\|\bm{\eta}(t)\|_{\mathbb{H}}^2 + \|\xi(t)\|_{H}^2
+ \nu\int_0^t \|\nabla \bm{\eta}(s)\|_{\mathbb{H}}^2\,ds
+ \kappa\int_0^t \|\nabla \xi(s)\|_{H}^2\,ds \\
&\le \|\bm{\eta}(0)\|_{\mathbb{H}}^2 + \|\xi(0)\|_{H}^2
+ \int_0^t A(s)\left(\|\bm{\eta}(s)\|_{\mathbb{H}}^2 + \|\xi(s)\|_{H}^2\right)ds
+ \int_0^t \|\bm{W}(s)\|_{\mathbb{H}}^2\,ds.
\end{aligned}
\]
Applying Grönwall's inequality to \eqref{coupled-ineq}, it follows that, for all $t\in[0,T]$, 
\begin{align}
\|\bm{\eta}(t)\|_{\mathbb{H}}^2 + \|\xi(t)\|_{H}^2
&\le \Big(\|\bm{\eta}(0)\|_{\mathbb{H}}^2 + \|\xi(0)\|_{H}^2
+ \int_0^t \|\bm{W}(s)\|_{\mathbb{H}}^2\,ds\Big) \exp\Big(\int_0^t A(s)\,ds\Big).
\label{final-bound}
\end{align}
Moreover, for all $\phi \in L^2(0,T;\mathbb{V})$, we have from \eqref{linearized-system-a} that 
\[
\begin{aligned}
&\left|
\int_0^T
{}_{\mathbb{V}'}\langle
{\partial_t\bm{\eta}(t)},
\phi(t)\rangle_{\mathbb{V}}
\,dt
\right| \\
&=
\Bigg|
\int_0^T
\Big(
-\nu A\bm{\eta}(t) + \nu \|\nabla \bm{y}_{\mathfrak{H}}(t)\|_{\mathbb{H}}^{2}\bm{\eta}(t)
+ 2\nu ( \nabla \bm{y}_{\mathfrak{H}}(t), \nabla \bm{\eta}(t) )_{\mathbb{H}}\,\bm{y}_{\mathfrak{H}}(t)
- B'(\bm{y}_{\mathfrak{H}}(t))\bm{\eta}(t) \\
&\quad + \beta {\textbf{Proj}_{\mathbb H}}(\xi(t) e_2)
- \beta\Big(( \xi(t) e_2, \bm{y}_{\mathfrak{H}}(t))_{\mathbb{H}}\bm{y}_{\mathfrak{H}}(t)
+ ( \theta_{\mathfrak{H}}(t) e_2, \bm{\eta}(t))_{\mathbb{H}}\bm{y}_{\mathfrak{H}}(t)
+ ( \theta_{\mathfrak{H}}(t) e_2, \bm{y}_{\mathfrak{H}}(t))_{\mathbb{H}}\bm{\eta}(t)\Big) \\
&\quad \left. + \bm{W}(t) - \Big(( \mathfrak{H}(t),\bm{y}_{\mathfrak{H}}(t))_{\mathbb{H}}\bm{\eta}(t)
+ ( \mathfrak{H}(t),\bm{\eta}(t))_{\mathbb{H}}\bm{y}_{\mathfrak{H}}(t)\Big),
\phi(t)
\right)_{\mathbb{H}}\,dt
\Bigg|.
\end{aligned}
\]
Using the Cauchy-Schwarz inequality, the Poincaré-Wirtinger inequality, the uniform bounds on $\bm{y}_{\mathfrak{H}}$ and $\theta_{\mathfrak{H}}$, and the facts that $\bm{y}_{\mathfrak{H}}\in\mathcal{M}_{1},\theta_{\mathfrak{H}}\in\mathcal M_2$, we estimate
\begin{align*}
&\left|
\int_0^T
{}_{\mathbb{V}'}\langle
{\partial_t\bm{\eta}(t)},
\phi(t)\rangle_{\mathbb{V}}
\,dt
\right| \\
&\le C_{\nu,\beta,\lambda_1,T}\Bigg[
\|\nabla \bm{\eta}\|_{L^2(0,T;\mathbb{H})}
+ \|\nabla \bm{y}_{\mathfrak{H}}\|_{L^\infty(0,T;\mathbb{H})}^2 \|\bm{\eta}\|_{L^\infty(0,T;\mathbb{H})} \\
&\quad + \|\nabla \bm{y}_{\mathfrak{H}}\|_{L^\infty(0,T;\mathbb{H})}\|\nabla\bm{\eta}\|_{L^2(0,T;\mathbb{H})}
+ \|\bm{y}_{\mathfrak{H}}\|_{L^{\infty}(0,T;\mathbb{V})} \|\nabla\bm{\eta}\|_{L^2(0,T;\mathbb{H})}  + \|\bm{\eta}\|_{L^4(0,T;\mathbb{L}^4)} \|\nabla\bm{y}_{\mathfrak{H}}\|_{L^4(0,T;\mathbb{L}^4)} \\
&\quad + \|\bm{W}\|_{L^2(0,T;\mathbb{H})}
+ \|\mathfrak{H}\|_{L^2(0,T;\mathbb{H})} \|\bm{\eta}\|_{L^\infty(0,T;\mathbb{H})}
+ \|\xi\|_{L^2(0,T;H)}
+ \|\bm{\eta}\|_{L^2(0,T;\mathbb{H})}
\Bigg] \|\phi\|_{L^2(0,T;\mathbb{V})},
\end{align*}
where $\lambda_1$ is the first eigenvalue of the Stokes operator.
Hence, $\frac{d\bm{\eta}}{dt} \in L^2(0,T;\mathbb{V}')$, and by the Lions--Magenes lemma, $\bm{\eta} \in C([0,T];\mathbb{H})$.
Similarly, for all $\psi \in L^2(0,T;V)$, we have from \eqref{linearized-system-b} that 
\begin{align*}
&\left|
\int_0^T
{}_{V'}\langle
\partial_t\xi(t),\psi(t)\rangle_V
\,dt
\right| \\
&=
\left|
\int_0^T
\Big(
-\kappa A\xi(t)
+\kappa\|\nabla\theta_{\mathfrak{H}}(t)\|_{H}^{2}\xi(t)
+2\kappa(\nabla\theta_{\mathfrak{H}}(t),\nabla\xi(t))_{H}
\,\theta_{\mathfrak{H}}(t)
-\bm{y}_{\mathfrak{H}}(t)\cdot\nabla\xi(t)
-\bm{\eta}(t)\cdot\nabla\theta_{\mathfrak{H}}(t),
\psi(t)
\Big)_{H}
\,dt
\right| \\
&\leq C_{\kappa,\mu_1,T}\Bigg[
\|\nabla\xi\|_{L^2(0,T;H)}
+\|\nabla\theta_{\mathfrak{H}}\|_{L^\infty(0,T;H)}^2
\|\xi\|_{L^\infty(0,T;H)}
+\|\nabla\theta_{\mathfrak{H}}\|_{L^\infty(0,T;H)}
\|\nabla\xi\|_{L^2(0,T;H)}
\\
&\quad+\|\bm{y}_{\mathfrak{H}}\|_{L^{\infty}(0,T;\mathbb{V})}
\|\nabla\xi\|_{L^2(0,T;H)}
+\|\bm{\eta}\|_{L^4(0,T;\mathbb{L}^4)}
\|\nabla\theta_{\mathfrak{H}}\|_{L^4(0,T;L^4)}
\Bigg]
\|\psi\|_{L^2(0,T;V)},
\end{align*}
where \(\mu_1\) is the first eigenvalue of the operator \(\mathcal{A}\) (see Section 2 for the definition of the operator).
Hence,
\(\displaystyle \frac{d\xi}{dt} \in L^2(0,T;V') \  \text{and hence }  \ \xi \in C([0,T];H)\).
The preceding analysis leads to the following well-posedness result (since the problem is linear) for weak solutions to problem \eqref{linearized-system-a}--\eqref{linearized-system-b}.

\begin{theorem}
Let \(\bm{\eta}_0 \in \mathbb{H}\), \(\xi_0 \in H\), and \(\bm{W}, \mathfrak{H} \in L^2(0,T;\mathbb{H})\) be given. Thus, we obtain a weak solution $(\bm{\eta},\xi)$, which is unique and satisfies 
\(\bm{\eta} \in L^\infty(0,T;\mathbb{H}) \cap L^2(0,T;\mathbb{V})\), 
\(\xi \in L^\infty(0,T;H) \cap L^2(0,T;V)\), with 
\(\frac{d\bm{\eta}}{dt} \in L^2(0,T;\mathbb{V}')\), 
\(\frac{d\xi}{dt} \in L^2(0,T;V')\), and hence 
\(\bm{\eta} \in C([0,T];\mathbb{H})\) and \(\xi \in C([0,T];H)\).
\allowdisplaybreaks
Moreover, for all \(t\in[0,T]\), we have the following energy identities:
\begin{align}
&\|\bm{\eta}(t)\|_{\mathbb{H}}^2 + 2\nu \int_0^t \|\nabla \bm{\eta}(s)\|_{\mathbb{H}}^2 ds
\nonumber\\&= \|\bm{\eta}_0\|_{\mathbb{H}}^2 
+ 2\nu \int_0^t \|\nabla \bm{y}_{\mathfrak{H}}(s)\|_{\mathbb{H}}^2 \|\bm{\eta}(s)\|_{\mathbb{H}}^2 ds + 4\nu \int_0^t ( \nabla \bm{y}_{\mathfrak{H}}(s), \nabla \bm{\eta}(s) )_{\mathbb{H}}
( \bm{y}_{\mathfrak{H}}(s), \bm{\eta}(s) )_{\mathbb{H}} ds
\notag\\
&\quad - 2 \int_0^t ( B(\bm{\eta}(s),\bm{y}_{\mathfrak{H}}(s)), \bm{\eta}(s))_{\mathbb{H}} ds  + 2\beta \int_0^t ( \xi(s) e_2, \bm{\eta}(s) )_{\mathbb{H}} ds
\notag\\
&\quad- 2\beta \int_0^t \Big[
( \xi(s) e_2,\bm{y}_{\mathfrak{H}}(s))_{\mathbb{H}} ( \bm{y}_{\mathfrak{H}}(s),\bm{\eta}(s))_{\mathbb{H}}  + ( \theta_{\mathfrak{H}}(s) e_2,\bm{\eta}(s))_{\mathbb{H}} ( \bm{y}_{\mathfrak{H}}(s),\bm{\eta}(s))_{\mathbb{H}}
+ ( \theta_{\mathfrak{H}}(s) e_2,\bm{y}_{\mathfrak{H}}(s))_{\mathbb{H}} \|\bm{\eta}(s)\|_{\mathbb{H}}^2
\Big] ds \notag\\
&\quad + 2 \int_0^t ( \bm{W}(s), \bm{\eta}(s) )_{\mathbb{H}} ds - 2 \int_0^t \Big[
( \mathfrak{H}(s),\bm{y}_{\mathfrak{H}}(s))_{\mathbb{H}} \|\bm{\eta}(s)\|_{\mathbb{H}}^2  + ( \mathfrak{H}(s),\bm{\eta}(s))_{\mathbb{H}} ( \bm{y}_{\mathfrak{H}}(s),\bm{\eta}(s))_{\mathbb{H}}
\Big] ds,
\label{energy-eta-correct}
\end{align}
and
\begin{align}
&\|\xi(t)\|_{H}^2 + 2\kappa \int_0^t \|\nabla \xi(s)\|_{H}^2 ds
\nonumber\\&= \|\xi_0\|_{H}^2 
+ 2\kappa \int_0^t \|\nabla \theta_{\mathfrak{H}}(s)\|_{\mathbb H}^2 \|\xi(s)\|_{H}^2 ds  + 4\kappa \int_0^t ( \nabla \theta_{\mathfrak{H}}(s), \nabla \xi(s) )_{\mathbb H}
( \theta_{\mathfrak{H}}(s), \xi(s) )_{H} ds
\notag\\
&\quad- 2 \int_0^t (\bm{\eta}(s)\cdot \nabla \theta_{\mathfrak{H}}(s), \xi(s))_{H}\, ds,
\label{energy-xi-correct}
\end{align}
for all \(t\in[0,T].\)
Consequently, we obtain the estimate given below:
\begin{align}
\sup_{t\in[0,T]} \big(\|\bm{\eta}(t)\|_{\mathbb{H}}^2 + \|\xi(t)\|_{H}^2\big)
&+ \nu \int_0^T \|\nabla \bm{\eta}(t)\|_{\mathbb{H}}^2 dt
+ \kappa \int_0^T \|\nabla \xi(t)\|_{\mathbb H}^2 dt \notag\\
&\le C_{\nu,\kappa,\beta,T}\left( \|\bm{\eta}_0\|_{\mathbb{H}}^2 + \|\xi_0\|_{H}^2
+ \int_0^T \|\bm{W}(t)\|_{\mathbb{H}}^2 dt \right).
\end{align}
\end{theorem}
\begin{remark}
In the linearized problem, we may choose
\[
\bm{W}  = \mathfrak{H} - ( \mathfrak{H},\bm{y}_{\mathfrak{H}} )_{\mathbb{H}}\,\bm{y}_{\mathfrak{H}} ,
\ \  \mathfrak{H} \in L^2(0,T;\mathbb{H}).
\]
Since \(\bm{y}_{\mathfrak{H}}(t) \in \mathcal{M}_1\) for all \(t \in [0,T]\), we have
\[
\|\bm{W}(t)\|_{\mathbb{H}}^2 
= \|\mathfrak{H}(t)\|_{\mathbb{H}}^2 - |( \mathfrak{H}(t),\bm{y}_{\mathfrak{H}}(t))_{\mathbb{H}}|^2 
\le \|\mathfrak{H}(t)\|_{\mathbb{H}}^2,
\]
and consequently,
\[
\int_0^T \|\bm{W}(t)\|_{\mathbb{H}}^2 dt 
\le \int_0^T \|\mathfrak{H}(t)\|_{\mathbb{H}}^2 dt < \infty.
\]
Thus, \(\bm{W} \in L^2(0,T;\mathbb{H})\), which shows that the problem \eqref{linearized-system-a}--\eqref{linearized-system-b} admits a weak solution .
\end{remark}
\subsection{The adjoint system.}{\label{s3.2}}
It is necessary to introduce the adjoint system corresponding to the solutions of the Boussinesq system \eqref{eq:constrained_Boussinesq} in order to derive the necessary condition for first-order optimality,
which constitutes the main result of this work. To this end, let \((\bm{\vartheta}(\cdot), q(\cdot))\) be the adjoint variables given by the solution to the following system:
\begin{equation}{\label{adjj}}
\begin{cases}
-\dfrac{d\bm{\vartheta}(t)}{dt}
= -\nu A \bm{\vartheta}(t) + \nu \|\nabla \bm{y}_{\mathfrak{H}}(t)\|_{\mathbb{H}}^2 \bm{\vartheta}(t)
+ 2\nu ( \bm{y}_{\mathfrak{H}}(t),\bm{\vartheta}(t))_{\mathbb{H}}\, A \bm{y}_{\mathfrak{H}}(t) - (B'(\bm{y}_{\mathfrak{H}}(t)))^* \bm{\vartheta}(t) \\
\qquad - \beta ( \theta_{\mathfrak{H}}(t) e_2, \bm{y}_{\mathfrak{H}}(t))_{\mathbb{H}}\, \bm{\vartheta}(t)
- \beta ( \bm{y}_{\mathfrak{H}}(t),\bm{\vartheta}(t))_{\mathbb{H}}\,\theta_{\mathfrak{H}}(t) e_2
- ( \mathfrak{H}(t),\bm{y}_{\mathfrak{H}}(t))_{\mathbb{H}}\, \bm{\vartheta}(t) \\
\qquad - ( \bm{y}_{\mathfrak{H}}(t),\bm{\vartheta}(t))_{\mathbb{H}}\,\mathfrak{H}(t)
+ (\bm{y}_{\mathfrak{H}}(t)-\bm{y}_d) + A\bm{y}_{\mathfrak{H}}(t) - (\nabla \theta_{\mathfrak{H}}(t))^{\top} q(t), 
\quad \text{in } \mathbb{V}', \\[4pt]
-\dfrac{dq(t)}{dt}
= -\kappa {\mathcal{A} q(t)} + \kappa \|\nabla \theta_{\mathfrak{H}}(t)\|_{\mathbb H}^2 q(t)
+ 2\kappa ( \theta_{\mathfrak{H}}(t),q(t))_{H}\, \mathcal{A}\theta_{\mathfrak{H}}(t) + \bm{y}_{\mathfrak{H}}(t)\cdot\nabla q(t) \\
\qquad + \beta (\bm{\vartheta}(t)\cdot e_2) - \beta ( \bm{y}_{\mathfrak{H}}(t),\bm{\vartheta}(t))_{\mathbb{H}} (\bm{y}_{\mathfrak{H}}(t)\cdot e_2)
+ (\theta_{\mathfrak{H}}(t)-\theta_d),
\quad \text{in } V', \\[4pt]
\bm{\vartheta}(T) = \bm{y}_{\mathfrak{H}}(T) - \bm{y}_f, \ \  q(T) = \theta_{\mathfrak{H}}(T) - \theta_f,
\end{cases}
\end{equation} 
for a.e.\ $t \in [0,T]$, where $((B'(\bm{y}_{\mathfrak{H}}))^* \bm{\vartheta},\bm{k})_{\mathbb{H}}=(\bm{\vartheta},B'(\bm{y}_{\mathfrak{H}})\bm{k})_{\mathbb{H}}=(\bm{\vartheta},B(\bm{y}_{\mathfrak{H}},\bm{k}))_{\mathbb{H}}+(\bm{\vartheta},B(\bm{k},\bm{y}_{\mathfrak{H}}))_{\mathbb{H}}$.
We first establish an \emph{a priori} energy estimate for \(\bm{\vartheta}\). Taking the \(\mathbb{H}\)-inner product of the first adjoint equation with \(\bm{\vartheta}(\cdot)\), and using
$(B(\bm{y}_{\mathfrak{H}},\bm{\vartheta}),\bm{\vartheta})_{\mathbb{H}}=0$ together with \((\bm{y}_{\mathfrak{H}},\theta_{\mathfrak{H}})\in\mathcal{M}\), we obtain
\begin{align}
&-\frac{1}{2}\frac{d}{dt}\|\bm{\vartheta}(t)\|_{\mathbb{H}}^2 + \nu \|\nabla \bm{\vartheta}(t)\|_{\mathbb{H}}^2\nonumber
\\&= \nu \|\nabla \bm{y}_{\mathfrak{H}}(t)\|_{\mathbb{H}}^2 \|\bm{\vartheta}(t)\|_{\mathbb{H}}^2
+ 2\nu ( \bm{y}_{\mathfrak{H}}(t),\bm{\vartheta}(t))_{\mathbb{H}} ( \nabla \bm{y}_{\mathfrak{H}}(t),\nabla \bm{\vartheta}(t))_{\mathbb{H}}  - ( B(\bm{\vartheta}(t),\bm{y}_{\mathfrak{H}}(t)),\bm{\vartheta}(t))_{\mathbb{H}}
\notag\\
&\quad- \beta ( \theta_{\mathfrak{H}}(t) e_2, \bm{y}_{\mathfrak{H}}(t))_{\mathbb{H}} \|\bm{\vartheta}(t)\|_{\mathbb{H}}^2  - \beta ( \bm{y}_{\mathfrak{H}}(t),\bm{\vartheta}(t))_{\mathbb{H}} ( \theta_{\mathfrak{H}}(t) e_2,\bm{\vartheta}(t))_{\mathbb{H}}
- ( \mathfrak{H}(t),\bm{y}_{\mathfrak{H}}(t))_{\mathbb{H}} \|\bm{\vartheta}(t)\|_{\mathbb{H}}^2 \notag\\
&\quad -( \bm{y}_{\mathfrak{H}}(t),\bm{\vartheta}(t))_{\mathbb{H}} ( \mathfrak{H}(t),\bm{\vartheta}(t))_{\mathbb{H}}
+ ( \bm{y}_{\mathfrak{H}}(t)-\bm{y}_d,\bm{\vartheta}(t))_{\mathbb{H}} +  ( A\bm{y}_{\mathfrak{H}}(t),\bm{\vartheta}(t))_{\mathbb{H}}
- ( (\nabla \theta_{\mathfrak{H}}(t))^{\top} q(t), \bm{\vartheta}(t))_{\mathbb{H}},
\label{energy_p_1}
\end{align}
for a.e. \(t\in[0,T].\)
We next derive estimates for each term appearing on the right-hand side.
Using $\|\bm{y}_{\mathfrak{H}}\|_{\mathbb H}=1$, together with
the Cauchy--Schwarz and Young inequalities, we have
\begin{align}{\label{1st}}
2\nu\left|
(\bm{y}_{\mathfrak{H}},\bm{\vartheta})_{\mathbb H}
(\nabla\bm{y}_{\mathfrak{H}},\nabla\bm{\vartheta})_{\mathbb H}
\right|
&\le
2\nu\|\bm{\vartheta}\|_{\mathbb H}
\|\nabla\bm{\vartheta}\|_{\mathbb H}
\|\nabla\bm{y}_{\mathfrak{H}}\|_{\mathbb H}\le
\frac{\nu}{6}\|\nabla\bm{\vartheta}\|_{\mathbb H}^2
+C_\nu\|\nabla\bm{y}_{\mathfrak{H}}\|_{\mathbb H}^2
\|\bm{\vartheta}\|_{\mathbb H}^2.
\end{align}
Using the properties of the trilinear form \eqref{trii},
Ladyzhenskaya's inequality, and Young's inequality, we get
\begin{align}{\label{2nd}}
|(B(\bm{\vartheta},\bm{y}_{\mathfrak{H}}),\bm{\vartheta})_{\mathbb H}|
&\le
C\|\bm{\vartheta}\|_{\mathbb L^4}^2
\|\nabla\bm{y}_{\mathfrak{H}}\|_{\mathbb H}\notag\\
&\le
C\|\bm{\vartheta}\|_{\mathbb H}
\|\nabla\bm{\vartheta}\|_{\mathbb H}
\|\nabla\bm{y}_{\mathfrak{H}}\|_{\mathbb H}\notag\\
&\le
\frac{\nu}{6}\|\nabla\bm{\vartheta}\|_{\mathbb H}^2
+C_\nu\|\nabla\bm{y}_{\mathfrak{H}}\|_{\mathbb H}^2
\|\bm{\vartheta}\|_{\mathbb H}^2.
\end{align}
Using
$\|\bm{y}_{\mathfrak{H}}\|_{\mathbb H}=\|\theta_{\mathfrak{H}}\|_H=1$,
we estimate the coupling, control, and cost terms as follows:
\begin{align}{\label{3rd}}
&-\beta(\theta_{\mathfrak{H}}e_2,\bm{y}_{\mathfrak{H}})_{\mathbb H}
\|\bm{\vartheta}\|_{\mathbb H}^2
-\beta(\bm{y}_{\mathfrak{H}},\bm{\vartheta})_{\mathbb H}
(\theta_{\mathfrak{H}}e_2,\bm{\vartheta})_{\mathbb H}\notag\\
&\quad
-(\mathfrak{H},\bm{y}_{\mathfrak{H}})_{\mathbb H}\|\bm{\vartheta}\|_{\mathbb H}^2
-(\bm{y}_{\mathfrak{H}},\bm{\vartheta})_{\mathbb H}
(\mathfrak{H},\bm{\vartheta})_{\mathbb H}\notag\\
&\quad
+(\bm{y}_{\mathfrak{H}}-\bm{y}_d,\bm{\vartheta})_{\mathbb H}
+(A\bm{y}_{\mathfrak{H}},\bm{\vartheta})_{\mathbb H}
-((\nabla\theta_{\mathfrak{H}})^\top q,\bm{\vartheta})_{\mathbb H}\notag\\
&\le
2|\beta|\|\bm{\vartheta}\|_{\mathbb H}^2
+2\|\mathfrak{H}\|_{\mathbb H}\|\bm{\vartheta}\|_{\mathbb H}^2
+\|\bm{y}_{\mathfrak{H}}-\bm{y}_d\|_{\mathbb H}\|\bm{\vartheta}\|_{\mathbb H}\notag\\
&\quad
+\|\nabla\bm{y}_{\mathfrak{H}}\|_{\mathbb H}
\|\nabla\bm{\vartheta}\|_{\mathbb H}
+\|\theta_{\mathfrak{H}}\|_{L^4}
\|\bm{\vartheta}\|_{\mathbb L^4}\|\nabla q\|_{\mathbb H}.
\end{align}
By Ladyzhenskaya's and Young's inequalities, we write
\begin{align}{\label{4th}}
\|\theta_{\mathfrak{H}}\|_{L^4}
\|\bm{\vartheta}\|_{\mathbb L^4}\|\nabla q\|_{\mathbb H}
&\le
\frac{\kappa}{4}\|\nabla q\|_{\mathbb H}^2
+C_\kappa\|\theta_{\mathfrak{H}}\|_{L^4}^2
\|\bm{\vartheta}\|_{\mathbb H}\|\nabla\bm{\vartheta}\|_{\mathbb H}\nonumber\\
&\le
\frac{\kappa}{4}\|\nabla q\|_{\mathbb H}^2
+\frac{\nu}{12}\|\nabla\bm{\vartheta}\|_{\mathbb H}^2
+C_{\nu,\kappa}\|\theta_{\mathfrak{H}}\|_{L^4}^4
\|\bm{\vartheta}\|_{\mathbb H}^2.
\end{align}
Moreover, Young's inequality yields
\begin{align}
\|\nabla\bm{y}_{\mathfrak{H}}\|_{\mathbb H}
\|\nabla\bm{\vartheta}\|_{\mathbb H}
&\le
\frac{\nu}{12}\|\nabla\bm{\vartheta}\|_{\mathbb H}^2
+C_\nu\|\nabla\bm{y}_{\mathfrak{H}}\|_{\mathbb H}^2,
\label{5th}\\
\|\bm{y}_{\mathfrak{H}}-\bm{y}_d\|_{\mathbb H}\|\bm{\vartheta}\|_{\mathbb H}
&\le
\frac12\|\bm{y}_{\mathfrak{H}}-\bm{y}_d\|_{\mathbb H}^2
+\frac12\|\bm{\vartheta}\|_{\mathbb H}^2.
\label{6th}
\end{align}
Combining the estimates \eqref{1st}--\eqref{6th}, we obtain
\begin{equation}\label{energy-p-estimate}
\begin{aligned}
-\frac12\frac{d}{dt}\|\bm{\vartheta}(t)\|_{\mathbb H}^2
+\frac{\nu}{2}\|\nabla\bm{\vartheta}(t)\|_{\mathbb H}^2
&\le
C_{\nu,\kappa,\beta}
\Big(
1+\|\nabla\bm{y}_{\mathfrak{H}}(t)\|_{\mathbb H}^2
+\|\mathfrak{H}(t)\|_{\mathbb H}
+\|\theta_{\mathfrak{H}}(t)\|_{L^4}^4
\Big)\|\bm{\vartheta}(t)\|_{\mathbb H}^2\\
&\quad
+\frac{\kappa}{4}\|\nabla q(t)\|_{\mathbb H}^2
+C\|\bm{y}_{\mathfrak{H}}(t)-\bm{y}_d(t)\|_{\mathbb H}^2
+C_\nu\|\nabla\bm{y}_{\mathfrak{H}}(t)\|_{\mathbb H}^2,
\end{aligned}
\end{equation}
for a.e.\ $t\in(0,T)$.
In a similar manner, taking the $H$-inner product of the second adjoint equation with $q(\cdot)$, we have
\begin{align}
-&\frac{1}{2}\frac{d}{dt}\|q(t)\|_{H}^2 + \kappa \|\nabla q(t)\|_{\mathbb H}^2
\nonumber\\&= \kappa \|\nabla \theta_{\mathfrak{H}}(t)\|_{\mathbb H}^2 \|q(t)\|_{H}^2
+ 2\kappa ( \theta_{\mathfrak{H}}(t),q(t))_{H} ( \nabla \theta_{\mathfrak{H}}(t),\nabla q(t))_{H}  + ( \bm{y}_{\mathfrak{H}}(t)\cdot \nabla q(t), q(t))_{H}
\notag\\
&\quad+ \beta (\bm{\vartheta}(t)\cdot e_2, q(t))_{H}
- \beta ( \bm{y}_{\mathfrak{H}}(t),\bm{\vartheta}(t))_{\mathbb{H}} ( \bm{y}_{\mathfrak{H}}(t)\cdot e_2, q(t))_{H}  + ( \theta_{\mathfrak{H}}(t)-\theta_d,q(t))_{H},
\label{energy_q_1}
\end{align}
for a.e. \(t\in[0,T].\)
Using the fact that
$(\bm{y}_{\mathfrak{H}},\theta_{\mathfrak{H}})\in\mathcal M$,
together with the Cauchy--Schwarz and Young inequalities,
we obtain
\begin{align}{\label{q1}}
2\kappa\left|
(\theta_{\mathfrak{H}},q)_H
(\nabla\theta_{\mathfrak{H}},\nabla q)_H
\right|
&\le
2\kappa\|q\|_H
\|\nabla\theta_{\mathfrak{H}}\|_{\mathbb H}\|\nabla q\|_{\mathbb H}\le
\frac{\kappa}{2}\|\nabla q\|_{\mathbb H}^2
+C_\kappa\|\nabla\theta_{\mathfrak{H}}\|_{\mathbb H}^2\|q\|_H^2.
\end{align}
Similarly, we also have
\begin{align}{\label{q2}}
&\left|\beta(\bm{\vartheta}\cdot e_2,q)_H
-\beta(\bm{y}_{\mathfrak{H}},\bm{\vartheta})_{\mathbb H}
(\bm{y}_{\mathfrak{H}}\cdot e_2,q)_H\right|
+\left|(\theta_{\mathfrak{H}}-\theta_d,q)_H\right|\notag\\
&\quad\le
2|\beta|\|\bm{\vartheta}\|_{\mathbb H}\|q\|_H
+\|\theta_{\mathfrak{H}}-\theta_d\|_H\|q\|_H\notag\\
&\quad\le
\|\bm{\vartheta}\|_{\mathbb H}^2
+\left(\beta^2+\frac12\right)\|q\|_H^2
+\frac12\|\theta_{\mathfrak{H}}-\theta_d\|_H^2.
\end{align}
Using $(\bm{y}_{\mathfrak{H}}\cdot\nabla q,q)_H=0$ and
the estimates \eqref{q1} and \eqref{q2} in
\eqref{energy_q_1}, we obtain
\begin{align}
-\frac{d}{dt}\|q(t)\|_H^2
+\kappa\|\nabla q(t)\|_{\mathbb H}^2
&\le
C_{\kappa,\beta}
\big(1+\|\nabla\theta_{\mathfrak{H}}(t)\|_{\mathbb H}^2\big)\|q(t)\|_H^2
+C\Big(
\|\bm{\vartheta}(t)\|_{\mathbb H}^2
+\|\theta_{\mathfrak{H}}(t)-\theta_d(t)\|_H^2
\Big),
\label{energy_q_final}
\end{align}
for a.e.\ $t\in(0,T)$.
Combining \eqref{energy-p-estimate} and \eqref{energy_q_final}, we deduce
\begin{align}
-&\frac{d}{dt}\big(\|\bm{\vartheta}(t)\|_{\mathbb{H}}^2 + \|q(t)\|_{H}^2\big)
+ \nu \|\nabla \bm{\vartheta}(t)\|_{\mathbb{H}}^2 + \kappa \|\nabla q(t)\|_{\mathbb H}^2 \notag\\
&\le C_{\nu,\kappa,\beta}\big(1 + \|\nabla \bm{y}_{\mathfrak{H}}(t)\|_{\mathbb{H}}^2 + \|\nabla \theta_{\mathfrak{H}}(t)\|_{\mathbb H}^2 + \|\mathfrak{H}(t)\|_{\mathbb{H}}\big)
\big(\|\bm{\vartheta}(t)\|_{\mathbb{H}}^2 + \|q(t)\|_{H}^2\big) \notag\\
&\quad + C\big(\|\bm{y}_{\mathfrak{H}}(t)-\bm{y}_d\|_{\mathbb{H}}^2 + \|\theta_{\mathfrak{H}}(t)-\theta_d\|_{H}^2\big) + C_\nu \|\nabla \bm{y}_{\mathfrak{H}}(t)\|_{\mathbb{H}}^2,
\label{energy_combined}
\end{align}
for a.e. \(t\in[0,T].\)
Integrating \eqref{energy_combined} over \([t,T]\), we have we obtain for all \(t \in [0, T ]\),
\begin{align}
&\|\bm{\vartheta}(t)\|_{\mathbb{H}}^2 + \|q(t)\|_{H}^2
+ \nu \int_t^T \|\nabla \bm{\vartheta}(s)\|_{\mathbb{H}}^2 ds
+ \kappa \int_t^T \|\nabla q(s)\|_{\mathbb H}^2 ds \notag\\
&\le \|\bm{\vartheta}(T)\|_{\mathbb{H}}^2 + \|q(T)\|_{H}^2
+ C\int_t^T \Big[\|\bm{y}_{\mathfrak{H}}(s)-\bm{y}_d\|_{\mathbb{H}}^2 + \|\theta_{\mathfrak{H}}(s)-\theta_d\|_{H}^2\Big]ds \notag\\
&\quad + C_{\nu,\kappa,\beta} \int_t^T \Big[
\big(1+\|\nabla \bm{y}_{\mathfrak{H}}(s)\|_{\mathbb{H}}^2 + \|\nabla \theta_{\mathfrak{H}}(s)\|_{\mathbb H}^2 + \|\mathfrak{H}(s)\|_{\mathbb{H}}\big)
\big(\|\bm{\vartheta}(s)\|_{\mathbb{H}}^2 + \|q(s)\|_{H}^2\big)\Big]ds \notag\\
&\quad + C_\nu \int_t^T \|\nabla \bm{y}_{\mathfrak{H}}(s)\|_{\mathbb{H}}^2 ds.
\label{adjoint-backward}
\end{align}
Applying the backward Gr\"onwall inequality, we deduce
\begin{align}
\|\bm{\vartheta}(t)\|_{\mathbb H}^2+\|q(t)\|_H^2
&\le
\Bigg[
\|\bm{\vartheta}(T)\|_{\mathbb H}^2+\|q(T)\|_H^2
\notag\\
&\quad+C\int_t^T
\Big(
\|\bm{y}_{\mathfrak{H}}(s)-\bm{y}_d(s)\|_{\mathbb H}^2
+\|\theta_{\mathfrak{H}}(s)-\theta_d(s)\|_H^2
\Big)\,ds+C_\nu\int_t^T
\|\nabla\bm{y}_{\mathfrak{H}}(s)\|_{\mathbb H}^2\,ds
\Bigg]
\notag\\
&\quad\times
\exp\!\Bigg(
C_{\nu,\kappa,\beta}\int_t^T
\Big(
1+\|\nabla\bm{y}_{\mathfrak{H}}(s)\|_{\mathbb H}^2
+\|\nabla\theta_{\mathfrak{H}}(s)\|_{\mathbb H}^2
+\|\mathfrak{H}(s)\|_{\mathbb H}
\Big)\,ds
\Bigg),
\label{adjoint-gronwall1}
\end{align}
for every $t\in[0,T]$.
Hence, we infer 
\[
\bm{\vartheta} \in L^\infty(0,T;\mathbb{H}) \cap L^2(0,T;\mathbb{V}), 
\ \
q \in L^\infty(0,T;H) \cap L^2(0,T;V).
\]
Also, for all $\phi \in L^2(0,T;\mathbb{V})$, we have from \eqref{adjj} that 
\begin{align}
& -\int_0^T {}_{\mathbb V'}\!
\left\langle \partial_t \bm{\vartheta}(t),\varphi(t)\right\rangle_{\mathbb V}\,dt
\nonumber\\&=\int_0^T \Big[
-\nu ( \nabla \bm{\vartheta}(t), \nabla \phi(t) )_{\mathbb{H}}
+ \nu \|\nabla \bm{y}_{\mathfrak{H}}(t)\|_{\mathbb{H}}^2 (\bm{\vartheta}(t), \phi(t) )_{\mathbb{H}}  + 2\nu ( \bm{y}_{\mathfrak{H}}(t),\bm{\vartheta}(t))_{\mathbb{H}} ( \nabla \bm{y}_{\mathfrak{H}}(t), \nabla \phi(t) )_{\mathbb{H}}
\notag\\
&\quad-( B'(\bm{y}_{\mathfrak{H}}(t))^*\bm{\vartheta}(t), \phi(t) )_{\mathbb{H}}  - \beta ( \theta_{\mathfrak{H}}(t) e_2,\bm{y}_{\mathfrak{H}}(t))_{\mathbb{H}} (\bm{\vartheta}(t), \phi(t) )_{\mathbb{H}}
- \beta ( \bm{y}_{\mathfrak{H}}(t),\bm{\vartheta}(t))_{\mathbb{H}} ( \theta_{\mathfrak{H}}(t) e_2, \phi(t) )_{\mathbb{H}} \notag\\
&\quad - ( \mathfrak{H}(t),\bm{y}_{\mathfrak{H}}(t))_{\mathbb{H}} (\bm{\vartheta}(t), \phi(t) )_{\mathbb{H}}
- ( \bm{y}_{\mathfrak{H}}(t),\bm{\vartheta}(t))_{\mathbb{H}} ( \mathfrak{H}(t), \phi(t) )_{\mathbb{H}} \notag\\
&\quad + ( \bm{y}_{\mathfrak{H}}(t)-\bm{y}_d, \phi(t) )_{\mathbb{H}}
+  ( A\bm{y}_{\mathfrak{H}}(t), \phi(t) )_{\mathbb{H}}
- ( (\nabla \theta_{\mathfrak{H}}(t))^{\top} q(t), \phi(t) )_{\mathbb{H}}
\Big] dt.
\label{weak_p_1}
\end{align}
Using Hölder's and Young's inequalities, together with the uniform bounds on \(\bm{y}_{\mathfrak{H}}\), \(\theta_{\mathfrak{H}}\), and the facts that \(\|\bm{y}_{\mathfrak{H}}\|_{\mathbb{H}}=\|\theta_{\mathfrak{H}}\|_{H}=1\), we estimate
\begin{align*}
&\left|
\int_0^T
{}_{\mathbb V'}\langle\partial_t\bm{\vartheta}(t),\phi(t)\rangle_{\mathbb V}
\,dt
\right|\nonumber\\&\quad\le C_{\nu,\beta,\lambda_1,T}\Bigg[
\|\nabla\bm{\vartheta}\|_{L^2(0,T;\mathbb H)}
+\|\nabla\bm{y}_{\mathfrak{H}}\|_{L^4(0,T;\mathbb H)}^2
\|\bm{\vartheta}\|_{L^\infty(0,T;\mathbb H)}
+\|\nabla\bm{y}_{\mathfrak{H}}\|_{L^2(0,T;\mathbb H)}
\|\bm{\vartheta}\|_{L^\infty(0,T;\mathbb H)}
\\
&\qquad+\|\bm{y}_{\mathfrak{H}}\|_{L^4(0,T;\mathbb L^4)}
\|\bm{\vartheta}\|_{L^4(0,T;\mathbb L^4)}
+\|\nabla\bm{y}_{\mathfrak{H}}\|_{L^\infty(0,T;\mathbb H)}
\|\bm{\vartheta}\|_{L^2(0,T;\mathbb L^4)}
+\|\bm{\vartheta}\|_{L^2(0,T;\mathbb H)}
+\|\mathfrak{H}\|_{L^2(0,T;\mathbb H)}
\|\bm{\vartheta}\|_{L^\infty(0,T;\mathbb H)}
\\
&\qquad+\|\bm{y}_{\mathfrak{H}}-\bm{y}_d\|_{L^2(0,T;\mathbb H)}
+\|A\bm{y}_{\mathfrak{H}}\|_{L^2(0,T;\mathbb H)}
+\|\nabla\theta_{\mathfrak{H}}\|_{L^4(0,T;\mathbb H)}
\|q\|_{L^4(0,T;L^4)}
\Bigg]\|\phi\|_{L^2(0,T;\mathbb V)}.
\end{align*}
\allowdisplaybreaks
This implies that \(\partial_t \bm{\vartheta} \in L^2(0,T;\mathbb{V}')\), and hence, by the Lions--Magenes lemma, \(\bm{\vartheta} \in C([0,T];\mathbb{H})\).
Again for all $\psi \in L^2(0,T;V)$, similarly we have
\begin{align}
&-\int_0^T {}_{V'}\left\langle \partial_t q(t), \psi(t) \right\rangle_{V} \, dt
\nonumber\\&= \int_0^T \Big[
-\kappa (\nabla q(t),\nabla\psi(t))_{H}
+ \kappa \|\nabla\theta_{\mathfrak{H}}(t)\|_{H}^{2}
(q(t),\psi(t))_{H}
\nonumber\\
&\quad + 2\kappa(\theta_{\mathfrak{H}}(t),q(t))_{H}
(\nabla\theta_{\mathfrak{H}}(t),\nabla\psi(t))_{H}
+ (\bm{y}_{\mathfrak{H}}(t)\cdot\nabla q(t),\psi(t))_{H}
\nonumber\\
&\quad + \beta(\bm{\vartheta}(t)\cdot e_2,\psi(t))_{H}
- \beta(\bm{y}_{\mathfrak{H}}(t),\bm{\vartheta}(t))_{\mathbb{H}}
(\bm{y}_{\mathfrak{H}}(t)\cdot e_2,\psi(t))_{H}
+ (\theta_{\mathfrak{H}}(t)-\theta_d,\psi(t))_{H}
\Big]\,dt
\nonumber\\
&\leq  C_{\kappa,\beta,T,\mu_1}\Big[
\|\nabla q\|_{L^2(0,T;H)}
+\|\nabla\theta_{\mathfrak{H}}\|_{L^4(0,T;H)}^2
\|q\|_{L^\infty(0,T;H)}
+\|\nabla\theta_{\mathfrak{H}}\|_{L^2(0,T;H)}
\|q\|_{L^\infty(0,T;H)}
\nonumber\\
&\qquad+\|\bm{\vartheta}\|_{L^2(0,T;\mathbb H)}
+\|\theta_{\mathfrak{H}}-\theta_d\|_{L^2(0,T;H)}
+\|\bm{y}_{\mathfrak{H}}\|_{L^4(0,T;\mathbb L^4)}
\|q\|_{L^4(0,T;L^4)}
\Big]\|\psi\|_{L^2(0,T;V)}.
\label{weak_q_final}
\end{align}
From the above discussion, we obtain the following result on the well-posedness of the adjoint problem \eqref{adjj} in the weak sense.
\begin{theorem}
Let \(\mathfrak{H} \in L^2(0,T;\mathbb{H})\). Then the adjoint system \eqref{adjj} admits a unique weak solution \((\bm{\vartheta},q)\) such that
\[
\bm{\vartheta} \in L^\infty(0,T;\mathbb{H}) \cap L^2(0,T;\mathbb{V}), \quad
q \in L^\infty(0,T;H) \cap L^2(0,T;V),
\]
with \(\partial_t \bm{\vartheta} \in L^2(0,T;\mathbb{V}')\), \(\partial_t q \in L^2(0,T;V')\). Consequently, \(\bm{\vartheta} \in C([0,T];\mathbb{H})\) and \(q \in C([0,T];H)\).

Moreover, applying Grönwall's inequality, we obtain, for all \(t \in [0,T]\),
\[
\begin{aligned}
&\sup_{0\le t\le T}\big(\|\bm{\vartheta}(t)\|_{\mathbb{H}}^2 + \|q(t)\|_{H}^2\big)
+ \nu \int_0^T \|\nabla \bm{\vartheta}(t)\|_{\mathbb{H}}^2 dt
+ \kappa \int_0^T \|\nabla q(t)\|_{\mathbb H}^2 dt \\
&\le \Big(
\|\bm{\vartheta}(T)\|_{\mathbb{H}}^2 + \|q(T)\|_{H}^2
+ C \int_0^T \big(\|\bm{y}_{\mathfrak{H}}(s)-\bm{y}_d\|_{\mathbb{H}}^2 + \|\theta_{\mathfrak{H}}(s)-\theta_d\|_{H}^2\big)\,ds \\
&\quad + C_\nu \int_0^T \|\nabla \bm{y}_{\mathfrak{H}}(s)\|_{\mathbb{H}}^2\,ds
\Big)  \exp\Big(
C_{\nu,\kappa,\beta,T} \int_0^T \big(1 + \|\nabla \bm{y}_{\mathfrak{H}}(s)\|_{\mathbb{H}}^2 + \|\nabla\theta_{\mathfrak{H}}(s)\|_{\mathbb H}^2 + \|\mathfrak{H}(s)\|_{\mathbb{H}}\big)\,ds
\Big).
\end{aligned}
\]
Furthermore, for all \(t\in[0,T]\), the energy estimate is given by
\begin{align}
&\|\bm{\vartheta}(t)\|_{\mathbb H}^2
+2\nu\int_t^T\|\nabla\bm{\vartheta}(s)\|_{\mathbb H}^2\,ds
\notag\\
&=\|\bm{\vartheta}(T)\|_{\mathbb H}^2
+2\nu\int_t^T
\|\nabla\bm{y}_{\mathfrak{H}}(s)\|_{\mathbb H}^2
\|\bm{\vartheta}(s)\|_{\mathbb H}^2\,ds
+4\nu\int_t^T
(\bm{y}_{\mathfrak{H}}(s),\bm{\vartheta}(s))_{\mathbb H}
(\nabla\bm{y}_{\mathfrak{H}}(s),\nabla\bm{\vartheta}(s))_{\mathbb H}\,ds
\notag\\
&\quad
-2\int_t^T
(B(\bm{\vartheta}(s),\bm{y}_{\mathfrak{H}}(s)),\bm{\vartheta}(s))_{\mathbb H}\,ds
-2\beta\int_t^T
(\theta_{\mathfrak{H}}(s)e_2,\bm{y}_{\mathfrak{H}}(s))_{\mathbb H}
\|\bm{\vartheta}(s)\|_{\mathbb H}^2\,ds\notag\\
&\quad
-2\beta\int_t^T
(\bm{y}_{\mathfrak{H}}(s),\bm{\vartheta}(s))_{\mathbb H}
(\theta_{\mathfrak{H}}(s)e_2,\bm{\vartheta}(s))_{\mathbb H}\,ds
-2\int_t^T
(\mathfrak{H}(s),\bm{y}_{\mathfrak{H}}(s))_{\mathbb H}
\|\bm{\vartheta}(s)\|_{\mathbb H}^2\,ds
\notag\\
&\quad
-2\int_t^T
(\mathfrak{H}(s),\bm{\vartheta}(s))_{\mathbb H}
(\bm{y}_{\mathfrak{H}}(s),\bm{\vartheta}(s))_{\mathbb H}\,ds
+2\int_t^T
(\bm{y}_{\mathfrak{H}}(s)-\bm{y}_d,\bm{\vartheta}(s))_{\mathbb H}\,ds
\notag\\
&\quad
+2\int_t^T
(\nabla\bm{y}_{\mathfrak{H}}(s),\nabla\bm{\vartheta}(s))_{\mathbb H}\,ds
-2\int_t^T
((\nabla\theta_{\mathfrak{H}}(s))^\top q(s),\bm{\vartheta}(s))_{\mathbb H}\,ds.
\label{adjoint-energy-p}
\end{align}
\begin{align}
&\|q(t)\|_{H}^2 + 2\kappa \int_t^T \|\nabla q(s)\|_{\mathbb H}^2 ds
\nonumber\\&= \|q(T)\|_{H}^2 
+ 2\kappa \int_t^T \|\nabla \theta_{\mathfrak{H}}(s)\|_{\mathbb H}^2 \|q(s)\|_{H}^2 ds + 4\kappa \int_t^T ( \theta_{\mathfrak{H}}(s),q(s))_{H}
( \nabla \theta_{\mathfrak{H}}(s),\nabla q(s))_{H} ds \notag\\
&\quad + 2 \int_t^T ( \theta_{\mathfrak{H}}(s)-\theta_d, q(s))_{H} ds  + 2\beta \int_t^T \Big[
(\bm{\vartheta}(s)\cdot e_2, q(s))_{H}
- ( \bm{y}_{\mathfrak{H}}(s),\bm{\vartheta}(s))_{\mathbb{H}} 
( \bm{y}_{\mathfrak{H}}(s)\cdot e_2, q(s))_{H}
\Big] ds.
\label{adjoint-energy}
\end{align}
\end{theorem}
\section{First-order necessary optimality conditions}{\label{s4}}
We now derive a first-order necessary condition for the optimal control problem \eqref{opt} and express the optimal control using the associated adjoint variables. We first use the Hamiltonian to give a formal derivation of Pontryagin's minimum principle.
We define the Lagrangian as
\begin{align}
L(\bm{y},\theta,\mathfrak{H})
&:= \frac{1}{2}\|\bm{y} - \bm{y}_d\|_{\mathbb{H}}^2
+ \frac{1}{2}\|\nabla \bm{y}\|_{\mathbb{H}}^2
+ \frac{1}{2}\|\theta - \theta_d\|_{H}^2
+ \frac{1}{2}\|\mathfrak{H}\|_{\mathbb{H}}^2,
\end{align}
so that
\begin{align}
J(\bm{y},\theta,\mathfrak{H})
&= \int_0^T L(\bm{y}(t),\theta(t),\mathfrak{H}(t))\,dt
+ \frac{1}{2}\|\bm{y}(T)-\bm{y}_f\|_{\mathbb{H}}^2
+ \frac{1}{2}\|\theta(T)-\theta_f\|_{H}^2.
\end{align}
We introduce the Hamiltonian
\begin{align}
H(\bm{y},\theta,\mathfrak{H},\bm{\vartheta},q)
:= L(\bm{y},\theta,\mathfrak{H})
+ (\bm{\vartheta}, F(\bm{y},\theta,\mathfrak{H}) )_{\mathbb{H}}
+ ( q, G(\bm{y},\theta) )_{H},
\end{align}
where
\begin{align}
F(\bm{y},\theta,\mathfrak{H})
&= -\nu A\bm{y} + \nu\|\nabla \bm{y}\|_{\mathbb{H}}^2 \bm{y}
- B(\bm{y},\bm{y})
+ \beta {\textbf{Proj}_{\mathbb H}}(\theta e_2)
- \beta( \theta e_2,\bm{y})_{\mathbb{H}}\,\bm{y} \notag\\
&\quad + \mathfrak{H} - ( \mathfrak{H},\bm{y})_{\mathbb{H}}\,\bm{y},
\\
G(\bm{y},\theta)
&= -\kappa \mathcal{A}\theta
+ \kappa \|\nabla \theta\|_{\mathbb H}^2 \theta
- \bm{y}\cdot\nabla \theta.
\end{align}
The adjoint variables \(\bm{\vartheta}(\cdot)\) and \(q(\cdot)\) satisfy the system \eqref{addd} given below. Since \((\bm{y}^*,\theta^*,\mathfrak{H}^*)\) is an optimal triplet for problem \eqref{opt}, Pontryagin's minimum principle gives
\begin{align}
H(\bm{y}^*(t),\theta^*(t),\mathfrak{H}^*(t),\bm{\vartheta}(t),q(t))
\le H(\bm{y}^*(t),\theta^*(t),\bm{W},\bm{\vartheta}(t),q(t)),
\  \text{ for all } \ \bm{W} \in \mathcal{U}_{ad}, \ \text{ a.e. } \ t \in [0,T].
\end{align}
Thus, an optimal triplet \((\bm{y}^*,\theta^*,\mathfrak{H}^*) \in \mathcal{A}_{ad}\) satisfies the following minimum principle:
\begin{align}
&\frac{1}{2}\|\mathfrak{H}^*(t)\|_{\mathbb{H}}^2
+ \big(\bm{\vartheta}(t),\, \mathfrak{H}^*(t) - ( \mathfrak{H}^*(t),\bm{y}^*(t))_{\mathbb{H}}\bm{y}^*(t)\big)_{\mathbb{H}} \notag\\
&\qquad \le \frac{1}{2}\|\bm{W}\|_{\mathbb{H}}^2
+ \big(\bm{\vartheta}(t),\, \bm{W} - ( \bm{W},\bm{y}^*(t))_{\mathbb{H}}\bm{y}^*(t)\big)_{\mathbb{H}},
\end{align}
for all \(\bm{W} \in \mathcal{U}_{ad}\) and a.e.\ \(t \in [0,T]\).
Rearranging, we obtain the variational inequality
\begin{align}
\frac{1}{2}\|\bm{W}\|_{\mathbb{H}}^2 - \frac{1}{2}\|\mathfrak{H}^*(t)\|_{\mathbb{H}}^2
+ \big(\bm{\vartheta}(t) - (\bm{\vartheta}(t),\bm{y}^*(t))_{\mathbb{H}}\,\bm{y}^*(t),\, \bm{W} - \mathfrak{H}^*(t)\big)_{\mathbb{H}} \ge 0,
\end{align}
for all \(\bm{W} \in \mathcal{U}_{ad}\) and a.e.\ \(t \in [0,T]\).
This implies
\[
- \bm{\vartheta}(t) + (\bm{\vartheta}(t),\bm{y}^*(t))_{\mathbb{H}}\,\bm{y}^*(t)
\in \partial\!\left(\frac{1}{2}\|\mathfrak{H}^*(t)\|_{\mathbb{H}}^2\right),
\quad \text{for a.e. } t \in [0,T],
\]
where \(\partial h(\cdot)\) is the subdifferential of \(h(\cdot)\).
Since the functional \(\mathfrak{H} \mapsto \tfrac{1}{2}\|\mathfrak{H}\|_{\mathbb{H}}^2\) is Fréchet differentiable, its subdifferential is single-valued. Hence, we deduce
\begin{align}
\mathfrak{H}^*(t)
= -\bm{\vartheta}(t) + (\bm{\vartheta}(t),\bm{y}^*(t))_{\mathbb{H}}\,\bm{y}^*(t)
= -\pi_{\bm{y}^*(t)}(\bm{\vartheta}(t)),
\quad \text{for a.e.\ } t \in [0,T].
\end{align}
Moreover, the optimal control \(\mathfrak{H}^*\) is orthogonal to \(\bm{y}^*\) in \(\mathbb{H}\), that is,
\[
( \mathfrak{H}^*(t), \bm{y}^*(t) )_{\mathbb{H}} = 0,
\  \text{ for a.e. } \ t \in [0,T].
\]
Therefore, we infer
\[
\mathfrak{H}^*(t) \in T_{\bm{y}^*(t)}\mathcal{M}_1,
\  \text{ for a.e. } \ t \in [0,T].
\]
Recall that the tangent space at \(\bm{y} \in \mathcal{M}_1\) is given by
$
T_{\bm{y}} \mathcal{M}_1
= \big\{ \bm{\varrho} \in \mathbb{H} : ( \bm{y}, \bm{\varrho} )_{\mathbb{H}} = 0 \big\}.
$
The tangent bundle is defined as
\[
T\mathcal{M}_1 := \bigsqcup_{\bm{y} \in \mathcal{M}_1} T_{\bm{y}} \mathcal{M}_1.
\]
Since
$
\mathfrak{H}^*(t) \in T_{\bm{y}^*(t)}\mathcal{M}_1, 
\  \text{ for a.e. } \ t \in [0,T],
$
the map \(t \mapsto (\bm{y}^*(t), \mathfrak{H}^*(t))\) is a curve in the tangent bundle \(T\mathcal{M}_1\). We summarize the above analysis in the following theorem, which states the first-order necessary optimality conditions obtained from Pontryagin's maximum principle. The proof follows from analytical arguments similar to those developed in \cite{abergel,manil,Sritharan1998,MohanNS}.
\begin{theorem}{\label{main}}
Let \((\bm{y}^*, \theta^*, \mathfrak{H}^*) \in \mathcal{A}_{ad}\) be an optimal solution to the problem \eqref{opt} obtained in Theorem~\ref{optt}. Consequently, there is a unique weak solution \((\bm{\vartheta},q)\) of the adjoint system such that
\[
\bm{\vartheta} \in L^\infty(0,T;\mathbb{H}) \cap L^2(0,T;\mathbb{V}), \ \
q \in L^\infty(0,T;H) \cap L^2(0,T;V),
\]
with \(\partial_t \bm{\vartheta} \in L^2(0,T;\mathbb{V}')\) and \(\partial_t q \in L^2(0,T;V')\).

Moreover, \((\bm{\vartheta},q)\) satisfies, for a.e.\ \(t\in[0,T]\),
\begin{equation}{\label{addd}}
\begin{cases}
-\dfrac{d\bm{\vartheta}(t)}{dt} = -\nu A \bm{\vartheta}(t) + \nu \|\nabla \bm{y}^*(t)\|_{\mathbb{H}}^2 \bm{\vartheta}(t)
+ 2\nu ( \bm{y}^*(t),\bm{\vartheta}(t))_{\mathbb{H}} A \bm{y}^*(t)
- (B'(\bm{y}^*(t)))^*\bm{\vartheta}(t) \\[4pt]
\qquad - \beta ( \theta^*(t) e_2,\bm{y}^*(t))_{\mathbb{H}} \bm{\vartheta}(t)
- \beta ( \bm{y}^*(t),\bm{\vartheta}(t))_{\mathbb{H}} \theta^*(t) e_2
- ( \mathfrak{H}^*(t),\bm{y}^*(t))_{\mathbb{H}} \bm{\vartheta}(t)
- ( \bm{y}^*(t),\bm{\vartheta}(t))_{\mathbb{H}} \mathfrak{H}^*(t) \\[4pt]
\qquad - (\nabla \theta^*(t))^{\top} q(t) + (\bm{y}^*(t) - \bm{y}_d) + A \bm{y}^*(t) \ \text{ in }\ \mathbb{V}^{\prime}, \\[6pt]
-\dfrac{dq(t)}{dt} = -\kappa {\mathcal{A} q(t)} + \kappa \|\nabla \theta^*(t)\|_{\mathbb H}^2 q(t)
+ 2\kappa ( \theta^*(t),q(t))_{H} A \theta^*(t)
+ \bm{y}^*(t)\cdot\nabla q(t) \\[4pt]
\qquad + \beta (\bm{\vartheta}(t)\cdot e_2)
- \beta ( \bm{y}^*(t),\bm{\vartheta}(t))_{\mathbb{H}} (\bm{y}^*(t)\cdot e_2)
+ (\theta^*(t) - \theta_d)\ \text{ in }\ V',
\end{cases}
\end{equation}
with terminal conditions
\[
\bm{\vartheta}(T) = \bm{y}^*(T) - \bm{y}_f, \quad q(T) = \theta^*(T) - \theta_f.
\]
Furthermore, the optimal control is characterized by
$$
\mathfrak{H}^*(t) = -\bm{\vartheta}(t) + (\bm{\vartheta}(t),\bm{y}^*(t))_{\mathbb{H}}\,\bm{y}^*(t)= -\pi_{\bm{y}^*(t)}(\bm{\vartheta}(t)),
$$
for a.e.\ \(t \in [0,T]\), and satisfies
\(
( \mathfrak{H}^*(t),\bm{y}^*(t))_{\mathbb{H}} = 0.
\)
Consequently,
\(
\mathfrak{H}^*(t) \in T_{\bm{y}^*(t)}\mathcal{M}_1.
\)
\end{theorem}

\begin{remark}
Since \(( \mathfrak{H}^*(t), \bm{y}^*(t))_{\mathbb{H}} = 0\) for all \(t\in[0,T]\), under the optimal control \((\bm{y}^*,\theta^*,\mathfrak{H}^*)\), the state system reduces to
\begin{equation}
\begin{cases}
\dfrac{d\bm{y}^*(t)}{dt}
= -\nu A \bm{y}^*(t) + \nu \|\nabla \bm{y}^*(t)\|_{\mathbb{H}}^2 \bm{y}^*(t)
- B(\bm{y}^*(t),\bm{y}^*(t)) + \beta {\textbf{Proj}_{\mathbb H}}(\theta^*(t) e_2)
- \beta( \theta^*(t) e_2, \bm{y}^*(t))_{\mathbb{H}} \bm{y}^*(t)\\ \qquad\qquad + \mathfrak{H}^*(t) \ \text{ in } \mathbb{H}, \\[4pt]
\dfrac{d\theta^*(t)}{dt}
= -\kappa A \theta^*(t) + \kappa \|\nabla \theta^*(t)\|_{\mathbb H}^2 \theta^*(t)
- \bm{y}^*(t)\cdot\nabla \theta^*(t)\ \text{ in }\ H,
\\[4pt]
(\bm{y}^*,\theta^*)(0) = (\bm{y}_0,\theta_0)\in(\mathbb{V}\cap\mathcal{M}_1)\times (V\cap\mathcal{M}_2).
\end{cases}
\end{equation}
\allowdisplaybreaks
The adjoint system \eqref{addd} reduces to
\begin{equation}{\label{adjjj1}}
\begin{cases}
-\dfrac{d\bm{\vartheta}(t)}{dt}
= -\nu A \bm{\vartheta}(t) + \nu \|\nabla \bm{y}^*(t)\|_{\mathbb{H}}^2 \bm{\vartheta}(t)
+ 2\nu ( \bm{y}^*(t),\bm{\vartheta}(t))_{\mathbb{H}} A \bm{y}^*(t)
 - (B'(\bm{y}^*(t)))^*\bm{\vartheta}(t)  \\[4pt]
\qquad\qquad \quad- \beta ( \theta^*(t) e_2, \bm{y}^*(t))_{\mathbb{H}} \bm{\vartheta}(t)
- \beta ( \bm{y}^*(t),\bm{\vartheta}(t))_{\mathbb{H}} \theta^*(t) e_2 \\
\qquad\qquad\quad -\left[\bm{\vartheta}(t)
-(\bm{\vartheta}(t),\bm{y}^*(t))_{\mathbb H}\bm{y}^*(t)\right]
(\bm{y}^*(t),\bm{\vartheta}(t))_{\mathbb H}- (\nabla \theta^*(t))^{\top} q(t) + (\bm{y}^*(t) - \bm{y}_d) + A \bm{y}^*(t)\ \text{ in }\ \mathbb{V}', \\[6pt]
-\dfrac{dq(t)}{dt}
= -\kappa {\mathcal{A} q(t)} + \kappa \|\nabla \theta^*(t)\|_{\mathbb H}^2 q(t)
+ 2\kappa ( \theta^*(t),q(t))_{H} A \theta^*(t)
+ \bm{y}^*(t)\cdot\nabla q(t) \\[4pt]
\qquad\qquad  + \beta (\bm{\vartheta}(t)\cdot e_2)
- \beta ( \bm{y}^*(t),\bm{\vartheta}(t))_{\mathbb{H}} (\bm{y}^*(t)\cdot e_2)
+ (\theta^*(t) - \theta_d) \ \text{ in }\ V', \\[6pt]
\bm{\vartheta}(T) = \bm{y}^*(T) - \bm{y}_f, \quad
q(T) = \theta^*(T) - \theta_f.
\end{cases}
\end{equation}
\end{remark}
\begin{lemma}{\label{incest}}
Let \((\bm{y}^*,\theta^*,\mathfrak{H}^*) \in \mathcal{A}_{ad}\) be an optimal triplet, and let 
\((\bm{y}^*,\theta^*) = (\bm{y}_{\mathfrak{H}^*}, \theta_{\mathfrak{H}^*})\).
For sufficiently small \(\tau > 0\) and \(\mathfrak{H} \in \mathcal{U}_{ad}\), let 
\((\bm{y}^\tau,\theta^\tau) := (\bm{y}_{\mathfrak{H}^*+\tau \mathfrak{H}}, \theta_{\mathfrak{H}^*+\tau \mathfrak{H}})\) be the corresponding unique strong solution of \eqref{eq:constrained_Boussinesq} associated with \(\mathfrak{H}^*+\tau \mathfrak{H}\).
Then the following estimates hold:
\begin{align}
&\sup_{t\in[0,T]} \Big(
\|\bm{y}^\tau(t) - \bm{y}^*(t)\|_{\mathbb{H}}^2 
+ \|\theta^\tau(t) - \theta^*(t)\|_{H}^2
\Big) \nonumber\\
&\quad + \nu \int_0^T \|\nabla (\bm{y}^\tau(t) - \bm{y}^*(t))\|_{\mathbb{H}}^2 \, dt
+ \kappa \int_0^T \|\nabla (\theta^\tau(t) - \theta^*(t))\|_{H}^2 \, dt \nonumber\\
&\le \tau^2 K \int_0^T \|\mathfrak{H}(t)\|_{\mathbb{H}}^2 \, dt,
\label{lemma_low_order}
\end{align}
where \(K = K(\nu,\beta,\kappa,T,\|\bm{y}_0\|_{\mathbb{V}},\|\theta_0\|_{V},\|\mathfrak{H}^*\|_{L^2(0,T;\mathbb{H})},\|\mathfrak{H}\|_{L^2(0,T;\mathbb{H})})\).
Moreover,
\begin{align}
&\sup_{t\in[0,T]} \Big(
\|\nabla (\bm{y}^\tau(t) - \bm{y}^*(t))\|_{\mathbb{H}}^2 
+ \|\nabla (\theta^\tau(t) - \theta^*(t))\|_{H}^2
\Big) \nonumber\\
&\quad + \nu \int_0^T \|A(\bm{y}^\tau(t) - \bm{y}^*(t))\|_{\mathbb{H}}^2 \, dt
+ \kappa \int_0^T \|\mathcal A(\theta^\tau(t) - \theta^*(t))\|_{H}^2 \, dt \nonumber\\
&\le \tau^2 K^{(1)} \int_0^T \|\mathfrak{H}(t)\|_{\mathbb{H}}^2 \, dt,
\label{lemma_high_order}
\end{align}
where \(K^{(1)} = K^{(1)}(\nu,\beta,\kappa,T,\|\bm{y}_0\|_{\mathbb{V}},\|\theta_0\|_{V},\|\mathfrak{H}^*\|_{L^2(0,T;\mathbb{H})},\|\mathfrak{H}\|_{L^2(0,T;\mathbb{H})})\).
\end{lemma}
\begin{proof}
Let \((\bm{y}^*,\theta^*)=(\bm{y}_{\mathfrak{H}^*},\theta_{\mathfrak{H}^*})\) and \((\bm{y}^\tau,\theta^\tau)=(\bm{y}_{\mathfrak{H}^*+\tau \mathfrak{H}},\theta_{\mathfrak{H}^*+\tau \mathfrak{H}})\), and set \(\widetilde{\bm{y}}=\bm{y}^\tau-\bm{y}^*\), \(\widetilde{\theta}=\theta^\tau-\theta^*\).
Then \((\widetilde{\bm{y}},\widetilde{\theta})\) satisfies the system below for a.e. $t\in(0,T)$:
\begin{equation}\label{deff}
\begin{cases}
\displaystyle
\frac{d\widetilde{\bm{y}}(t)}{dt}
= -\nu A\widetilde{\bm{y}}(t)
+ \nu\Big( \|\nabla \bm{y}^\tau(t)\|_{\mathbb{H}}^2\, \bm{y}^\tau(t) - \|\nabla \bm{y}^*(t)\|_{\mathbb{H}}^2\, \bm{y}^*(t) \Big)
- \Big( B(\bm{y}^\tau(t),\bm{y}^\tau(t)) - B(\bm{y}^*(t),\bm{y}^*(t)) \Big) \\[4pt]
\qquad
+ \beta {\textbf{Proj}_{\mathbb H}}(\widetilde{\theta}(t) e_2)
- \beta\Big( ( \theta^\tau(t) e_2, \bm{y}^\tau(t))_{\mathbb{H}}\,\bm{y}^\tau(t) - ( \theta^*(t) e_2, \bm{y}^*(t))_{\mathbb{H}}\,\bm{y}^*(t) \Big)
+ \tau \mathfrak{H}(t) \\[4pt]
\qquad
- \Big( ( \mathfrak{H}^*(t)+\tau \mathfrak{H}(t), \bm{y}^\tau(t))_{\mathbb{H}}\,\bm{y}^\tau(t) - ( \mathfrak{H}^*(t),\bm{y}^*(t))_{\mathbb{H}}\,\bm{y}^*(t) \Big),
\quad \text{in } \mathbb{V'}, \\[8pt]

\displaystyle
\frac{d\widetilde{\theta}(t)}{dt}
= -\kappa {\mathcal{A}\widetilde{\theta}}(t)
+ \kappa\Big( \|\nabla \theta^\tau(t)\|_{\mathbb H}^2\, \theta^\tau(t)
- \|\nabla \theta^*(t)\|_{\mathbb H}^2\, \theta^*(t) \Big) \\[4pt]
\qquad
- \Big( \bm{y}^\tau(t) \cdot \nabla \theta^\tau(t) - \bm{y}^*(t) \cdot \nabla \theta^*(t) \Big),
\quad \text{in } V', \\[8pt]

\widetilde{\bm{y}}(0)=0, \ \ \widetilde{\theta}(0)=0.
\end{cases}
\end{equation}
Let $(\bm{y}^*,\theta^*,\mathfrak{H}^*)\in\mathcal A_{\rm ad}$ be an optimal
triplet associated with the optimal control problem \eqref{opt}.
Then the following estimate holds:
\begin{align}{\label{neww}}
&\sup_{0\le t\le T}\|\nabla \bm{y}_{\mathfrak{H}^*}(t)\|_{H}^{2}
+\int_{0}^{T}
\left\|
\frac{d\bm{y}_{\mathfrak{H}^*}(t)}{dt}
\right\|_{\mathbb H}^{2}\,dt
+\nu\int_{0}^{T}
\|A\bm{y}_{\mathfrak{H}^*}(t)\|_{\mathbb H}^{2}\,dt
\nonumber\\
&\quad
+\sup_{0\le t\le T}\|\nabla\theta_{\mathfrak{H}^*}(t)\|_{H}^{2}
+\int_{0}^{T}
\left\|
\frac{d\theta_{\mathfrak{H}^*}(t)}{dt}
\right\|_{H}^{2}\,dt
+\kappa\int_{0}^{T}
\|A\theta_{\mathfrak{H}^*}(t)\|_{H}^{2}\,dt
\nonumber\\
&\le
C\left(
\nu,\beta,\kappa,T,
\|\bm{y}_0\|_{\mathbb V},
\|\theta_0\|_{V},
\|\mathfrak{H}^*\|_{L^{2}(0,T;\mathbb H)}
\right).
\end{align}

Moreover, for every fixed $\mathfrak{H}\in \mathcal U_{\rm ad}$ and sufficiently small
$\tau>0$, the perturbed state
$(\bm{y}_{\mathfrak{H}^*+\tau\mathfrak{H}},\theta_{\mathfrak{H}^*+\tau\mathfrak{H}})$
also satisfies the above estimate, with a constant
\[
C=C\left(
\nu,\beta,\kappa,T,
\|\bm{y}_0\|_{\mathbb V},
\|\theta_0\|_{V},
\|\mathfrak{H}^*\|_{L^{2}(0,T;\mathbb H)},
\|\mathfrak{H}\|_{L^{2}(0,T;\mathbb H)}
\right)
\]
independent of $\tau$.

Taking the $\mathbb{H}$-inner product of the first equation in \eqref{deff} with \(\widetilde{\bm{y}}(\cdot)\), we have
\begin{align}
&\frac{1}{2}\frac{d}{dt}\|\widetilde{\bm{y}}(t)\|_{\mathbb{H}}^2 + \nu \|\nabla \widetilde{\bm{y}}(t)\|_{\mathbb{H}}^2\notag
\\& = \nu \|\nabla \bm{y}^*(t)\|_{\mathbb{H}}^2 \|\widetilde{\bm{y}}(t)\|_{\mathbb{H}}^2 + \nu \big(\|\nabla \bm{y}^\tau(t)\|_{\mathbb{H}}^2 - \|\nabla \bm{y}^*(t)\|_{\mathbb{H}}^2\big)
( \bm{y}^\tau(t), \widetilde{\bm{y}}(t))_{\mathbb{H}} - ( B(\widetilde{\bm{y}}(t),\bm{y}^{\tau}(t)), \widetilde{\bm{y}}(t))_{\mathbb{H}} 
\notag\\
&\quad + \beta ( {\textbf{Proj}}_{\mathbb H}(\widetilde{\theta}(t)e_2), \widetilde{\bm{y}}(t))_{\mathbb{H}} - \beta (\widetilde{\theta}(t)e_2, \bm{y}^{\tau}(t))_{\mathbb{H}}
( \bm{y}^*(t), \widetilde{\bm{y}}(t))_{\mathbb{H}} - \beta ( \theta^*(t)e_2, \widetilde{\bm{y}}(t))_{\mathbb{H}}
( \bm{y}^*(t), \widetilde{\bm{y}}(t))_{\mathbb{H}}\notag\\
&\quad -\beta ( \theta^{\tau}(t)e_2, \bm{y}^{\tau}(t))_{\mathbb{H}}
\|\widetilde{\bm{y}}(t)\|_{\mathbb{H}}^2 + \tau ( \mathfrak{H}(t), \widetilde{\bm{y}}(t))_{\mathbb{H}} - \tau ( \mathfrak{H}(t), \bm{y}^\tau(t))_{\mathbb{H}}
( \bm{y}^\tau(t), \widetilde{\bm{y}}(t))_{\mathbb{H}} \notag\\
&\quad - ( \mathfrak{H}^*(t), \widetilde{\bm{y}}(t))_{\mathbb{H}}
( \bm{y}^*(t), \widetilde{\bm{y}}(t))_{\mathbb{H}}- ( \mathfrak{H}^*(t), \bm{y}^{\tau}(t))_{\mathbb{H}}
\|\widetilde{\bm{y}}(t)\|_{\mathbb{H}}^2,
\label{tau-est-1}
\end{align}
for a.e. \(t\in(0,T)\).
We now estimate each term on the right-hand side of \eqref{tau-est-1}. Throughout, we use the facts that \(\bm{y}^\tau,\bm{y}^*\in\mathcal{M}_1\) and \(\theta^\tau,\theta^*\in\mathcal{M}_2\), together with the uniform bound given by \eqref{neww}.
Applying the Cauchy-Schwarz inequality followed by Young's inequality yields
\[
\begin{aligned}
\nu \big(\|\nabla \bm{y}^\tau\|_{\mathbb{H}}^2 - \|\nabla \bm{y}^*\|_{\mathbb{H}}^2\big)
( \bm{y}^\tau, \widetilde{\bm{y}})_{\mathbb{H}}
&\le \frac{\nu}{6}\|\nabla \widetilde{\bm{y}}\|_{\mathbb{H}}^2
+ C_\nu \big(\|\nabla \bm{y}^\tau\|_{\mathbb{H}}^2 + \|\nabla \bm{y}^*\|_{\mathbb{H}}^2\big)
\|\widetilde{\bm{y}}\|_{\mathbb{H}}^2.
\end{aligned}
\]
Using Hölder's, Ladyzhenskaya's, and Young's inequalities, the convective term is bounded as
\[
\begin{aligned}
- ( B(\widetilde{\bm{y}},\bm{y}^{\tau}), \widetilde{\bm{y}})_{\mathbb{H}}
&\le \frac{\nu}{6}\|\nabla \widetilde{\bm{y}}\|_{\mathbb{H}}^2
+ C_\nu \|\nabla \bm{y}^{\tau}\|_{\mathbb{H}}^2 \|\widetilde{\bm{y}}\|_{\mathbb{H}}^2.
\end{aligned}
\]
For the temperature coupling terms, we use the continuity of the projection operator \({\textbf{Proj}_{\mathbb H}}\), the Cauchy-Schwarz inequality, and Young's inequality to find 
\[
\beta ( {\textbf{Proj}}_{\mathbb H}(\widetilde{\theta} e_2), \widetilde{\bm{y}})_{\mathbb{H}}
\le \beta \|\widetilde{\theta}\|_{H}\|\widetilde{\bm{y}}\|_{\mathbb H}
\le C_\beta \|\widetilde{\theta}\|_{H}^2 + \frac{1}{4}\|\widetilde{\bm{y}}\|_{\mathbb H}^2.
\]
Similarly, using \(\bm{y}^\tau, \bm{y}^* \in \mathcal{M}_1\), we deduce
\[
\begin{aligned}
- \beta (\widetilde{\theta} e_2, \bm{y}^{\tau})_{\mathbb{H}}
( \bm{y}^*, \widetilde{\bm{y}})_{\mathbb{H}}
&\le \beta \|\widetilde{\theta}\|_{H}\|\bm{y}^\tau\|_{\mathbb H}\|\bm{y}^*\|_{\mathbb H}\|\widetilde{\bm{y}}\|_{\mathbb H}\le \beta \|\widetilde{\theta}\|_{H}\|\widetilde{\bm{y}}\|_{\mathbb H}
\le C_\beta \|\widetilde{\theta}\|_{H}^2 + \frac{1}{4}\| \widetilde{\bm{y}}\|_{\mathbb H}^2,
\end{aligned}
\]
and, using \(\bm{y}^*, \bm{y}^\tau \in \mathcal{M}_1\) and \(\theta^*, \theta^\tau \in \mathcal{M}_2\), we get
\[
\begin{aligned}
- \beta ( \theta^* e_2, \widetilde{\bm{y}})_{\mathbb{H}}
( \bm{y}^*, \widetilde{\bm{y}})_{\mathbb{H}}
\le \beta \|\widetilde{\bm{y}}\|_{\mathbb H}^2 \ \text{ and }\ 
- \beta ( \theta^{\tau} e_2, \bm{y}^{\tau})_{\mathbb{H}}
\|\widetilde{\bm{y}}\|_{\mathbb{H}}^2
\le \beta \|\widetilde{\bm{y}}\|_{\mathbb H}^2.
\end{aligned}
\]
Applying the Cauchy-Schwarz and Young inequalities, we find
\[
\tau ( \mathfrak{H}, \widetilde{\bm{y}})_{\mathbb{H}}
\le \tau \|\mathfrak{H}\|_{\mathbb H}\|\widetilde{\bm{y}}\|_{\mathbb H}
\le C\tau^2 \|\mathfrak{H}\|_{\mathbb H}^2 + \frac{1}{4} \|\widetilde{\bm{y}}\|_{\mathbb H}^2.
\]
Using \(\bm{y}^\tau \in \mathcal{M}_1\) and Young's inequality, we obtain
\[
- \tau ( \mathfrak{H}, \bm{y}^\tau)_{\mathbb{H}}
( \bm{y}^\tau, \widetilde{\bm{y}})_{\mathbb{H}}
\le \tau \|\mathfrak{H}\|_{\mathbb H}\|\bm{y}^\tau\|_{\mathbb H}\|\bm{y}^\tau\|_{\mathbb H}\|\widetilde{\bm{y}}\|_{\mathbb H}
\le C\tau^2 \|\mathfrak{H}\|_{\mathbb H}^2 + \frac{1}{4} \|\widetilde{\bm{y}}\|_{\mathbb H}^2.
\]
Finally, using \(\bm{y}^*, \bm{y}^\tau \in \mathcal{M}_1\), the remaining terms are bounded since 
\begin{align*}
- ( \mathfrak{H}^*, \widetilde{\bm{y}})_{\mathbb{H}}
( \bm{y}^*, \widetilde{\bm{y}})_{\mathbb{H}}
&\le \|\mathfrak{H}^*\|_{\mathbb H}\|\widetilde{\bm{y}}\|_{\mathbb H}\|\bm{y}^*\|_{\mathbb H}\|\widetilde{\bm{y}}\|_{\mathbb H}
\le \|\mathfrak{H}^*\|_{\mathbb H}\|\widetilde{\bm{y}}\|_{\mathbb H}^2,\\
- ( \mathfrak{H}^*, \bm{y}^{\tau})_{\mathbb{H}}
\|\widetilde{\bm{y}}\|_{\mathbb{H}}^2
&\le \|\mathfrak{H}^*\|_{\mathbb H}\|\bm{y}^\tau\|_{\mathbb H}\|\widetilde{\bm{y}}\|_{\mathbb H}^2
\le \|\mathfrak{H}^*\|_{\mathbb H}\|\widetilde{\bm{y}}\|_{\mathbb H}^2.
\end{align*}
Combining the above estimates and absorbing the terms like \(\frac{\nu}{6}\|\nabla \widetilde{\bm{y}}(t)\|_{\mathbb{H}}^2\) into the left-hand side of \eqref{tau-est-1}, we obtain
\begin{align}\label{final-ineq}
&\frac{1}{2}\frac{d}{dt}\|\widetilde{\bm{y}}(t)\|_{\mathbb{H}}^2 + \frac{2\nu}{3} \|\nabla \widetilde{\bm{y}}(t)\|_{\mathbb{H}}^2
\nonumber\\&\le C_{\nu,\beta} \big(1 + \|\nabla \bm{y}^\tau(t)\|_{\mathbb{H}}^2 + \|\nabla \bm{y}^*(t)\|_{\mathbb{H}}^2  + \|\mathfrak{H}^*(t)\|_{\mathbb H}\big) \|\widetilde{\bm{y}}(t)\|_{\mathbb{H}}^2
+ C_\beta \|\widetilde{\theta}(t)\|_{H}^2 + C \tau^2 \|\mathfrak{H}(t)\|_{\mathbb H}^2,
\end{align}
for a.e. \(t\in[0,T]\).
Now taking the \(H\)-inner product of the second equation in \eqref{deff} with \(\widetilde{\theta}(\cdot)\), we obtain
\begin{align}
\frac{1}{2}\frac{d}{dt}\|\widetilde{\theta}(t)\|_{H}^2
+ \kappa \|\nabla \widetilde{\theta}(t)\|_{\mathbb H}^2
&= \kappa \|\nabla \theta^*(t)\|_{\mathbb H}^2 \|\widetilde{\theta}(t)\|_{H}^2 
+ \kappa \Big(
\|\nabla \theta^\tau(t)\|_{\mathbb H}^2 - \|\nabla \theta^*(t)\|_{\mathbb H}^2
\Big)
( \theta^\tau(t), \widetilde{\theta}(t) )_{H} \nonumber\\ &\quad - \Big(
\bm{y}^\tau(t)\cdot\nabla \theta^\tau(t)
- \bm{y}^*(t)\cdot\nabla \theta^*(t),
\widetilde{\theta}(t)
\Big)_{H},
\end{align}
for a.e. \(t\in[0,T]\).
Using \( \bm{y}^\tau \cdot \nabla\theta^\tau-\bm{y}^* \cdot \nabla\theta^*
= \widetilde{\bm{y}} \cdot \nabla\theta^\tau+\bm{y}^* \cdot \nabla\widetilde{\theta} \)
and \( (\bm{y}^* \cdot \nabla\widetilde{\theta},\widetilde{\theta})_H=0 \),
we write
\[
-(\bm{y}^\tau \cdot \nabla\theta^\tau-\bm{y}^* \cdot \nabla\theta^*,\widetilde{\theta})_H
= -(\widetilde{\bm{y}} \cdot \nabla\theta^\tau,\widetilde{\theta})_H - (\bm{y}^* \cdot \nabla\widetilde{\theta},\widetilde{\theta})_H
= -(\widetilde{\bm{y}} \cdot \nabla\theta^\tau,\widetilde{\theta})_H.
\]
We now estimate the remaining term using Hölder's, Ladyzhenskaya's, and Young's inequalities:
\begin{align}
\big|(\widetilde{\bm{y}}\cdot\nabla\theta^\tau,\widetilde{\theta})_{H}\big|
&=\big|(\widetilde{\bm{y}}\cdot\nabla\widetilde{\theta},\theta^\tau)_{H}\big|
\leq \|\widetilde{\bm{y}}\|_{\mathbb{L}^{4}}
\|\nabla\widetilde{\theta}\|_{\mathbb H}
\|\theta^\tau\|_{L^{4}}
\nonumber\\
&\leq \frac{\kappa}{8}\|\nabla\widetilde{\theta}\|_{\mathbb H}^{2}
+C_{\kappa}\|\widetilde{\bm{y}}\|_{\mathbb{H}}
\|\nabla\widetilde{\bm{y}}\|_{\mathbb{H}}
\|\theta^\tau\|_{L^{4}}^{2}
\nonumber\\
&\leq \frac{\kappa}{8}\|\nabla\widetilde{\theta}\|_{\mathbb H}^{2}
+\frac{\nu}{12}\|\nabla\widetilde{\bm{y}}\|_{\mathbb{H}}^{2}
+C_{\nu,\kappa}\|\theta^\tau\|_{L^{4}}^{4}
\|\widetilde{\bm{y}}\|_{\mathbb{H}}^{2},\\
\big|(\widetilde{\bm{y}}\cdot\nabla\theta^*,\widetilde{\theta})_{H}\big|
&=\big|(\widetilde{\bm{y}}\cdot\nabla\widetilde{\theta},\theta^*)_{H}\big|
\leq \|\widetilde{\bm{y}}\|_{\mathbb{L}^{4}}
\|\nabla\widetilde{\theta}\|_{\mathbb H}
\|\theta^*\|_{L^{4}}
\nonumber\\
&\leq \frac{\kappa}{8}\|\nabla\widetilde{\theta}\|_{\mathbb H}^{2}
+C_{\kappa}\|\widetilde{\bm{y}}\|_{\mathbb{H}}
\|\nabla\widetilde{\bm{y}}\|_{\mathbb{H}}
\|\theta^*\|_{L^{4}}^{2}
\nonumber\\
&\leq \frac{\kappa}{8}\|\nabla\widetilde{\theta}\|_{\mathbb H}^{2}
+\frac{\nu}{12}\|\nabla\widetilde{\bm{y}}\|_{\mathbb{H}}^{2}
+C_{\nu,\kappa}\|\theta^*\|_{L^{4}}^{4}
\|\widetilde{\bm{y}}\|_{\mathbb{H}}^{2}.
\end{align}
Finally, for the nonlinear term, we estimate
\begin{align}
&\kappa \Big(
\|\nabla \theta^\tau\|_{\mathbb H}^2 - \|\nabla \theta^*\|_{\mathbb H}^2
\Big)
( \theta^\tau, \widetilde{\theta} )_{H} \nonumber \\
&\le \kappa \big(\|\nabla \theta^\tau\|_{\mathbb H} + \|\nabla \theta^*\|_{\mathbb H}\big)
\|\nabla \widetilde{\theta}\|_{H} \|\theta^\tau\|_{H} \|\widetilde{\theta}\|_{H} \nonumber \\
&\le \frac{\kappa}{4}\|\nabla \widetilde{\theta}\|_{H}^2
+ C_{\kappa}\big(\|\nabla \theta^\tau\|_{\mathbb H} + \|\nabla \theta^*\|_{\mathbb H}\big)^2
\|\widetilde{\theta}\|_{H}^2.
\end{align}
Combining the above estimates, we obtain
\begin{align}{\label{theta_f}}
&\frac{1}{2}\frac{d}{dt}\|\widetilde{\theta}(t)\|_{H}^{2}
+\frac{\kappa}{2}\|\nabla\widetilde{\theta}(t)\|_{H}^{2}
\nonumber\\&\leq
C_{\kappa}\Big(
\|\nabla\theta^\tau(t)\|_{\mathbb H}^{2}
+\|\nabla\theta^*(t)\|_{\mathbb H}^{2}
+1
\Big)
\|\widetilde{\theta}(t)\|_{H}^{2}
+C_{\nu,\kappa}\|\theta^*\|_{L^{4}}^{4}\|\widetilde{\bm{y}}(t)\|_{\mathbb{H}}^{2}
+\frac{\nu}{6}
\|\nabla\widetilde{\bm{y}}(t)\|_{\mathbb{H}}^{2},
\end{align}
for a.e. \(t\in[0,T]\).
Adding \eqref{final-ineq} and \eqref{theta_f}, we obtain for a.e. \(t\in[0,T]\)
\begin{align}
&\frac{d}{dt}\Big(\|\widetilde{\bm{y}}(t)\|_{\mathbb{H}}^2 + \|\widetilde{\theta}(t)\|_{H}^2\Big)
+ \nu\|\nabla \widetilde{\bm{y}}(t)\|_{\mathbb{H}}^2
+ \kappa\|\nabla \widetilde{\theta}(t)\|_{\mathbb H}^2 \nonumber\\
&\le C_{\nu,\beta,\kappa}\Big(
1 + \|\nabla \bm{y}^\tau(t)\|_{\mathbb{H}}^2
+ \|\nabla \bm{y}^*(t)\|_{\mathbb{H}}^2
+ \|\nabla \theta^\tau(t)\|_{\mathbb H}^2
+ \|\nabla \theta^*(t)\|_{\mathbb H}^2+\|\theta^*(t)\|_{L^{4}}^{4}
+ \|\mathfrak{H}^*(t)\|_{\mathbb{H}}
\Big)\nonumber\\
&\qquad\times\Big(\|\widetilde{\bm{y}}(t)\|_{\mathbb{H}}^2 + \|\widetilde{\theta}(t)\|_{H}^2\Big)  + C_{\nu}\tau^2\|\mathfrak{H}(t)\|_{\mathbb{H}}^2.
\end{align}
Consequently, we obtain
\begin{align}{\label{ytest}}
\sup_{0\le t\le T}&\left(\|\widetilde{\bm{y}}(t)\|_{\mathbb{H}}^2 + \|\widetilde{\theta}(t)\|_{H}^2\right)
+ \nu \int_0^T \|\nabla \widetilde{\bm{y}}(t)\|_{\mathbb{H}}^2 dt
+ \kappa \int_0^T \|\nabla \widetilde{\theta}(t)\|_{\mathbb H}^2 dt \notag \\
&\le \bigg(C_{\nu,\beta,\kappa}\,\tau^2
\int_0^T \|\mathfrak{H}(s)\|_{\mathbb{H}}^2 ds\bigg)
\exp\!\left(
C_{\nu,\beta,\kappa}\int_0^T \Big(
1 + \|\nabla \bm{y}^\tau(s)\|_{\mathbb{H}}^2 + \|\nabla \bm{y}^*(s)\|_{\mathbb{H}}^2 \right. \nonumber\\
&\hspace{1cm}
+ \|\nabla \theta^\tau(s)\|_{H}^2 + \|\nabla \theta^*(s)\|_{\mathbb H}^2+\|\theta^*(s)\|_{L^{4}}^{4}
+ \|\mathfrak{H}^*(s)\|_{\mathbb{H}}
\Big)\,ds
\Bigg)\notag\\
&\le K \tau^2 \int_0^T \|\mathfrak{H}(s)\|_{\mathbb{H}}^2 ds,
\end{align}
for all \(t \in [0, T]\), provided that the initial data satisfy \(\|\widetilde{\bm{y}}(0)\|_{\mathbb{H}}^2 + \|\widetilde{\theta}(0)\|_{H}^2 = 0\), where
\[
K = K\Big(\nu,\beta,\kappa,T,\|\bm{y}_0\|_{\mathbb{V}},\|\theta_0\|_{V},\|\mathfrak{H}^*\|_{L^2(0,T;\mathbb{H})},\|\mathfrak{H}\|_{L^2(0,T;\mathbb{H})}\Big)
\]
is a positive constant depending on the indicated terms.
Taking the inner product in \(\mathbb{H}\) of the velocity equation in \eqref{deff} with \(A\widetilde{\bm{y}}(\cdot)\), we obtain for a.e. \(t \in [0,T]\),
\begin{align}
&\frac{1}{2}\frac{d}{dt}\|\nabla \widetilde{\bm{y}}(t)\|_{\mathbb{H}}^2
+ \nu \|A\widetilde{\bm{y}}(t)\|_{\mathbb{H}}^2
\\&= \nu \|\nabla \bm{y}^*(t)\|_{\mathbb{H}}^2 \|\nabla \widetilde{\bm{y}}(t)\|_{\mathbb{H}}^2 + \nu \Big(
\|\nabla \bm{y}^\tau(t)\|_{\mathbb{H}}^2 - \|\nabla \bm{y}^*(t)\|_{\mathbb{H}}^2
\Big)
( \nabla \bm{y}^\tau(t), \nabla \widetilde{\bm{y}}(t) )_{\mathbb{H}}+ \beta ( {\textbf{Proj}_{\mathbb H}}(\widetilde{\theta}(t) e_2), A\widetilde{\bm{y}}(t) )_{\mathbb{H}} \nonumber \\
&\quad -\Big(
B(\bm{y}^\tau(t),\bm{y}^\tau(t)) - B(\bm{y}^*(t),\bm{y}^*(t)),
A\widetilde{\bm{y}}(t)
\Big)_{\mathbb{H}}
 - \beta \Big(
( \theta^\tau(t) e_2,\bm{y}^\tau(t))_{\mathbb{H}}\bm{y}^\tau(t)
- ( \theta^*(t) e_2,\bm{y}^*(t))_{\mathbb{H}}\bm{y}^*(t),
A\widetilde{\bm{y}}(t)
\Big)_{\mathbb{H}} \nonumber \\
&\quad + \tau ( \mathfrak{H}(t), A\widetilde{\bm{y}}(t) )_{\mathbb{H}} - \Big(( \mathfrak{H}^*(t)+\tau \mathfrak{H}(t),\bm{y}^\tau(t))_{\mathbb{H}}\bm{y}^\tau(t)
- ( \mathfrak{H}^*(t),\bm{y}^*(t))_{\mathbb{H}}\bm{y}^*(t),
A\widetilde{\bm{y}}(t)
\Big)_{\mathbb{H}}.
\end{align}
Using the Cauchy--Schwarz inequality, we obtain
\[
\begin{aligned}
\nu\Big(\|\nabla\bm{y}^\tau\|_{\mathbb H}^2-\|\nabla\bm{y}^*\|_{\mathbb H}^2\Big)
(\nabla\bm{y}^\tau,\nabla\widetilde{\bm{y}})_{\mathbb H}&\leq
\nu\Big|\|\nabla\bm{y}^\tau\|_{\mathbb H}^2-\|\nabla\bm{y}^*\|_{\mathbb H}^2\Big|
\|\nabla\bm{y}^\tau\|_{\mathbb H}
\|\nabla\widetilde{\bm{y}}\|_{\mathbb H}\\
& \leq
\nu\big(\|\nabla\bm{y}^\tau\|_{\mathbb H}
+\|\nabla\bm{y}^*\|_{\mathbb H}\big)
\|\nabla\bm{y}^\tau\|_{\mathbb H}
\|\nabla\widetilde{\bm{y}}\|_{\mathbb H}^2.
\end{aligned}
\]
Next, using Agmon's and Young's inequalities, we infer
\begin{align}
|(B(\widetilde{\bm{y}},\bm{y}^*),
A\widetilde{\bm{y}})_{\mathbb H}|
&\le
\frac{\nu}{16}\|A\widetilde{\bm{y}}\|_{\mathbb H}^2
+C_\nu\|\nabla\bm{y}^*\|_{\mathbb H}^4
\|\widetilde{\bm{y}}\|_{\mathbb H}^2,\\
|(B(\bm{y}^\tau,\widetilde{\bm{y}}),
A\widetilde{\bm{y}})_{\mathbb H}|
&\le
\frac{\nu}{16}\|A\widetilde{\bm{y}}\|_{\mathbb H}^2
+C_\nu\|\bm{y}^\tau\|_{\mathbb L^\infty}^2
\|\nabla\widetilde{\bm{y}}\|_{\mathbb H}^2.
\end{align}

Using the Cauchy-Schwarz and Young inequalities, we estimate
\begin{align}
\tau ( \mathfrak{H}, A\widetilde{\bm{y}} )_{\mathbb{H}}
&\le \tau \|\mathfrak{H}\|_{\mathbb{H}} \|A\widetilde{\bm{y}}\|_{\mathbb{H}}
\le \frac{\nu}{16}\|A\widetilde{\bm{y}}\|_{\mathbb{H}}^2
+ C_{\nu}\tau^2\|\mathfrak{H}\|_{\mathbb{H}}^2.
\end{align}
Similarly, using \(\bm{y}^\tau \in \mathcal{M}_1\), we obtain
\begin{align}
- \tau ( \mathfrak{H}, \bm{y}^\tau )_{\mathbb{H}} ( \bm{y}^\tau, A\widetilde{\bm{y}} )_{\mathbb{H}}
&\le \tau \|\mathfrak{H}\|_{\mathbb{H}} \|A\widetilde{\bm{y}}\|_{\mathbb{H}}
\le \frac{\nu}{16}\|A\widetilde{\bm{y}}\|_{\mathbb{H}}^2
+ C_{\nu}\tau^2\|\mathfrak{H}\|_{\mathbb{H}}^2.
\end{align}
Moreover, using \(\bm{y}^\tau, \bm{y}^* \in \mathcal{M}_1\), we estimate
\begin{align}
&-(\mathfrak{H}^*,\widetilde{\bm{y}})_{\mathbb H}
(\bm{y}^\tau,A\widetilde{\bm{y}})_{\mathbb H}
-(\mathfrak{H}^*,\bm{y}^*)_{\mathbb H}
\|\nabla\widetilde{\bm{y}}\|_{\mathbb H}^2
\notag\\
&\quad\le
\|\mathfrak{H}^*\|_{\mathbb H}
\|\widetilde{\bm{y}}\|_{\mathbb H}
\|A\widetilde{\bm{y}}\|_{\mathbb H}
+\|\mathfrak{H}^*\|_{\mathbb H}
\|\nabla\widetilde{\bm{y}}\|_{\mathbb H}^2
\notag\\
&\quad\le
\frac{\nu}{16}\|A\widetilde{\bm{y}}\|_{\mathbb H}^2
+C_\nu\|\mathfrak{H}^*\|_{\mathbb H}^2
\|\widetilde{\bm{y}}\|_{\mathbb H}^2
+\|\mathfrak{H}^*\|_{\mathbb H}
\|\nabla\widetilde{\bm{y}}\|_{\mathbb H}^2.
\end{align}
For the temperature coupling terms involving \(A\widetilde{\bm{y}}\), we deduce
\begin{align}
\beta ( {\textbf{Proj}}_{\mathbb H}(\widetilde{\theta} e_2), A\widetilde{\bm{y}} )_{\mathbb{H}}
&\le \beta \|\widetilde{\theta}\|_{H} \|A\widetilde{\bm{y}}\|_{\mathbb{H}}
\le \frac{\nu}{16}\|A\widetilde{\bm{y}}\|_{\mathbb{H}}^2
+ C_{\nu,\beta} \|\widetilde{\theta}\|_{H}^2,
\end{align}
and
\begin{align}
-& \beta \Big(
( \theta^\tau e_2,\bm{y}^\tau)_{\mathbb{H}}\bm{y}^\tau
- ( \theta^* e_2,\bm{y}^*)_{\mathbb{H}}\bm{y}^*,
A\widetilde{\bm{y}}
\Big)_{\mathbb{H}} \nonumber\\
&= - \beta \Big(
(\widetilde{\theta} e_2,\bm{y}^\tau)_{\mathbb{H}}\bm{y}^*
+ ( \theta^* e_2,\widetilde{\bm{y}})_{\mathbb{H}}\bm{y}^*
+ ( \theta^{\tau} e_2,\bm{y}^{\tau})_{\mathbb{H}}\widetilde{\bm{y}},
A\widetilde{\bm{y}}
\Big)_{\mathbb{H}} \nonumber\\
&\le \frac{\nu}{16}\|A\widetilde{\bm{y}}\|_{\mathbb H}^2
+ C_{\nu,\beta} \|\widetilde{\theta}\|_{H}^2
+ C_{\nu,\beta} \|\widetilde{\bm{y}}\|_{\mathbb H}^2.
\end{align}
where we have used \(\bm{y}^\tau, \bm{y}^* \in \mathcal{M}_1\), \(\theta^*, \theta^\tau \in \mathcal{M}_2\), and Young's inequality.
Combining all the above estimates, we find for a.e. \(t\in[0,T]\),
\begin{align}{\label{autilde}}
&\frac12\frac{d}{dt}
\|\nabla\widetilde{\bm{y}}(t)\|_{\mathbb H}^2
+\nu\|A\widetilde{\bm{y}}(t)\|_{\mathbb H}^2
\notag\\
&\quad\le
C_\nu\Big(
\|\bm{y}^\tau(t)\|_{\mathbb L^\infty}^2
+\|\nabla\bm{y}^\tau(t)\|_{\mathbb H}^2
+\|\nabla\bm{y}^*(t)\|_{\mathbb H}^2
+\|\mathfrak{H}^*(t)\|_{\mathbb H}
\Big)\|\nabla\widetilde{\bm{y}}(t)\|_{\mathbb H}^2
+\frac{7\nu}{16}\|A\widetilde{\bm{y}}(t)\|_{\mathbb H}^2
\notag\\
&\qquad+C_{\nu,\beta}\|\widetilde\theta(t)\|_H^2
+C_{\nu,\beta}\Big(
1+\|\mathfrak{H}^*(t)\|_{\mathbb H}^2
+\|\nabla\bm{y}^*(t)\|_{\mathbb H}^4
\Big)\|\widetilde{\bm{y}}(t)\|_{\mathbb H}^2
+C_\nu\tau^2\|\mathfrak{H}(t)\|_{\mathbb H}^2.
\end{align}
Taking the inner product in \(H\) of the second equation in system \eqref{deff} with \({\mathcal{A}\widetilde{\theta}}(\cdot)\), we obtain for a.e. \(t\in[0,T]\),
\begin{align}
\frac{1}{2}\frac{d}{dt}\|\nabla \widetilde{\theta}(t)\|_{\mathbb H}^2
+ \kappa \|{\mathcal{A}\widetilde{\theta}}(t)\|_{H}^2
&= \kappa \|\nabla \theta^*(t)\|_{\mathbb H}^2 \|\nabla \widetilde{\theta}(t)\|_{\mathbb H}^2 + \kappa\Big(
\|\nabla \theta^\tau(t)\|_{\mathbb H}^2 - \|\nabla \theta^*(t)\|_{\mathbb H}^2
\Big)( \nabla \theta^\tau(t), \nabla \widetilde{\theta}(t))_{H} \nonumber\\
&\quad- ( \bm{y}^\tau(t)\cdot\nabla\theta^\tau(t) - \bm{y}^*(t)\cdot\nabla \theta^*(t),
{\mathcal{A}\widetilde{\theta}}(t))_{H}.
\end{align}
Applying H\"older's and Young's inequalities, we infer
\[
\kappa\Big(\|\nabla\theta^\tau(t)\|_{\mathbb H}^2-\|\nabla\theta^*(t)\|_{\mathbb H}^2\Big)
(\nabla\theta^\tau(t),\nabla\widetilde{\theta}(t))_{H}
\leq \frac{\kappa}{6}\|\mathcal{A}\widetilde{\theta}(t)\|_{H}^2
+C_{\kappa}\big(\|\nabla\theta^\tau(t)\|_{\mathbb H}+\|\nabla\theta^*(t)\|_{\mathbb H}\big)^2
\|\nabla\widetilde{\theta}(t)\|_{H}^2.
\]
Next, using \(\bm{y}^\tau\cdot\nabla\theta^\tau-\bm{y}^*\cdot\nabla\theta^*
=\widetilde{\bm{y}}\cdot\nabla\theta^\tau+\bm{y}^*\cdot\nabla\widetilde{\theta}\), we obtain
\[
-(\bm{y}^\tau\cdot\nabla\theta^\tau-\bm{y}^*\cdot\nabla\theta^*,\mathcal{A}\widetilde{\theta})_{H}
=-(\widetilde{\bm{y}}\cdot\nabla\theta^\tau,\mathcal{A}\widetilde{\theta})_{H}
-(\bm{y}^*\cdot\nabla\widetilde{\theta},\mathcal{A}\widetilde{\theta})_{H}.
\]
Using Hölder's, Agmon's, and Young's inequalities, we further have
\begin{align}
|(\widetilde{\bm{y}}\cdot\nabla \theta^\tau, {\mathcal{A}\widetilde{\theta}})_{H}|
&\le \|\widetilde{\bm{y}}\|_{\mathbb L^\infty}\|\nabla \theta^\tau\|_{\mathbb H}\|{\mathcal{A}\widetilde{\theta}}\|_{H} \nonumber\\
&\le C \|\widetilde{\bm{y}}\|_{\mathbb{H}}^{1/2}\|A\widetilde{\bm{y}}\|_{\mathbb{H}}^{1/2}
\|\nabla \theta^\tau\|_{\mathbb H}\|{\mathcal{A}\widetilde{\theta}}\|_{H} \nonumber\\
&\le \frac{\kappa}{6}\|{\mathcal{A}\widetilde{\theta}}\|_{H}^2
+ \frac{\nu}{16}\|A\widetilde{\bm{y}}\|_{\mathbb{H}}^2
+ C_{\nu,\kappa}\|\nabla \theta^\tau\|_{\mathbb H}^4
\|\widetilde{\bm{y}}\|_{\mathbb{H}}^2.
\end{align}
Similarly, using Agmon's and Young's inequalities, we get
\begin{align*}
|(\bm{y}^*\cdot\nabla\widetilde\theta,
\mathcal A\widetilde\theta)_H|
&\le
\|\bm{y}^*\|_{\mathbb L^\infty}
\|\nabla\widetilde\theta\|_H
\|\mathcal A\widetilde\theta\|_H\\
&\le
\frac{\kappa}{6}\|\mathcal A\widetilde\theta\|_H^2
+C_\kappa\|\bm{y}^*\|_{\mathbb L^\infty}^2
\|\nabla\widetilde\theta\|_H^2\\
&\le
\frac{\kappa}{6}\|\mathcal A\widetilde\theta\|_H^2
+C_\kappa\|\bm{y}^*\|_{\mathbb H}
\|A\bm{y}^*\|_{\mathbb H}
\|\nabla\widetilde\theta\|_H^2.
\end{align*}
Combining the above estimates, we deduce
\begin{align}{\label{thl}}
\frac12\frac{d}{dt}\|\nabla\widetilde\theta(t)\|_H^2
+\frac{\kappa}{2}\|\mathcal A\widetilde\theta(t)\|_H^2
&\le
C_\kappa\Big(
\|\nabla\theta^\tau(t)\|_H^2
+\|\nabla\theta^*(t)\|_H^2
+\|\bm{y}^*(t)\|_{\mathbb H}\|A\bm{y}^*(t)\|_{\mathbb H}
\Big)\|\nabla\widetilde\theta(t)\|_H^2
\notag\\
&\qquad
+C_{\nu,\kappa}\|\nabla\theta^\tau(t)\|_H^4
\|\widetilde{\bm{y}}(t)\|_{\mathbb H}^2
+\frac{\nu}{16}\|A\widetilde{\bm{y}}(t)\|_{\mathbb H}^2.
\end{align}
for a.e. \(t \in [0,T]\).
Adding \eqref{autilde} and \eqref{thl}, and absorbing the terms involving \(\|A\widetilde{\bm{y}}(t)\|_{\mathbb{H}}^2\) and \(\|{\mathcal{A}\widetilde{\theta}}(t)\|_{H}^2\) to the left hand side, we obtain for a.e. \(t\in[0,T]\),
\begin{align}{\label{authe}}
&\frac12\frac{d}{dt}\Big(
\|\nabla\widetilde{\bm{y}}(t)\|_{\mathbb H}^2
+\|\nabla\widetilde\theta(t)\|_H^2
\Big)
+\frac{\nu}{2}\|A\widetilde{\bm{y}}(t)\|_{\mathbb H}^2
+\frac{\kappa}{2}\|\mathcal A\widetilde\theta(t)\|_H^2
\notag\\
&\quad\le
C_\nu\Big(
\|\bm{y}^\tau(t)\|_{\mathbb L^\infty}^2
+\|\nabla\bm{y}^\tau(t)\|_{\mathbb H}^2
+\|\nabla\bm{y}^*(t)\|_{\mathbb H}^2
+\|\mathfrak{H}^*(t)\|_{\mathbb H}
\Big)\|\nabla\widetilde{\bm{y}}(t)\|_{\mathbb H}^2
\notag\\
&\qquad
+C_\kappa\Big(
\|\nabla\theta^\tau(t)\|_H^2
+\|\nabla\theta^*(t)\|_H^2
+\|\bm{y}^*(t)\|_{\mathbb H}\|A\bm{y}^*(t)\|_{\mathbb H}
\Big)\|\nabla\widetilde\theta(t)\|_H^2
\notag\\
&\qquad
+C_{\nu,\beta}\|\widetilde\theta(t)\|_H^2
+C_{\nu,\beta,\kappa}\Big(
1+\|\mathfrak{H}^*(t)\|_{\mathbb H}^2
+\|\nabla\bm{y}^*(t)\|_{\mathbb H}^4
+\|\nabla\theta^\tau(t)\|_H^4
\Big)\|\widetilde{\bm{y}}(t)\|_{\mathbb H}^2
+C_\nu\tau^2\|\mathfrak{H}(t)\|_{\mathbb H}^2,
\end{align}
for a.e. \(t \in [0,T]\).
From the previous estimate \eqref{ytest}, we have
\begin{align}
\sup_{0\le t\le T}\left(\|\widetilde{\bm{y}}(t)\|_{\mathbb{H}}^2 + \|\widetilde{\theta}(t)\|_{H}^2\right)
&+ \nu \int_0^T \|\nabla \widetilde{\bm{y}}(t)\|_{\mathbb{H}}^2 dt
+ \kappa \int_0^T \|\nabla \widetilde{\theta}(t)\|_{\mathbb H}^2 dt \le K \tau^2 \int_0^T \|\mathfrak{H}(s)\|_{\mathbb{H}}^2 ds,
\end{align}
for all \(t \in [0,T]\).
From \eqref{autilde} and \eqref{thl}, we obtain for a.e. \(t\in[0,T]\),
\[
\frac{d}{dt}\Big(\|\nabla \widetilde{\bm{y}}(t)\|_{\mathbb{H}}^2 + \|\nabla \widetilde{\theta}(t)\|_{\mathbb H}^2\Big)
\le \Phi(t) \Big(\|\nabla \widetilde{\bm{y}}(t)\|_{\mathbb{H}}^2 + \|\nabla \widetilde{\theta}(t)\|_{\mathbb H}^2\Big)
+ S(t)\big(\|\widetilde{\bm{y}}(t)\|_{\mathbb{H}}^2 + \|\widetilde{\theta}(t)\|_{H}^2\big)
+ C_{\nu}\tau^2\|\mathfrak{H}(t)\|_{\mathbb{H}}^2,
\]
where
\begin{align*}
\Phi(t)
&:=C_{\nu,\kappa}\Big(
1+\|\bm{y}^\tau(t)\|_{\mathbb L^\infty}^2
+\|\nabla\bm{y}^\tau(t)\|_{\mathbb H}^2
+\|\nabla\bm{y}^*(t)\|_{\mathbb H}^2
+\|\mathfrak{H}^*(t)\|_{\mathbb H}\\
&\qquad\qquad
+\|\nabla\theta^\tau(t)\|_{\mathbb H}^2
+\|\nabla\theta^*(t)\|_{\mathbb H}^2
+\|\bm{y}^*(t)\|_{\mathbb H}\|A\bm{y}^*(t)\|_{\mathbb H}
\Big),\\
S(t)
&:=C_{\nu,\beta,\kappa}\Big(
1+\|\mathfrak{H}^*(t)\|_{\mathbb H}^2
+\|\nabla\bm{y}^*(t)\|_{\mathbb H}^4
+\|\nabla\theta^\tau(t)\|_H^4
\Big).
\end{align*}
Since \(\|\nabla \widetilde{\bm{y}}(0)\|_{\mathbb{H}}^2 + \|\nabla \widetilde{\theta}(0)\|_{H}^2 = 0\), applying Grönwall's inequality to the above differential inequality yields
\[
\begin{aligned}
&\sup_{0\le t\le T}\left(\|\nabla \widetilde{\bm{y}}(t)\|_{\mathbb{H}}^2 + \|\nabla \widetilde{\theta}(t)\|_{\mathbb H}^2\right) \le C_{\nu}\tau^2 \int_0^T \|\mathfrak{H}(s)\|_{\mathbb{H}}^2
\exp\left(\int_s^T \Phi(r)\,dr\right) ds \\
&\quad + \int_0^T S(s)\big(\|\widetilde{\bm{y}}(s)\|_{\mathbb{H}}^2 + \|\widetilde{\theta}(s)\|_{H}^2\big)
\exp\left(\int_s^T \Phi(r)\,dr\right) ds.
\end{aligned}
\]
Hence we deduce
\[
\sup_{0\le t\le T}\left(\|\nabla \widetilde{\bm{y}}(t)\|_{\mathbb{H}}^2 + \|\nabla \widetilde{\theta}(t)\|_{\mathbb H}^2\right)
+ \nu \int_0^T \|A\widetilde{\bm{y}}(t)\|_{\mathbb{H}}^2 dt
+ \kappa \int_0^T \|{\mathcal{A}\widetilde{\theta}}(t)\|_{H}^2 dt
\le K^{(1)}\,\tau^2 \int_0^T \|\mathfrak{H}(s)\|_{\mathbb{H}}^2\,ds,
\]
where \(K^{(1)}\) is a positive constant depending on the same parameters as \(K\), namely,
\[
K^{(1)} = K^{(1)}\Big(\nu,\beta,\kappa,T,\|\bm{y}_0\|_{\mathbb{V}},\|\theta_0\|_{V},\|\mathfrak{H}^*\|_{L^2(0,T;\mathbb{H})},\|\mathfrak{H}\|_{L^2(0,T;\mathbb{H})}\Big),
\]
which completes the proof. 
\end{proof}
\begin{proposition}{\label{prop4.4}}
Let \((\bm{y}_0,\theta_0) \in (\mathbb{V}\cap \mathcal{M}_1) \times (V\cap \mathcal{M}_2)\) be given. Then the control-to-state mapping
\[
\mathcal{U}_{ad} \ni \mathfrak{H}^* \mapsto (\bm{y}_{\mathfrak{H}^*}, \theta_{\mathfrak{H}^*})
\in \big(C([0,T];\mathbb{H})\cap L^2(0,T;\mathbb{V})\big)
\times \big(C([0,T];H)\cap L^2(0,T;V)\big)
\]
is Fréchet differentiable.
Moreover, for any \(\mathfrak{H} \in \mathcal{U}_{ad}\), we have
\begin{align}{\label{inc1}}
\lim_{\tau \downarrow 0}
\frac{\|\bm{y}_{\mathfrak{H}^*+\tau \mathfrak{H}}(t)-\bm{y}_{\mathfrak{H}^*}(t)-\tau\bm{\eta}(t)\|_{\mathbb{H}}}{\tau} = 0,
\quad
\lim_{\tau \downarrow 0}
\frac{\|\theta_{\mathfrak{H}^*+\tau \mathfrak{H}}(t)-\theta_{\mathfrak{H}^*}(t)-\tau\xi(t)\|_{H}}{\tau} = 0,
\end{align}
for all \(t\in[0,T]\), and
\begin{align}{\label{inc2}}
\lim_{\tau \downarrow 0}
\frac{\|\bm{y}_{\mathfrak{H}^*+\tau \mathfrak{H}}-\bm{y}_{\mathfrak{H}^*}-\tau\bm{\eta}\|_{L^2(0,T;\mathbb{V})}}{\tau} = 0,
\quad
\lim_{\tau \downarrow 0}
\frac{\|\theta_{\mathfrak{H}^*+\tau \mathfrak{H}}-\theta_{\mathfrak{H}^*}-\tau\xi\|_{L^2(0,T;V)}}{\tau} = 0.
\end{align}

Here, \((\bm{\eta},\xi)\) denotes the unique weak solution to the linearized system:
\begin{equation}
\begin{cases}
\dfrac{d\bm{\eta}(t)}{dt}
= -\nu A\bm{\eta}(t) 
+ \nu \|\nabla \bm{y}_{\mathfrak{H}^*}(t)\|_{\mathbb{H}}^2 \bm{\eta}(t)
+ 2\nu ( \nabla \bm{y}_{\mathfrak{H}^*}(t), \nabla \bm{\eta}(t) )_{\mathbb{H}} \bm{y}_{\mathfrak{H}^*}(t)- B'(\bm{y}_{\mathfrak{H}^*}(t))\bm{\eta}(t) 
+ \beta {\textbf{Proj}_{\mathbb H}}(\xi(t) e_2) \\
\qquad - \beta \Big[
( \xi(t) e_2,\bm{y}_{\mathfrak{H}^*}(t))_{\mathbb{H}} \bm{y}_{\mathfrak{H}^*}(t)
+ ( \theta_{\mathfrak{H}^*}(t) e_2,\bm{\eta}(t))_{\mathbb{H}} \bm{y}_{\mathfrak{H}^*}(t)
+ ( \theta_{\mathfrak{H}^*}(t) e_2,\bm{y}_{\mathfrak{H}^*}(t))_{\mathbb{H}} \bm{\eta}(t)
\Big] \\
\qquad + \mathfrak{H}(t) - ( \mathfrak{H}(t),\bm{y}_{\mathfrak{H}^*}(t))_{\mathbb{H}} \bm{y}_{\mathfrak{H}^*}(t)  - \Big[
( \mathfrak{H}^*(t),\bm{y}_{\mathfrak{H}^*}(t))_{\mathbb{H}} \bm{\eta}(t)
+ ( \mathfrak{H}^*(t),\bm{\eta}(t))_{\mathbb{H}} \bm{y}_{\mathfrak{H}^*}(t)
\Big],\ \text{ in }\ \mathbb{V}^{\prime},
\\[6pt]
\dfrac{d\xi(t)}{dt}
= -\kappa A\xi(t)
+ \kappa \|\nabla \theta_{\mathfrak{H}^*}(t)\|_{\mathbb H}^2 \xi(t)
+ 2\kappa ( \nabla \theta_{\mathfrak{H}^*}(t), \nabla \xi(t) )_{H} \theta_{\mathfrak{H}^*}(t) - \bm{\eta}(t) \cdot \nabla \theta_{\mathfrak{H}^*}(t)
\\ \qquad - \bm{y}_{\mathfrak{H}^*}(t) \cdot \nabla \xi(t),\  \text{ in } \ V',
\\
\bm{\eta}(0)=\boldsymbol{0}, \ \ \xi(0)=0,
\end{cases}
\end{equation}
for a.e.\ \(t \in [0,T]\), where \((\bm{y}_{\mathfrak{H}^*},\theta_{\mathfrak{H}^*})\) and \((\bm{y}_{\mathfrak{H}^*+\tau \mathfrak{H}},\theta_{\mathfrak{H}^*+\tau \mathfrak{H}})\) denote the unique strong solutions of \eqref{eq:constrained_Boussinesq} corresponding to the controls \(\mathfrak{H}^*\) and \(\mathfrak{H}^*+\tau \mathfrak{H}\), respectively.
Furthermore, \(\bm{\eta}(t) \in T_{\bm{y}_{\mathfrak{H}^*}(t)}\mathcal{M}_1\) and \(\xi(t) \in T_{\theta_{\mathfrak{H}^*}(t)}\mathcal{M}_2\) for all \(t\in[0,T]\), i.e.,
\[
( \bm{\eta}(t), \bm{y}_{\mathfrak{H}^*}(t))_{\mathbb{H}} = 0,
\ \
( \xi(t), \theta_{\mathfrak{H}^*}(t))_{H} = 0,
\]
so that \((\bm{\eta},\xi)\) lies in the tangent cone of admissible perturbations.
\end{proposition}
\begin{proof}
We show that the mapping
\[
\mathcal{U}_{ad} \ni \mathfrak{H}^* \mapsto (\bm{y}_{\mathfrak{H}^*},\theta_{\mathfrak{H}^*})
\]
is Fréchet differentiable, it suffices to verify the conditions \eqref{inc1} and \eqref{inc2}.
Set \(\bm{y}^\tau := \bm{y}_{\mathfrak{H}^*+\tau\mathfrak{H}}\), \(\theta^\tau := \theta_{\mathfrak{H}^*+\tau\mathfrak{H}}\), and
\(\widetilde{\bm{y}} := \bm{y}^\tau - \bm{y}_{\mathfrak{H}^*}\), \(\widetilde{\theta} := \theta^\tau - \theta_{\mathfrak{H}^*}\).
In addition, set \(\bm{\alpha} := \widetilde{\bm{y}} - \tau\bm{\eta}\) and \(z := \widetilde{\theta} - \tau\xi\).
\allowdisplaybreaks
The pair \((\bm{\alpha}, z)\) satisfies the following system for a.e. $t\in(0,T)$:
\begin{equation}\label{eq:y-z-system}
\begin{aligned}
\dfrac{d\bm{\alpha}(t)}{dt}
&= -\nu A\bm{\alpha}(t)
+ \nu \|\nabla \bm{y}^*(t)\|_{\mathbb{H}}^2\, \bm{\alpha}(t)  + \nu \Big(\|\nabla \bm{y}^{\tau}(t)\|_{\mathbb{H}}^{2}  - \|\nabla \bm{y}^*(t)\|_{\mathbb{H}}^{2}\Big)\,\widetilde{\bm{y}}(t) \\
&\quad + \nu \Big[
\Big(\|\nabla \bm{y}^{\tau}(t)\|_{\mathbb{H}}^{2}
      - \|\nabla \bm{y}^*(t)\|_{\mathbb{H}}^{2}\Big)
- 2( \nabla \bm{y}^*(t),\nabla \widetilde{\bm{y}}(t))_{\mathbb{H}}
\Big]\bm{y}^*(t) \\
&\quad
+ 2\nu ( \nabla \bm{y}^*(t),\nabla \bm{\alpha}(t))_{\mathbb{H}}\,\bm{y}^*(t)
- B'(\bm{y}^*(t))\,\bm{\alpha}(t)-\Big[
B(\bm{y}^{\tau}(t),\bm{y}^{\tau}(t))
- B(\bm{y}^*(t),\bm{y}^*(t))
- B'(\bm{y}^*(t))\,\widetilde{\bm{y}}(t)
\Big] \\
&\quad + \beta{\textbf{Proj}_{\mathbb H}}(z(t) e_2)
- \beta\Big[
(\widetilde{\theta}(t) e_2,\bm{y}^{\tau}(t))_{\mathbb{H}}\widetilde{\bm{y}}(t)
+ ( \theta^*(t) e_2,\widetilde{\bm{y}}(t))_{\mathbb{H}}\widetilde{\bm{y}}(t) \\
&\qquad + (\widetilde{\theta}(t) e_2,\widetilde{\bm{y}}(t))_{\mathbb{H}}\bm{y}^*(t)
+ ( z(t) e_2,\bm{y}^*(t))_{\mathbb H}\bm{y}^*(t)
+ ({\theta}^*(t) e_2,\bm{\alpha}(t))_{\mathbb H}\bm{y}^*(t)+ ( \theta^*(t) e_2,\bm{y}^*(t))_{\mathbb{H}}\bm{\alpha}(t)
\Big] \\
&\quad
\quad-\tau ( \mathfrak{H}(t),\widetilde{\bm{y}}(t))_{\mathbb{H}}\,\bm{y}^{\tau}(t)
-\tau ( \mathfrak{H}(t),\bm{y}^*(t))_{\mathbb{H}}\,\widetilde{\bm{y}}(t)
-( \mathfrak{H}^*(t),\bm{y}^*(t))_{\mathbb{H}}\,\bm{\alpha}(t) \\
&\quad
\quad-( \mathfrak{H}^*(t),\bm{\alpha}(t))_{\mathbb{H}}\,\bm{y}^*(t)
-( \mathfrak{H}^*(t),\widetilde{\bm{y}}(t))_{\mathbb{H}}\,\widetilde{\bm{y}}(t), \ \text{ in }\ \mathbb{V}^{\prime},\\[8pt]
\dfrac{dz(t)}{dt}
&= -\kappa A z(t)
+ \kappa \|\nabla \theta^*(t)\|_{\mathbb H}^2\, z(t) \\
&\quad + \kappa \Big(\|\nabla \theta^{\tau}(t)\|_{H}^{2} - \|\nabla \theta^*(t)\|_{\mathbb H}^{2}\Big)\,\widetilde{\theta}(t)  + \kappa \Big[
\Big(\|\nabla \theta^{\tau}(t)\|_{H}^{2} - \|\nabla \theta^*(t)\|_{\mathbb H}^{2}\Big)
- 2( \nabla \theta^*(t), \nabla \widetilde{\theta}(t))_{H}
\Big]\theta^*(t) \\
&\quad + 2\kappa ( \nabla \theta^*(t), \nabla z(t))_{H}\,\theta^*(t)
- \widetilde{\bm{y}}(t)\cdot\nabla\widetilde{\theta}(t)
-  \bm{y}^*(t)\cdot\nabla{z}(t)
- \bm{\alpha}(t)\cdot\nabla \theta^*(t), \ \text{ in }\ V', \\[6pt]
\bm{\alpha}(0) &= \bm 0, \ \ z(0) = 0.
\end{aligned}
\end{equation}
Taking the inner product in $\mathbb{H}$ of the first equation with $\bm{\alpha}(s)$ and integrating over $[0,t]$, we obtain
\allowdisplaybreaks
\begin{align}
&\frac{1}{2}\|\bm{\alpha}(t)\|_{\mathbb{H}}^2 + \nu \int_0^t \|\nabla \bm{\alpha}(s)\|_{\mathbb{H}}^2\,ds\notag
\\\quad&= \frac{1}{2}\|\bm{\alpha}(0)\|_{\mathbb{H}}^2
+ \nu \int_0^t \|\nabla \bm{y}^*(s)\|_{\mathbb{H}}^2 \|\bm{\alpha}(s)\|_{\mathbb{H}}^2\,ds+ \nu \int_0^t \Big(\|\nabla \bm{y}^{\tau}(s)\|_{\mathbb{H}}^2 - \|\nabla \bm{y}^*(s)\|_{\mathbb{H}}^2\Big) (\widetilde{\bm{y}}(s), \bm{\alpha}(s))_{\mathbb{H}}\,ds \notag \\
&\quad + \nu \int_0^t \Big[
\Big(\|\nabla \bm{y}^{\tau}(s)\|_{\mathbb{H}}^2 - \|\nabla \bm{y}^*(s)\|_{\mathbb{H}}^2\Big)
- 2( \nabla \bm{y}^*(s), \nabla \widetilde{\bm{y}}(s))_{\mathbb{H}}
\Big] ( \bm{y}^*(s), \bm{\alpha}(s))_{\mathbb{H}}\,ds \notag \\
&\quad + 2\nu \int_0^t ( \nabla \bm{y}^*(s), \nabla \bm{\alpha}(s))_{\mathbb{H}} ( \bm{y}^*(s), \bm{\alpha}(s))_{\mathbb{H}}\,ds - \int_0^t ( B'(\bm{y}^*(s)) \bm{\alpha}(s), \bm{\alpha}(s))_{\mathbb{H}}\,ds   \notag \\
&\quad- \int_0^t \Big( B(\bm{y}^{\tau}(s), \bm{y}^{\tau}(s)) - B(\bm{y}^*(s), \bm{y}^*(s)) - B'(\bm{y}^*(s)) \widetilde{\bm{y}}(s), \bm{\alpha}(s)\Big)_{\mathbb{H}}\,ds + \beta \int_0^t ( {\textbf{Proj}_{\mathbb H}}(z(s) e_2), \bm{\alpha}(s))_{\mathbb{H}}\,ds  \notag \\
&\quad - \beta \int_0^t \Big[
(\widetilde{\theta}(s) e_2, \bm{y}^{\tau}(s))_{\mathbb{H}} (\widetilde{\bm{y}}(s), \bm{\alpha}(s))_{\mathbb{H}} 
+ ( \theta^*(s) e_2, \widetilde{\bm{y}}(s))_{\mathbb{H}} (\widetilde{\bm{y}}(s), \bm{\alpha}(s))_{\mathbb{H}} + (\widetilde{\theta}(s) e_2, \widetilde{\bm{y}}(s))_{\mathbb{H}} ( \bm{y}^*(s), \bm{\alpha}(s))_{\mathbb{H}} 
 \notag \\
&\quad\quad+ ( z(s) e_2, \bm{y}^*(s))_{\mathbb{H}} ( \bm{y}^*(s), \bm{\alpha}(s))_{\mathbb{H}} + (\theta^*(s)e_2,\bm{\alpha}(s))_{\mathbb H}
(\bm{y}^*(s),\bm{\alpha}(s))_{\mathbb H} 
+ ( \theta^*(s) e_2, \bm{y}^*(s))_{\mathbb{H}} \|\bm{\alpha}(s)\|_{\mathbb{H}}^2
\Big]\,ds \notag \\
&\quad - \tau \int_0^t ( \mathfrak{H}(s), \widetilde{\bm{y}}(s))_{\mathbb{H}} ( \bm{y}^{\tau}(s), \bm{\alpha}(s))_{\mathbb{H}}\,ds 
- \tau \int_0^t ( \mathfrak{H}(s), \bm{y}^*(s))_{\mathbb{H}} (\widetilde{\bm{y}}(s), \bm{\alpha}(s))_{\mathbb{H}}\,ds \notag \\
&\quad - \int_0^t ( \mathfrak{H}^*(s), \bm{y}^*(s))_{\mathbb{H}} \|\bm{\alpha}(s)\|_{\mathbb{H}}^2\,ds 
 - \int_0^t ( \mathfrak{H}^*(s), \bm{\alpha}(s))_{\mathbb{H}} ( \bm{y}^*(s), \bm{\alpha}(s))_{\mathbb{H}}\,ds\notag \\
&\quad - \int_0^t ( \mathfrak{H}^*(s), \widetilde{\bm{y}}(s))_{\mathbb{H}} (\widetilde{\bm{y}}(s), \bm{\alpha}(s))_{\mathbb{H}}\,ds, 
\end{align}
for all \(t\in[0,T].\)
We now estimate the terms on the right-hand side of the above equality. In the following estimates, we shall repeatedly use the following estimates (see Lemma \ref{incest}).
\begin{align*}
    \sup_{0\le t\le T}\|\widetilde{\bm{y}}(t)\|_{\mathbb{H}}^2
+ \sup_{0\le t\le T}\|\nabla \widetilde{\bm{y}}(t)\|_{\mathbb{H}}^2
+ \int_0^T \|A\widetilde{\bm{y}}(t)\|_{\mathbb{H}}^2\,dt
\le C\,\mathcal{O}(\tau^2),\\
\sup_{0\le t\le T}\|\widetilde{\theta}(t)\|_{H}^2
+ \sup_{0\le t\le T}\|\nabla \widetilde{\theta}(t)\|_{\mathbb H}^2
+ \int_0^T \|\mathcal A\widetilde{\theta}(t)\|_{H}^2\,dt
\le C\,\mathcal{O}(\tau^2).
\end{align*}
We will use these facts throughout the following estimates without further mention. Using the Cauchy-Schwarz and Young inequalities, we find
\[
\begin{aligned}
&\nu \int_0^t \Big(\|\nabla \bm{y}^{\tau}(s)\|_{\mathbb{H}}^2 - \|\nabla \bm{y}^*(s)\|_{\mathbb{H}}^2\Big) (\widetilde{\bm{y}}(s), \bm{\alpha}(s))_{\mathbb{H}}\,ds \\
&\le \frac{\nu^2}{2} \int_0^t \big(\|\nabla \bm{y}^{\tau}(s)\|_{\mathbb{H}} + \|\nabla \bm{y}^*(s)\|_{\mathbb{H}}\big)^2\|\bm{\alpha}(s)\|_{\mathbb{H}}^2\,ds
+ C_\nu \int_0^t \|\nabla\widetilde{\bm{y}}(s)\|_{\mathbb{H}}^2\|\widetilde{\bm{y}}(s)\|_{\mathbb{H}}^2\,ds \\
&\le \frac{\nu^2}{2} \int_0^t \big(\|\nabla \bm{y}^{\tau}(s)\|_{\mathbb{H}} + \|\nabla \bm{y}^*(s)\|_{\mathbb{H}}\big)^2\|\bm{\alpha}(s)\|_{\mathbb{H}}^2\,ds
+ C_\nu\,\mathcal{O}(\tau^4).
\end{aligned}
\]
Using the identity
\[
\|\nabla \bm{y}^{\tau}\|_{\mathbb{H}}^2 - \|\nabla \bm{y}^*\|_{\mathbb{H}}^2 - 2( \nabla \bm{y}^*, \nabla \widetilde{\bm{y}})_{\mathbb{H}}
= \|\nabla \widetilde{\bm{y}}\|_{\mathbb{H}}^2,
\]
and by the Cauchy--Schwarz and Young inequalities, we get
\[
\begin{aligned}
&\nu \int_0^t \Big[
\Big(\|\nabla \bm{y}^{\tau}(s)\|_{\mathbb{H}}^2 - \|\nabla \bm{y}^*(s)\|_{\mathbb{H}}^2\Big)
- 2( \nabla \bm{y}^*(s), \nabla \widetilde{\bm{y}}(s))_{\mathbb{H}}
\Big] ( \bm{y}^*(s), \bm{\alpha}(s))_{\mathbb{H}}\,ds \\
&= \nu \int_0^t \|\nabla \widetilde{\bm{y}}(s)\|_{\mathbb{H}}^2 ( \bm{y}^*(s), \bm{\alpha}(s))_{\mathbb{H}}\,ds \le \frac{\nu^2}{2} \int_0^t \|\bm{\alpha}(s)\|_{\mathbb{H}}^2\,ds + C_\nu \int_0^t \|\nabla \widetilde{\bm{y}}(s)\|_{\mathbb{H}}^4\,ds \notag
\\& \le \frac{\nu^2}{2} \int_0^t \|\bm{\alpha}(s)\|_{\mathbb{H}}^2\,ds + C_\nu\,\mathcal{O}(\tau^4).
\end{aligned}
\]
By the Cauchy-Schwarz and Young inequalities, we write
\[
\begin{aligned}
2\nu \int_0^t ( \nabla \bm{y}^*(s), \nabla \bm{\alpha}(s))_{\mathbb{H}} ( \bm{y}^*(s), \bm{\alpha}(s))_{\mathbb{H}}\,ds
&\le 2\nu \int_0^t \|\nabla \bm{y}^*(s)\|_{\mathbb{H}} \|\nabla \bm{\alpha}(s)\|_{\mathbb{H}} \|\bm{\alpha}(s)\|_{\mathbb{H}}\,ds \\
&\le \frac{\nu}{8} \int_0^t \|\nabla \bm{\alpha}(s)\|_{\mathbb{H}}^2\,ds + C_\nu \int_0^t \|\nabla \bm{y}^*(s)\|_{\mathbb{H}}^2 \|\bm{\alpha}(s)\|_{\mathbb{H}}^2\,ds.
\end{aligned}
\]
Using the skew-symmetry property of the trilinear form and Ladyzhenskaya's inequality, we have
\[
\begin{aligned}
\int_0^t ( B(\bm{\alpha}(s), \bm{y}^*(s)), \bm{\alpha}(s))_{\mathbb{H}}\,ds 
&\le C \int_0^t \|\bm{\alpha}(s)\|_{\mathbb L^4}^2 \|\nabla \bm{y}^*(s)\|_{\mathbb{H}}\,ds\le C \int_0^t \|\bm{\alpha}(s)\|_{\mathbb{H}} \|\nabla \bm{\alpha}(s)\|_{\mathbb{H}} \|\nabla \bm{y}^*(s)\|_{\mathbb{H}}\,ds \\
&\le \frac{\nu}{8} \int_0^t \|\nabla \bm{\alpha}(s)\|_{\mathbb{H}}^2\,ds + C_\nu \int_0^t \|\nabla \bm{y}^*(s)\|_{\mathbb{H}}^2 \|\bm{\alpha}(s)\|_{\mathbb{H}}^2\,ds.
\end{aligned}
\]
\allowdisplaybreaks
A similar calculation yields
\[
\begin{aligned}
&- \int_0^t \Big( B(\bm{y}^{\tau}(s), \bm{y}^{\tau}(s)) - B(\bm{y}^*(s), \bm{y}^*(s)) - B'(\bm{y}^*(s)) \widetilde{\bm{y}}(s), \bm{\alpha}(s)\Big)_{\mathbb{H}}\,ds = -\int_0^t ( B(\widetilde{\bm{y}}(s), \widetilde{\bm{y}}(s)), \bm{\alpha}(s))_{\mathbb{H}}\,ds \\
&\le C \int_0^t \|\widetilde{\bm{y}}(s)\|_{\mathbb L^4}^2 \|\nabla {\bm{\alpha}}(s)\|_{\mathbb{H}}\,ds\le \frac{\nu}{8} \int_0^t \|\nabla \bm{\alpha}(s)\|_{\mathbb{H}}^2\,ds + C_\nu \int_0^t \|\widetilde{\bm{y}}(s)\|_{\mathbb{H}}^2 \|\nabla \widetilde{\bm{y}}(s)\|_{\mathbb{H}}^2\,ds \\
& \le \frac{\nu}{8} \int_0^t \|\nabla \bm{\alpha}(s)\|_{\mathbb{H}}^2\,ds + C_\nu\,\mathcal{O}(\tau^4).
\end{aligned}
\]
For the temperature coupling term, using the continuity of the projection operator \({\textbf{Proj}}_{\mathbb H}\), the Cauchy-Schwarz inequality, and Young's inequality, we get
\[
\int_0^t \beta ( {\textbf{Proj}_{\mathbb H}}(z(s) e_2), \bm{\alpha}(s))_{\mathbb{H}}\,ds
\le C_\beta \int_0^t \|z(s)\|_{H}^2\,ds + \frac{1}{2}\int_0^t \|\bm{\alpha}(s)\|_{\mathbb{H}}^2\,ds.
\]
\allowdisplaybreaks
Using \(\bm{y}^{\tau}, \bm{y}^* \in \mathcal{M}_1\) and \(\theta^{\tau}, \theta^* \in \mathcal{M}_2\), together with the Cauchy-Schwarz and Young inequalities, we estimate
\begin{align*}
&-\beta\int_0^t\Big[
(\widetilde\theta(s)e_2,\bm{y}^\tau(s))_{\mathbb H}
(\widetilde{\bm{y}}(s),\bm{\alpha}(s))_{\mathbb H}
+(\theta^*(s)e_2,\widetilde{\bm{y}}(s))_{\mathbb H}
(\widetilde{\bm{y}}(s),\bm{\alpha}(s))_{\mathbb H}\\
&\qquad
+(\widetilde\theta(s)e_2,\widetilde{\bm{y}}(s))_{\mathbb H}
(\bm{y}^*(s),\bm{\alpha}(s))_{\mathbb H}
+(z(s)e_2,\bm{y}^*(s))_{\mathbb H}
(\bm{y}^*(s),\bm{\alpha}(s))_{\mathbb H}\\
&\qquad
+(\theta^*(s)e_2,\bm{\alpha}(s))_{\mathbb H}
(\bm{y}^*(s),\bm{\alpha}(s))_{\mathbb H}
+(\theta^*(s)e_2,\bm{y}^*(s))_{\mathbb H}
\|\bm{\alpha}(s)\|_{\mathbb H}^2
\Big]\,ds\\
&\quad\le
C_\beta\int_0^t\|z(s)\|_H^2\,ds
+C_\beta\int_0^t\|\bm{\alpha}(s)\|_{\mathbb H}^2\,ds
+C\mathcal{O}(\tau^4).
\end{align*}
For the control involving terms, we use the Cauchy--Schwarz and Young's inequalities to obtain
\[
\begin{aligned}
-& \tau \int_0^t ( \mathfrak{H}(s), \widetilde{\bm{y}}(s))_{\mathbb{H}} ( \bm{y}^{\tau}(s), \bm{\alpha}(s))_{\mathbb{H}}\,ds 
- \tau \int_0^t ( \mathfrak{H}(s), \bm{y}^*(s))_{\mathbb{H}} (\widetilde{\bm{y}}(s), \bm{\alpha}(s))_{\mathbb{H}}\,ds \\
&\le C_\nu \tau^2 \int_0^t \|\mathfrak{H}(s)\|_{\mathbb{H}}^2 \|\widetilde{\bm{y}}(s)\|_{\mathbb{H}}^2\,ds + \frac{\nu}{4}\int_0^t \| \bm{\alpha}(s)\|_{\mathbb{H}}^2\,ds\le C_\nu \int_0^t \|\bm{\alpha}(s)\|_{\mathbb{H}}^2\,ds + C_\nu\,\mathcal{O}(\tau^4).
\end{aligned}
\]
Using the Cauchy-Schwarz inequality and the fact that \(\bm{y}^* \in \mathcal{M}_1\), we obtain
\[
\begin{aligned}
&- \int_0^t ( \mathfrak{H}^*(s), \bm{y}^*(s))_{\mathbb{H}} \|\bm{\alpha}(s)\|_{\mathbb{H}}^2\,ds 
- \int_0^t ( \mathfrak{H}^*(s), \bm{\alpha}(s))_{\mathbb{H}} ( \bm{y}^*(s), \bm{\alpha}(s))_{\mathbb{H}}\,ds \le C \int_0^t \|\mathfrak{H}^*(s)\|_{\mathbb{H}} \|\bm{\alpha}(s)\|_{\mathbb{H}}^2\,ds.
\end{aligned}
\]
Finally, applying the Cauchy-Schwarz and Young inequalities, we infer
\[
- \int_0^t ( \mathfrak{H}^*(s), \widetilde{\bm{y}}(s))_{\mathbb{H}} (\widetilde{\bm{y}}(s), \bm{\alpha}(s))_{\mathbb{H}}\,ds
\le C \int_0^t \|\bm{\alpha}(s)\|_{\mathbb{H}}^2\,ds + C_\nu\,\mathcal{O}(\tau^4).
\]
Collecting all the above estimates and absorbing terms like \(\|\nabla \bm{\alpha}(s)\|_{\mathbb{H}}^2\) into the left-hand side, we obtain
\begin{align}\label{est-y}
\frac{1}{2}\|\bm{\alpha}(t)\|_{\mathbb{H}}^2 + \frac{5\nu}{8}\int_0^t \|\nabla \bm{\alpha}(s)\|_{\mathbb{H}}^2\,ds
&\le \frac{1}{2}\|\bm{\alpha}(0)\|_{\mathbb{H}}^2
+ C_{\nu,\beta}\int_0^t \Big(1 + \|\nabla \bm{y}^{\tau}(s)\|_{\mathbb{H}}^2 + \|\nabla \bm{y}^*(s)\|_{\mathbb{H}}^2\nonumber\\& \quad + \|\mathfrak{H}^*(s)\|_{\mathbb{H}}\Big)\|\bm{\alpha}(s)\|_{\mathbb{H}}^2\,ds+ C_{\nu,\beta}\int_0^t \|z(s)\|_{H}^2\,ds + C \mathcal{O}(\tau^4),
\end{align}
\allowdisplaybreaks
for a.e. \(t \in [0,T]\).

Taking the $H$-inner product of the second equation in
\eqref{eq:y-z-system} with $z(\cdot)$, using
$(\bm{y}^*\cdot\nabla z,z)_H=0$, and integrating over
$(0,t)$, we obtain
\[
\begin{aligned}
&\frac{1}{2}\|z(t)\|_{H}^2 + \kappa \int_0^t \|\nabla z(s)\|_{\mathbb H}^2\,ds\\
&\quad= \frac{1}{2}\|z(0)\|_{H}^2
+ \kappa \int_0^t \|\nabla \theta^*(s)\|_{\mathbb H}^2 \|z(s)\|_{H}^2\,ds + \kappa \int_0^t \Big(\|\nabla \theta^{\tau}(s)\|_{\mathbb H}^2 - \|\nabla \theta^*(s)\|_{\mathbb H}^2\Big) (\widetilde{\theta}(s), z(s))_{H}\,ds \\
&\quad \quad+ \kappa \int_0^t \Big[
\Big(\|\nabla \theta^{\tau}(s)\|_{\mathbb H}^2 - \|\nabla \theta^*(s)\|_{\mathbb H}^2\Big)
- 2( \nabla \theta^*(s), \nabla \widetilde{\theta}(s))_{H}
\Big] ( \theta^*(s), z(s))_{H}\,ds \\
&\quad\quad  + 2\kappa \int_0^t ( \nabla \theta^*(s), \nabla z(s))_{H} ( \theta^*(s), z(s))_{H}\,ds- \int_0^t (\widetilde{\bm{y}}(s) \cdot \nabla \widetilde{\theta}(s), z(s))_{H}\,ds 
\nonumber\\&\qquad- \int_0^t (\bm{\alpha}(s) \cdot \nabla \theta^*(s), z(s))_{H}\,ds,
\end{aligned}
\]
for a.e. \(t\in[0,T].\)
We now estimate each term on the right-hand side. Applying the Cauchy--Schwarz and Young's inequalities,
and noting that \(\theta^*(s) \in \mathcal{M}_2\), we have
\[
\begin{aligned}
&\kappa \int_0^t \Big(\|\nabla \theta^{\tau}(s)\|_{\mathbb H}^2 - \|\nabla \theta^*(s)\|_{\mathbb H}^2\Big) (\widetilde{\theta}(s), z(s))_{H}\,ds \\
&\le C_\kappa \int_0^t \big(\|\nabla \theta^{\tau}(s)\|_{\mathbb H}^2 + \|\nabla \theta^*(s)\|_{\mathbb H}^2\big)\|z(s)\|_{H}^2\,ds
+ C_\kappa \int_0^t \|\nabla\widetilde{\theta}(s)\|_{\mathbb H}^2\|\widetilde{\theta}(s)\|_{H}^2\,ds \\
&\le C_\kappa \int_0^t \big(\|\nabla \theta^{\tau}(s)\|_{\mathbb H}^2 + \|\nabla \theta^*(s)\|_{\mathbb H}^2\big)\|z(s)\|_{H}^2\,ds
+ C_\kappa\,\mathcal{O}(\tau^4).
\end{aligned}
\]
Using the identity
$
\|\nabla \theta^{\tau}(s)\|_{\mathbb H}^2 - \|\nabla \theta^*(s)\|_{\mathbb H}^2 - 2( \nabla \theta^*(s), \nabla \widetilde{\theta}(s))_{\mathbb H}
= \|\nabla \widetilde{\theta}(s)\|_{\mathbb H}^2,$
and applying the Cauchy-Schwarz and Young inequalities, we estimate
\begin{align*}
&\kappa \int_0^t \Big[
\Big(\|\nabla \theta^{\tau}(s)\|_{\mathbb H}^2 - \|\nabla \theta^*(s)\|_{\mathbb H}^2\Big)
- 2( \nabla \theta^*(s), \nabla \widetilde{\theta}(s))_{H}
\Big] ( \theta^*(s), z(s))_{H}\,ds \\
&= \kappa \int_0^t \|\nabla \widetilde{\theta}(s)\|_{\mathbb H}^2 ( \theta^*(s), z(s))_{H}\,ds\le C_\kappa \int_0^t \|z(s)\|_{H}^2\,ds + C_\kappa \int_0^t \|\nabla \widetilde{\theta}(s)\|_{\mathbb H}^4\,ds \nonumber\\&\le C_\kappa \int_0^t \|z(s)\|_{H}^2\,ds + C_\kappa\,\mathcal{O}(\tau^4).
\end{align*}
By the Cauchy-Schwarz and Young inequalities, we get
\begin{align}
2\kappa \int_0^t ( \nabla \theta^*(s), \nabla z(s))_{H} ( \theta^*(s), z(s))_{H}\,ds
&\le \frac{\kappa}{6} \int_0^t \|\nabla z(s)\|_{\mathbb H}^2\,ds + C_\kappa \int_0^t \|\nabla \theta^*(s)\|_{\mathbb H}^2 \|z(s)\|_{H}^2\,ds.
\end{align}
Invoking H\"older's, Ladyzhenskaya's, and Young's inequalities, we write
\begin{align}
- \int_0^t (\widetilde{\bm{y}}(s) \cdot \nabla \widetilde{\theta}(s), z(s))_{H}\,ds
&\le \int_0^t \|\widetilde{\bm{y}}(s)\|_{\mathbb L^4} \|\nabla \widetilde{\theta}(s)\|_{\mathbb H} \|z(s)\|_{L^4}\,ds\nonumber \\
&\le C \int_0^t \|\widetilde{\bm{y}}(s)\|_{\mathbb{H}}^{1/2} \|\nabla \widetilde{\bm{y}}(s)\|_{\mathbb{H}}^{1/2} \|\nabla \widetilde{\theta}(s)\|_{\mathbb H} \|z(s)\|_{H}^{1/2} \|\nabla z(s)\|_{\mathbb H}^{1/2}\,ds\nonumber \\
&\le \frac{\kappa}{6} \int_0^t \|\nabla z(s)\|_{\mathbb H}^2\,ds + C_{\kappa} \int_0^t \|\nabla \widetilde{\bm{y}}(s)\|_{\mathbb{H}} \|\widetilde{\bm{y}}(s)\|_{\mathbb{H}}\|\nabla \widetilde{\theta}(s)\|_{\mathbb H}^2\,ds \nonumber\\&\quad+ C_{\kappa} \int_0^t \|z(s)\|_{H}^2\,ds\nonumber\\
&\le C_{\kappa} \int_0^t \|z(s)\|_{H}^2\,ds + \frac{\kappa}{6} \int_0^t \|\nabla z(s)\|_{\mathbb H}^2\,ds + C_{\kappa}\,\mathcal{O}(\tau^4),
\end{align}
and
\begin{align*}
\Big|\int_0^t
(\bm{\alpha}(s)\cdot\nabla\theta^*(s),z(s))_{H}\,ds\Big|
&=
\Big|\int_0^t
(\bm{\alpha}(s)\cdot\nabla z(s),\theta^*(s))_{H}\,ds\Big|
\leq
\int_0^t
\|\bm{\alpha}(s)\|_{\mathbb{L}^{4}}
\|\nabla z(s)\|_{\mathbb H}
\|\theta^*(s)\|_{L^{4}}\,ds
\\
&\leq
\frac{\kappa}{6}
\int_0^t\|\nabla z(s)\|_{\mathbb H}^{2}\,ds
+
C_{\kappa}
\int_0^t
\|\bm{\alpha}(s)\|_{\mathbb{H}}
\|\nabla\bm{\alpha}(s)\|_{\mathbb{H}}
\|\theta^*(s)\|_{L^{4}}^{2}\,ds
\nonumber\\&\leq
\frac{\kappa}{6}
\int_0^t\|\nabla z(s)\|_{\mathbb H}^{2}\,ds
+
\frac{\nu}{8}
\int_0^t\|\nabla\bm{\alpha}(s)\|_{\mathbb{H}}^{2}\,ds
+C_{\nu,\kappa}
\int_0^t
\|\theta^*(s)\|_{L^{4}}^{4}
\|\bm{\alpha}(s)\|_{\mathbb{H}}^{2}\,ds.
\end{align*}
Collecting the above estimates, we obtain, for every $t\in[0,T]$,
\begin{equation}\label{est-z-integrated}
\begin{aligned}
\frac{1}{2}\|z(t)\|_{H}^2 + \frac{\kappa}{2}\int_0^t \|\nabla z(s)\|_{\mathbb H}^2\,ds
&\le \frac{1}{2}\|z(0)\|_{H}^2 + C_{\nu,\kappa}\int_0^t \Big(1 + \|\nabla \theta^{\tau}(s)\|_{\mathbb H}^2 + \|\nabla \theta^*(s)\|_{\mathbb H}^2\Big)\|z(s)\|_{H}^2\,ds \\
&\quad + \frac{\nu}{8}\int_0^t \|\nabla \bm{\alpha}(s)\|_{\mathbb{H}}^2\,ds +C_{\nu,\kappa}\int_0^t
\|\theta^*(s)\|_{L^4}^4\|\bm{\alpha}(s)\|_{\mathbb H}^2\,ds+ C_{\nu,\kappa}\,\mathcal{O}(\tau^4).
\end{aligned}
\end{equation}
Adding \eqref{est-y} and \eqref{est-z-integrated}, we obtain
\begin{equation}\label{eq:gronwall-ready-integrated}
\begin{aligned}
&\frac12\Big(\|\bm{\alpha}(t)\|_{\mathbb H}^2+\|z(t)\|_H^2\Big)
+\frac{\nu}{2}\int_0^t\|\nabla\bm{\alpha}(s)\|_{\mathbb H}^2\,ds
+\frac{\kappa}{2}\int_0^t\|\nabla z(s)\|_{\mathbb H}^2\,ds\\
&\quad\le
\frac12\Big(\|\bm{\alpha}(0)\|_{\mathbb H}^2+\|z(0)\|_H^2\Big)
+\int_0^t a(s)
\Big(\|\bm{\alpha}(s)\|_{\mathbb H}^2+\|z(s)\|_H^2\Big)\,ds
+C\mathcal{O}(\tau^4),
\end{aligned}
\end{equation}
where
\begin{align*}
a(s):=C_{\nu,\kappa,\beta}\Big(
&1+\|\nabla\bm{y}^\tau(s)\|_{\mathbb H}^2
+\|\nabla\bm{y}^*(s)\|_{\mathbb H}^2
+\|\nabla\theta^\tau(s)\|_{\mathbb H}^2\\
&+\|\nabla\theta^*(s)\|_{\mathbb H}^2
+\|\theta^*(s)\|_{L^4}^4+\|\mathfrak{H}^*(s)\|_{\mathbb H}
\Big).
\end{align*}
The uniform estimates (See \eqref{neww}) imply that
$\|a\|_{L^1(0,T)}<\infty.$ 
From Gr\"onwall's inequality, it follows that
\begin{equation}\label{eq:gronwall-final}
\sup_{0\leq s\leq t}\bigl(\|\bm{\alpha}(s)\|_{\mathbb{H}}^2+\|z(s)\|_{H}^2\bigr)
+\int_0^t\bigl(\|\nabla\bm{\alpha}(s)\|_{\mathbb{H}}^2+\|\nabla z(s)\|_{\mathbb H}^2\bigr)\,ds
\leq C\bigl(\|\bm{\alpha}(0)\|_{\mathbb{H}}^2+\|z(0)\|_{H}^2\bigr)
+C\mathcal{O}(\tau^4)
\end{equation}
for all \(t \in [0,T]\).
Since \(\|\bm{\alpha}(0)\|_{\mathbb{H}}^2+\|z(0)\|_{H}^2=0\), we obtain
\begin{equation}\label{eq:gronwall-est}
\sup_{0\leq s\leq t}\bigl(\|\bm{\alpha}(s)\|_{\mathbb{H}}^2+\|z(s)\|_{H}^2\bigr)
\leq C_T\mathcal{O}(\tau^4),
\end{equation}
for all \(t\in[0,T]\), where \(C_T\) depends on \(T\) but is independent of \(\tau\). 
Consequently, dividing by \(\tau^2\) and letting \(\tau\downarrow0\), we conclude that
\begin{equation}\label{eq:limit-pointwise}
\frac{\|\bm{\alpha}(t)\|_{\mathbb{H}}}{\tau}\longrightarrow0
 \ \text{ and }\ 
\frac{\|z(t)\|_{H}}{\tau}\longrightarrow0
\  \text{ as }\ \tau\downarrow0,
\end{equation}
for each \(t\in[0,T]\). Moreover, using \eqref{eq:gronwall-final} and \eqref{eq:gronwall-est}, we obtain
\begin{equation}\label{eq:limit-L2}
\frac{\|\bm{\alpha}\|_{L^2(0,T;\mathbb{V})}}{\tau}\longrightarrow0
 \ \text{ and }\ 
\frac{\|z\|_{L^2(0,T;V)}}{\tau}\longrightarrow0
 \ \text{ as }\ \tau\downarrow0.
\end{equation}
This completes the proof of \eqref{inc1} and \eqref{inc2}.
\qed

Now we will show that \((\bm{\eta}, \xi)\) lies in the tangent cone of admissible perturbations, that is, we establish that \(\bm{\eta}(t) \in T_{\bm{y}_{\mathfrak{H}^*}(t)} \mathcal{M}_1\) and \(\xi(t) \in T_{\theta_{\mathfrak{H}^*}(t)} \mathcal{M}_2\), i.e., 
\[
( \bm{\eta}(t), \bm{y}_{\mathfrak{H}^*}(t) )_{\mathbb H} = 0 
\quad \text{and} \quad
( \xi(t), \theta_{\mathfrak{H}^*}(t) )_{H} = 0
\quad \text{for all } t \in [0,T].
\]
We start by establishing the velocity part. Because \(\bm{y}_{\mathfrak{H}^*}\in C([0,T];\mathbb{V})\) and \(\bm{\eta}\in C([0,T];\mathbb{H})\), we use them as test functions in the weak formulations associated with \(\bm{\eta}\) and \(\bm{y}_{\mathfrak{H}^*}\), respectively. Now, we define
\[
\phi(t) := ( \bm{\eta}(t), \bm{y}_{\mathfrak{H}^*}(t) )_{\mathbb H}.
\]
Then, for a.e. \(t \in [0,T]\),
\[
\begin{aligned}
\frac{d}{dt} \phi(t)
&= \left( \frac{d\bm{\eta}(t)}{dt}, \bm{y}_{\mathfrak{H}^*}(t) \right)_{\mathbb H}
+ \left( \bm{\eta}(t), \frac{d\bm{y}_{\mathfrak{H}^*}(t)}{dt} \right)_{\mathbb H}.
\end{aligned}
\]
\allowdisplaybreaks
Using the linearized system \eqref{linearized-system-a} for \(\bm{\eta}\) and the state equation \eqref{eq:constrained_Boussinesq} for \(\bm{y}_{\mathfrak{H}^*}\), we obtain
\[
\begin{aligned}
\frac{d}{dt}\phi(t)
&=
\Big(
-\nu A\bm{\eta}(t)
+\nu\|\nabla \bm{y}_{\mathfrak{H}^*}(t)\|_{\mathbb H}^{2}\bm{\eta}(t)
+2\nu(\nabla \bm{y}_{\mathfrak{H}^*}(t),\nabla \bm{\eta}(t))_{\mathbb H}\bm{y}_{\mathfrak{H}^*}(t)
-B(\bm{y}_{\mathfrak{H}^*}(t),\bm{\eta}(t))-B(\bm{\eta}(t),\bm{y}_{\mathfrak{H}^*}(t))  \\
&\qquad
+\beta{\textbf{Proj}_{\mathbb H}}(\xi(t) e_2)
-\beta\Big[
(\xi(t) e_2,\bm{y}_{\mathfrak{H}^*}(t))_{\mathbb H}\bm{y}_{\mathfrak{H}^*}(t)
+( \theta_{\mathfrak{H}^*}(t)e_2,\bm{\eta}(t))_{\mathbb H}\bm{y}_{\mathfrak{H}^*}(t)
+( \theta_{\mathfrak{H}^*}(t)e_2,\bm{y}_{\mathfrak{H}^*}(t))_{\mathbb H}\bm{\eta}(t)
\Big]  \\
&\qquad
+\mathfrak{H}(t)-( \mathfrak{H}(t),\bm{y}_{\mathfrak{H}^*}(t))_{\mathbb H}\bm{y}_{\mathfrak{H}^*}(t)
-( \mathfrak{H}^*(t),\bm{y}_{\mathfrak{H}^*}(t))_{\mathbb H}\bm{\eta}(t)
-( \mathfrak{H}^*(t),\bm{\eta}(t))_{\mathbb H}\bm{y}_{\mathfrak{H}^*}(t),
\,\bm{y}_{\mathfrak{H}^*}(t)
\Big)_{\mathbb H}  \\
&\quad
+\Big(\bm{\eta}(t),
-\nu A\bm{y}_{\mathfrak{H}^*}(t)
+\nu\|\nabla \bm{y}_{\mathfrak{H}^*}(t)\|_{\mathbb H}^{2}\bm{y}_{\mathfrak{H}^*}(t)
-B(\bm{y}_{\mathfrak{H}^*}(t),\bm{y}_{\mathfrak{H}^*}(t))
+\beta{\textbf{Proj}_{\mathbb H}}(\theta_{\mathfrak{H}^*}(t)e_2) \\
&\qquad
-\beta( \theta_{\mathfrak{H}^*}(t)e_2,\bm{y}_{\mathfrak{H}^*}(t))_{\mathbb H}\bm{y}_{\mathfrak{H}^*}(t)
+\mathfrak{H}^*(t)
-( \mathfrak{H}^*(t),\bm{y}_{\mathfrak{H}^*}(t))_{\mathbb H}\bm{y}_{\mathfrak{H}^*}(t)
\Big)_{\mathbb H}.
\end{aligned}
\]
Consequently,
\[
\frac{d}{dt}\phi(t)
=
2\Bigl(
\nu\|\nabla \bm{y}_{\mathfrak{H}^*}(t)\|_{\mathbb H}^{2}
-\beta( \theta_{\mathfrak{H}^*}(t)e_2,\bm{y}_{\mathfrak{H}^*}(t))_{\mathbb H}
-( \mathfrak{H}^*(t),\bm{y}_{\mathfrak{H}^*}(t))_{\mathbb H}
\Bigr)\phi(t)
\]
for a.e. \(t\in[0,T]\).
By the variation-of-constants formula, we get
\[
\phi(t)
=
\phi(0)
\exp\!\left(
2\int_0^t
\Bigl(
\nu\|\nabla \bm{y}_{\mathfrak{H}^*}(s)\|_{\mathbb H}^{2}
-\beta( \theta_{\mathfrak{H}^*}(s)e_2,\bm{y}_{\mathfrak{H}^*}(s))_{\mathbb H}
-( \mathfrak{H}^*(s),\bm{y}_{\mathfrak{H}^*}(s))_{\mathbb H}
\Bigr)\,ds
\right).
\]
Since \(\bm{\eta}(0)=\boldsymbol{0}\), we have \(\phi(0)=0\). Moreover, we have
\[
\int_0^T \left|
\nu\|\nabla \bm{y}_{\mathfrak{H}^*}(s)\|_{\mathbb H}^{2}
-\beta( \theta_{\mathfrak{H}^*}(s)e_2,\bm{y}_{\mathfrak{H}^*}(s))_{\mathbb H}
-( \mathfrak{H}^*(s),\bm{y}_{\mathfrak{H}^*}(s))_{\mathbb H}
\right| ds < \infty.
\]
Therefore, we infer
\[
\bm{\eta}(t)\in T_{\bm{y}_{\mathfrak{H}^*}(t)}\mathcal M_1,
\qquad t\in[0,T].
\]
We now prove the temperature part. Define 
\[
\psi(t) := ( \xi(t), \theta_{\mathfrak{H}^*}(t) )_{H}.
\]
Then, for a.e. \(t \in [0,T]\),
\[
\frac{d}{dt}\psi(t)
= \left(\frac{d\xi(t)}{dt}, \theta_{\mathfrak{H}^*}(t) \right)_{H}
+ \left( \xi(t), \frac{d\theta_{\mathfrak{H}^*}(t)}{dt} \right)_{H}.
\]
Using the linearized equation for \(\xi\) from \eqref{linearized-system-b} and the state equation for \(\theta_{\mathfrak{H}^*}\) from \eqref{eq:constrained_Boussinesq}, we obtain
\[
\begin{aligned}
\frac{d}{dt}\psi(t)
&=
\Big(
-\kappa A\xi(t)
+\kappa\|\nabla \theta_{\mathfrak{H}^*}(t)\|_{\mathbb H}^{2}\xi(t)
+2\kappa( \nabla \theta_{\mathfrak{H}^*}(t),\nabla \xi(t))_{\mathbb H}\theta_{\mathfrak{H}^*}(t)
-\bm{y}_{\mathfrak{H}^*}(t)\cdot\nabla \xi(t) \\
&\quad-\bm{\eta}(t)\cdot\nabla \theta_{\mathfrak{H}^*}(t),
\theta_{\mathfrak{H}^*}(t)
\Big)_{H}+
\Big(
\xi(t),
-\kappa A\theta_{\mathfrak{H}^*}(t)
+\kappa\|\nabla \theta_{\mathfrak{H}^*}(t)\|_{\mathbb H}^{2}\theta_{\mathfrak{H}^*}(t)
-\bm{y}_{\mathfrak{H}^*}(t)\cdot\nabla \theta_{\mathfrak{H}^*}(t)
\Big)_{H}.
\end{aligned}
\]
Expanding the inner products and using the fact that \(\|\theta_{\mathfrak{H}^*}\|_{H}=1\), we get
\[
\begin{aligned}
\frac{d}{dt}\psi(t)
&=
2\kappa\|\nabla \theta_{\mathfrak{H}^*}(t)\|_{\mathbb H}^{2}\psi(t)
-( \bm{y}_{\mathfrak{H}^*}(t)\cdot\nabla \xi(t), \theta_{\mathfrak{H}^*}(t))_{H}
-( \bm{\eta}(t)\cdot\nabla \theta_{\mathfrak{H}^*}(t), \theta_{\mathfrak{H}^*}(t))_{H}
 \\
&\quad-( \xi(t), \bm{y}_{\mathfrak{H}^*}(t)\cdot\nabla \theta_{\mathfrak{H}^*}(t))_{H},
\end{aligned}
\]
for a.e. \(t \in [0,T]\).
Therefore, we have
\[
\frac{d}{dt}\psi(t)
=
2\kappa\|\nabla \theta_{\mathfrak{H}^*}(t)\|_{\mathbb H}^{2}\psi(t),
\]
for a.e. \(t \in [0,T]\). By the variation-of-constants formula, we get
\[
\psi(t)=\psi(0)\exp\left(2\kappa\int_0^t \|\nabla \theta_{\mathfrak{H}^*}(s)\|_{\mathbb H}^{2}\,ds\right).
\]
Since \(\xi(0)=0\), we have \(\psi(0)=0\). Consequently, we deduce
\[
\psi(t) = 0 \ \ \text{for all } t \in [0,T].
\]
Thus, we infer 
\[
(\xi(t),\theta_{\mathfrak{H}^*}(t))_{H}=0,
\ \ t\in[0,T].
\]
Combining this with the velocity part, we conclude that
\[
(\bm{\eta}(t),\xi(t))
\in
T_{(\bm{y}_{\mathfrak{H}^*}(t),\theta_{\mathfrak{H}^*}(t))}\mathcal M,
\ \  t\in[0,T],
\]
which completes the proof.
\end{proof}

\subsection{Proof of Theorem \ref{main}} Finally, we provide a proof of  Theorem \ref{main}. 
\begin{proof}[Proof of Theorem \ref{main}]
Let $(\bm{y}^*,\theta^*,\mathfrak{H}^*)\in \mathcal{A}_{ad}$ be an optimal triplet corresponding to the control problem \eqref{opt}. Define the reduced cost functional
\[
G(\mathfrak{H}) := J(\bm{y}_{\mathfrak{H}},\theta_{\mathfrak{H}},\mathfrak{H}),
\]
where $(\bm{y}_{\mathfrak{H}},\theta_\mathfrak{H})$ is the unique strong solution to the controlled system \eqref{eq:constrained_Boussinesq} corresponding to the control $\mathfrak{H} \in \mathcal{U}_{ad}$.

Let $\tau > 0$ be sufficiently small such that $\mathfrak{H}^* + \tau\mathfrak{H} \in \mathcal{U}_{ad}$ and
$(\bm{y}_{\mathfrak{H}^*+\tau\mathfrak{H}},\,\theta_{\mathfrak{H}^*+\tau\mathfrak{H}},\,
\mathfrak{H}^*+\tau\mathfrak{H}) \in \mathcal{A}_{ad}$.
With these conditions, we have the following relation (cf.\ \cite{manil}):
\begin{align}
&G(\mathfrak{H}^*+\tau\mathfrak{H})-G(\mathfrak{H}^*)
\notag\\&=J(\bm{y}_{\mathfrak{H}^*+\tau\mathfrak{H}},
\theta_{\mathfrak{H}^*+\tau\mathfrak{H}},\mathfrak{H}^*+\tau\mathfrak{H})
-J(\bm{y}_{\mathfrak{H}^*},\theta_{\mathfrak{H}^*},\mathfrak{H}^*)
\notag\\
&=\frac12\int_0^T
\|\bm{y}_{\mathfrak{H}^*+\tau\mathfrak{H}}(t)-\bm{y}_{\mathfrak{H}^*}(t)\|_{\mathbb H}^2\,dt
+\int_0^T
(\bm{y}_{\mathfrak{H}^*+\tau\mathfrak{H}}(t)-\bm{y}_{\mathfrak{H}^*}(t),
\bm{y}_{\mathfrak{H}^*}(t)-\bm{y}_d(t))_{\mathbb H}\,dt
\notag\\
&\quad+\frac12\int_0^T
\|\nabla(\bm{y}_{\mathfrak{H}^*+\tau\mathfrak{H}}(t)-\bm{y}_{\mathfrak{H}^*}(t))\|_{\mathbb H}^2\,dt
+\int_0^T
(\bm{y}_{\mathfrak{H}^*+\tau\mathfrak{H}}(t)-\bm{y}_{\mathfrak{H}^*}(t),
A\bm{y}_{\mathfrak{H}^*}(t))_{\mathbb H}\,dt
\notag\\
&\quad+\frac12\int_0^T
\|\theta_{\mathfrak{H}^*+\tau\mathfrak{H}}(t)-\theta_{\mathfrak{H}^*}(t)\|_H^2\,dt
+\int_0^T
(\theta_{\mathfrak{H}^*+\tau\mathfrak{H}}(t)-\theta_{\mathfrak{H}^*}(t),
\theta_{\mathfrak{H}^*}(t)-\theta_d(t))_H\,dt
\notag\\
&\quad+\tau\int_0^T(\mathfrak{H}(t),\mathfrak{H}^*(t))_{\mathbb H}\,dt
+\frac{\tau^2}{2}\int_0^T\|\mathfrak{H}(t)\|_{\mathbb H}^2\,dt
\notag\\
&\quad+\frac12
\|\bm{y}_{\mathfrak{H}^*+\tau\mathfrak{H}}(T)-\bm{y}_{\mathfrak{H}^*}(T)\|_{\mathbb H}^2
+(\bm{y}_{\mathfrak{H}^*+\tau\mathfrak{H}}(T)-\bm{y}_{\mathfrak{H}^*}(T),
\bm{y}_{\mathfrak{H}^*}(T)-\bm{y}_f)_{\mathbb H}
\notag\\
&\quad+\frac12
\|\theta_{\mathfrak{H}^*+\tau\mathfrak{H}}(T)-\theta_{\mathfrak{H}^*}(T)\|_H^2
+(\theta_{\mathfrak{H}^*+\tau\mathfrak{H}}(T)-\theta_{\mathfrak{H}^*}(T),
\theta_{\mathfrak{H}^*}(T)-\theta_f)_H.
\label{eq:cost-expansion}
\end{align}
Using Lemma~\ref{incest} and taking the limit as $\tau \downarrow 0$, we infer
\begin{align}
\|\bm{y}_{\mathfrak{H}^*+\tau\mathfrak{H}} - \bm{y}_{\mathfrak{H}^*}\|_{{{L}^2(0,T;\mathbb{H})}}
+ \|\bm{y}_{\mathfrak{H}^*+\tau\mathfrak{H}} - \bm{y}_{\mathfrak{H}^*}\|_{L^2(0,T;\mathbb{V})}
+ \|\bm{y}_{\mathfrak{H}^*+\tau\mathfrak{H}}(T) - \bm{y}_{\mathfrak{H}^*}(T)\|_{\mathbb{H}}
&\to 0,
\label{eq:conv-u}\\
\|\theta_{\mathfrak{H}^*+\tau\mathfrak{H}} - \theta_{\mathfrak{H}^*}\|_{L^2(0,T;H)}
+ \|\theta_{\mathfrak{H}^*+\tau\mathfrak{H}} - \theta_{\mathfrak{H}^*}\|_{L^2(0,T;V)}
+ \|\theta_{\mathfrak{H}^*+\tau\mathfrak{H}}(T) - \theta_{\mathfrak{H}^*}(T)\|_{H}
&\to 0.
\label{eq:conv-theta}
\end{align}

Let $G'(\mathfrak{H}^*;\mathfrak{H})$ denote the directional derivative of $G$ at $\mathfrak{H}^*$ in the direction $\mathfrak{H}$, and let $(\bm{\eta},\xi)$ solve the linearized system \eqref{linearized-system-a}--\eqref{linearized-system-b}. Letting $\tau \downarrow 0$ and using Lemma~\ref{incest} and Proposition~\ref{prop4.4}, we obtain
\begin{align}
0 \leq G'(\mathfrak{H}^*;\mathfrak{H})
&= \lim_{\tau \downarrow 0}
   \frac{G(\mathfrak{H}^*+\tau\mathfrak{H}) - G(\mathfrak{H}^*)}{\tau} \notag\\
&= \int_0^T
   ( \bm{\eta}(t),\, \bm{y}^*(t) - \bm{y}_d(t) )_{\mathbb{H}}\,dt
 + \int_0^T
   ( \bm{\eta}(t),\, A\bm{y}^*(t) )_{\mathbb{H}}\,dt \notag\\
&\quad
 + \int_0^T
   ( \xi(t),\, \theta^*(t) - \theta_d(t) )_{H}\,dt
 + \int_0^T ( \mathfrak{H}^*(t),\, \mathfrak{H}(t) )_{\mathbb{H}}\,dt \notag\\
&\quad
 + ( \bm{\eta}(T),\, \bm{y}^*(T) - \bm{y}_f )_{\mathbb{H}}
 + ( \xi(T),\, \theta^*(T) - \theta_f )_{H}.
\label{dir-derivative-corrected}
\end{align}
Using these identities, the directional derivative and the linearized system \eqref{linearized-system-a}--\eqref{linearized-system-b}, we get
\begin{align}
G'(\mathfrak{H}^*;\mathfrak{H})
&= \int_0^T ( \mathfrak{H}(t), \bm{\vartheta}(t))_{\mathbb{H}} dt
- \int_0^T ( \mathfrak{H}(t),\bm{y}^*(t))_{\mathbb{H}}( \bm{y}^*(t),\bm{\vartheta}(t))_{\mathbb{H}} dt
+ \int_0^T ( \mathfrak{H}^*(t), \mathfrak{H}(t))_{\mathbb{H}} dt \notag\\
&= \int_0^T ( \mathfrak{H}(t), \bm{\vartheta}(t) - (\bm{\vartheta}(t),\bm{y}^*(t))_{\mathbb{H}}\bm{y}^*(t) + \mathfrak{H}^*(t))_{\mathbb{H}} dt,
\label{G-derivative-final}
\end{align}
for all \(\mathfrak{H}\in {{L}^2(0,T;\mathbb{H})}\).
The above relation implies that
\[
{\mathfrak{H}^*(t) = -\bm{\vartheta}(t) + (\bm{\vartheta}(t),\bm{y}^*(t))_{\mathbb{H}}\bm{y}^*(t),
\quad \text{for a.e. } t\in[0,T],}
\]
which completes the proof.
\end{proof}

\textbf{Acknowledgments.} The third author is grateful to the Indian Institute of Science Education and Research Thiruvananthapuram (IISER Thiruvananthapuram) for providing the necessary financial assistance. Support for M. T. Mohan's research received from the National Board of Higher Mathematics (NBHM), Department of Atomic Energy, Government of India (Project No. 02011/13/2025/NBHM(R.P)/R\&D II/1137). 

\printbibliography
\end{document}